\RequirePackage{etoolbox}
\documentclass[ss,preprint,twoside,linksfromyear,nameyear,natbib]{imsart}
\RequirePackage[colorlinks,citecolor=blue,urlcolor=blue]{hyperref}
\usepackage{amssymb,amsthm,amsmath}

\doi{10.1214/15-SS109}
\pubyear{2015}
\volume{9}
\issue{0}
\firstpage{32}
\lastpage{105}

\allowdisplaybreaks

\startlocaldefs
\newtheorem{Theorem}{Theorem}[section]
\newtheorem{Corollary}[Theorem]{Corollary}
\newtheorem{Proposition}[Theorem]{Proposition}
\newtheorem{Lemma}[Theorem]{Lemma}

\theoremstyle{definition}
\newtheorem{Definition}[Theorem]{Definition}
\newtheorem{Remark}[Theorem]{Remark}
\newtheorem{Example}[Theorem]{Example}

\numberwithin{equation}{section}

\def\Ex{\mathbb{E}}
\def\Pr{\mathbb{P}}
\def\Cov{\mathop{\mathrm{Cov}}\nolimits}
\def\Var{\mathop{\mathrm{Var}}\nolimits}
\def\bs{\boldsymbol}
\def\bmu{\bs{\mu}}
\def\bSigma{\bs{\Sigma}}
\def\hat{\widehat}
\newcommand{\M}{\mathbb{M}}
\newcommand{\R}{\mathbb{R}}
\newcommand{\V}{\mathbb{V}}
\newcommand{\W}{\mathbb{W}}
\newcommand{\LL}{\mathcal{L}}
\newcommand{\NN}{\mathcal{N}}
\newcommand{\PP}{\mathcal{P}}
\newcommand{\QQ}{\mathcal{Q}}
\newcommand{\VV}{\mathcal{V}}
\newcommand{\Rqq}{\R_{}^{q\times q}}
\newcommand{\Rqqns}{\R_{\rm ns}^{q\times q}}
\newcommand{\Rqqsym}{\R_{\rm sym}^{q\times q}}
\newcommand{\Rqqsympd}{\R_{{\rm sym},>0}^{q\times q}}
\newcommand{\Rqqsympsd}{\R_{{\rm sym},\ge0}^{q\times q}}
\newcommand{\diag}{\mathop{\mathrm{diag}}}
\newcommand{\tr}{\mathop{\mathrm{tr}}}

\endlocaldefs

\begin{document}

%
\begin{frontmatter}
\title{\boldmath$M$-functionals of multivariate scatter}
\runtitle{$M$-functionals of multivariate scatter}

\begin{aug}
%
\author{\fnms{Lutz} \snm{D\"umbgen}\corref{}\thanksref{t1}\ead
[label=e1]{duembgen@stat.unibe.ch}}
\address{University of Bern, Sidlerstr.\ 5, CH-3012 Bern, Switzerland\\
\printead{e1}}
\end{aug}
\begin{aug}
\author{\fnms{Markus} \snm{Pauly}\thanksref{t2}\ead
[label=e2]{markus.pauly@uni-ulm.de}}
\address{Ulm University, Helmholtzstr.\ 20, D-89081 Ulm, Germany\\
\printead{e2}}
\end{aug}
\medskip\textbf{\and}
\begin{aug}
\author{\fnms{Thomas} \snm{Schweizer}\thanksref{t1}\ead
[label=e3]{thomas-za.schweizer@ubs.com}}
\address{University of Bern, Sidlerstr.\ 5, CH-3012 Bern, Switzerland\\
\printead{e3}}
\end{aug}

\thankstext{t1}{Supported by Swiss National Science Foundation (SNF).}
\thankstext{t2}{Supported by a fellowship within the postdoc programme
of the German Academic Exchange Service (DAAD).}
\runauthor{L. D\"umbgen et al. }

\begin{abstract}
This survey provides a self-contained account of $M$-estimation of
multivariate scatter. In particular, we present new proofs for
existence of the underlying $M$-functionals and discuss their weak
continuity and differentiability. This is done in a rather general
framework with matrix-valued random variables. By doing so we reveal a
connection between Tyler's (\citeyear{Tyler_1987a}) $M$-functional of
scatter and the estimation of proportional covariance matrices.
Moreover, this general framework allows us to treat a new class of
scatter estimators, based on symmetrizations of arbitrary order.
Finally these results are applied to $M$-estimation of multivariate
location and scatter via multivariate $t$-distributions.
\end{abstract}



%
\received{\smonth{1} \syear{2014}}

\end{frontmatter}
\maketitle

\tableofcontents

\section{Introduction}

The study of $M$-estimation for certain parameters or functionals of
interest has a long history. Roughly speaking an $M$-estimator is the
maximizer of a random criterion function depending on the data and
corresponding to the estimation problem. Best known examples are
maximum-likelihood estimators as well as robust estimators of location,
e.g.\ the sample median, and scatter. In basic statistics courses it is
shown that especially maximum-likelihood estimators are asymptotically
normal and efficient under quite weak assumptions, see e.g.\ the
graduate textbooks by Serfling (\citeyear{Serfling_1980}), Lehmann and
Casella (\citeyear{Lehmann_Casella_1998}) and van der Vaart (\citeyear
{vdVaart_1998}). Specific $M$-estimators of one- and multidimensional
parameters can be shown to be asymptotically normal and quite efficient
under even weaker assumptions, see e.g.\ Huber (\citeyear{Huber_1964,
Huber_1973}), thus providing an interesting alternative to
classical unbiased estimators.

In the present survey we consider $M$-estimates and functionals of
multivariate location and scatter. Our purpose is to provide a concise
but self-contained presentation of the main ideas and results in this
context, the target audience being researchers and advanced graduate
students. The basic setting is as follows: Let $P$ be a probability
distribution on $\R^q$. Traditionally the center of $P$ is defined to
be the mean vector
\[
\bmu(P) \ := \ \int x \, P(dx),
\]
assuming that $\int\|x\| \, P(dx) < \infty$. Assuming also that $\int
\|x\|^2 \, P(dx) < \infty$, the covariance matrix of $P$ is defined as
\[
\bSigma(P) \ := \ \int(x - \bmu(P))(x - \bmu(P))^\top\, P(dx),
\]
where vectors are understood as column vectors and $(\cdot)^\top$
denotes transposition. Recall that for a random vector $X$ with
distribution $P$ and any fixed vector $v \in\R^q$,
\[
\Ex(v^\top X) \ = \ v^\top\bmu(P)
\quad\text{and}\quad
\Var(v^\top X) \ = \ v^\top\bSigma(P) v.
\]
Thus for a unit vector $v \in\R^q$, the spread of $P$ in direction
$v$ may be quantified by $\sqrt{v^\top\bSigma(P) v}$, the standard
deviation of $v^\top X$.

There are various good reasons to use different definitions of the
center $\bmu(P)$ and scatter matrix $\bSigma(P)$ of the distribution
$P$. For instance, suppose that $P$ has a unimodal density $f$ and is
elliptically symmetric with center $\mu\in\R^q$ and symmetric,
positive definite scatter matrix $\Sigma\in\R^{q\times q}$. That
means, $f$ may be written as
\[
f(x) \ = \ \tilde{f} \bigl( (x - \mu)^\top\Sigma^{-1} (x - \mu)
\bigr)
\]
for some decreasing function $\tilde{f} : [0,\infty) \to[0,\infty
)$. Then it would be natural to define the center of $P$ to be $\bmu
(P) := \mu$, and a scatter matrix $\bSigma(P)$ of $P$ should be equal
or at least proportional to $\Sigma$, even if $\int\|x\| \, P(dx)$ or
$\int\|x\|^2 \, P(dx)$ is infinite. A related issue is robustness: One
would like $\bmu(P)$ and $\bSigma(P)$ to change little if $P$ is
replaced with $(1 - \epsilon) P + \epsilon P'$ for some small number
$\epsilon> 0$ and an arbitrary distribution $P'$ on $\R^q$. Another
way to define robustness is weak continuity: It would be desirable that
$\bmu(P') \to\bmu(P)$ and $\bSigma(P') \to\bSigma(P)$ whenever
$P' \to P$ weakly.

Some people may feel overwhelmed by the diversity of scatter
functionals which are available. However, comparing two or more
different scatter matrices $\bSigma(P)$ allows one to find interesting
structures in the distribution $P$. For an explanation of this paradigm
and examples we refer to Nordhausen et al.\ (\citeyear
{Nordhausen_etal_2008}), Tyler et al.\ (\citeyear{Tyler_etal_2009})
and the references cited therein.

A special class of location and scatter functionals are multivariate
$M$-funct\-ionals. Introduced by Maronna (\citeyear{Maronna_1976}),
their properties have been analyzed by numerous authors, an incomplete
list of references being Huber (\citeyear{Huber_1981}), Hampel et al.\
(\citeyear{Hampel_etal_1986}), Tyler (\citeyear{Tyler_1987a,Tyler_1987b}),
Kent and Tyler (\citeyear{Kent_Tyler_1988,Kent_Tyler_1991})
and Dudley et~al.\ (\citeyear{Dudley_etal_2009}). In particular, Dudley
et~al.\ (\citeyear{Dudley_etal_2009}) prove existence and uniqueness of multivariate
$t$-functionals of location and scatter, generalizing results of Kent
and Tyler (\citeyear{Kent_Tyler_1988,Kent_Tyler_1991}).
Moreover, they provide an in-depth analysis of weak continuity and
differentiability of such functionals which implies consistency and
asymptotic normality of the corresponding estimators. Similar
considerations have been made by D\"umbgen (\citeyear{Duembgen_1998})
for the special $M$-functional of scatter due to Tyler (\citeyear
{Tyler_1987a}). As to the robustness of multivariate $t$-functionals of
location and scatter in terms of so-called breakdown points, we refer
to D\"umbgen and Tyler (\citeyear{Duembgen_Tyler_2005}) and the
references therein.

In many settings the location parameter $\bmu(P)$ is merely a nuisance
parameter while the main interest lies on the scatter matrix $\bSigma
(P)$. Moreover, often one only needs to know $\bSigma(P)$ up to a
positive scaling factor, e.g.\ when defining principal components or
correlations. On the other hand, a desirable feature is the following
block independence property: If $P$ describes the distribution of $X =
[X_1^\top, X_2^\top]^\top$ with two stochastically independent
random vectors $X_1 \in\R^{q(1)},X_2 \in\R^{q(2)}$, then $\bSigma
(P)$ should be block diagonal, i.e.
\[
\bSigma(P) \ = \
\begin{bmatrix}
\bSigma_1(P) & 0 \\
0 & \bSigma_2(P)
\end{bmatrix}
\]
with $\bSigma_i(P) \in\R^{q(i)\times q(i)}$. Unfortunately, the
$M$-functionals just mentioned do not have this property. However, as
explained later, any reasonable $M$-functional of scatter has the block
independence property when it is applied to the symmetrized
distribution $\LL(X - X')$ with independent random vectors $X, X' \sim
P$. (Here and throughout $\LL(Y)$ denotes the distribution of a random
variable $Y$, and $Y \sim Q$ is shorthand for ``$Y$ has distribution
$Q$''.) Note also that the symmetrized distribution $\LL(X - X')$ is
centered around $0 \in\R^q$, so we may avoid the estimation of a
location parameter and focus on estimation of scatter only. This trick
is used by many authors, e.g.\ Croux et al.\ (\citeyear
{Croux_etal_1994}), D\"umbgen (\citeyear{Duembgen_1998}), Sirki\"a et
al.\ (\citeyear{Sirkiae_etal_2007}), Nordhausen et al.\ (\citeyear
{Nordhausen_etal_2008}) and Tyler et al.\
(\citeyear{Tyler_etal_2009}).\looseness=1

Applying the $M$-functionals $\bmu(\cdot)$ and $\bSigma(\cdot)$ to
the empirical distribution $\hat{P}$ of independent random vectors
$X_1, X_2, \ldots, X_n$ with distribution $P$ yields $M$-\textit{esti\-mators}
$\hat{\mu} = \bmu(\hat{P})$ and $\hat{\Sigma} =
\bSigma(\hat{P})$.

The remainder of this survey is organized as follows: In Section~\ref
{sec:Equivariance} we review the concepts of affine and linear
equivariance and their main consequences. In Section~\ref{sec:MLE to
M} we motivate $M$-functionals of location and scatter by various
maximum-likelihood and other estimation problems. After these
introductory sections, we start with the main results about existence,
uniqueness, weak continuity and differentiability of the $M$-functionals.

The main part of our paper is devoted to scatter-only functionals,
treated in Sections~\ref{sec:Scatter}, \ref{sec:AnalysisL} and \ref
{sec:Scatter 2}. This is done in a generalized framework with
matrix-valued random variables. By doing so we reveal a connection
between Tyler's (\citeyear{Tyler_1987a}) $M$-functional of scatter and
the estimation of proportional covariance matrices as treated by Flury
(\citeyear{Flury_1986}), Eriksen (\citeyear{Eriksen_1987}) and Jensen
and Johansen (\citeyear{Jensen_Johansen_1987}). Moreover, this general
framework allows us to treat a new class of scatter estimators, based
on symmetrizations of arbitrary order. Part of this material is new.
Section~\ref{sec:Scatter} contains the main results about existence
and uniqueness of the scatter functionals. Section~\ref{sec:AnalysisL}
provides analytical tools to derive the aforementioned and later
results. As realized by Auderset et al.\ (\citeyear
{Auderset_etal_2005}) in the context of multivariate (real or complex)
Cauchy distributions and by Wiesel (\citeyear{Wiesel_2012}), among
others, working with matrix exponentials and logarithms in a suitable
way provides valuable new insights, and we are utilizing this approach,
too. In particular, the target functions to be minimized turn out to be
(strictly) convex in a certain sense which is essential for uniqueness.
In our opinion, the resulting proofs are more intuitive than some
derivations in the original papers. Based on the analytical results in
Section~\ref{sec:AnalysisL}, we discuss weak continuity and weak
differentiability of scatter functionals in Section~\ref{sec:Scatter 2}.

Finally, in Section~\ref{sec:Location and Scatter} we review a trick
by Kent and Tyler (\citeyear{Kent_Tyler_1991}) to treat location and
scatter functionals based on multivariate $t$-distributions by means of
the scatter-only methods. This allows one to prove weak
differentiability and central limit theorems as in Dudley et al.\
(\citeyear{Dudley_etal_2009}).

Various auxiliary results and most proofs are deferred to Section~\ref
{sec:Proofs}.

\paragraph{Notation}
Throughout this paper, the standard Euclidean norm of a vector $v \in
\R^d$ is denoted by $\|v\| = \sqrt{v^\top v}$. For matrices $A,B \in
\R^{q \times d}$ we use either the operator or the Frobenius norm,
\begin{align*}
\|A\| \
&:= \ \max_{v \in\R^d \setminus\{0\}} \, \frac{\|Av\|}{\|v\|}
\ = \ \max_{v \in\R^d : \|v\| = 1} \, \|Av\|, \\
\|A\|_F \
&:= \ \Bigl( \sum_{i,j} A_{ij}^2 \Bigr)^{1/2}
\ = \ \langle A,A\rangle^{1/2},
\end{align*}
where
\[
\langle A,B\rangle\ := \ \sum_{i,j} A_{ij} B_{ij}
\ = \ \tr(A^\top B)
\ = \ \tr(A B^\top).
\]
Note that $\langle A,B\rangle$ defines an inner product on $\R^{q
\times d}$. If $\mathrm{vec}(A)$ and $\mathrm{vec}(B)$ denote vectors
in $\R^{qd}$ containing the columns of $A$ and $B$, respectively, then
$\langle A,B\rangle$ is just the usual inner product $\mathrm
{vec}(A)^\top\mathrm{vec}(B)$. We shall consider the following
subsets of $\Rqq$:
\begin{align*}
\Rqqns\
:=& \ \bigl\{ A \in\Rqq: A \ \text{nonsingular} \bigr\}, \\
\Rqqsym\
:=& \ \bigl\{ A \in\Rqq: A = A^\top\bigr\}, \\
\Rqqsympsd\
:=& \ \bigl\{ A \in\Rqqsym: A \ \text{positive semidefinite} \bigr
\} \\
=& \ \bigl\{ A\in\Rqqsym: \lambda_{\rm min}(A) \ge0 \bigr\}, \\
\Rqqsympd\
:=& \ \bigl\{ A \in\Rqqsym: A \ \text{positive definite} \bigr\} \\
=& \ \bigl\{ A\in\Rqqsym: \lambda_{\rm min}(A) > 0 \bigr\}.
\end{align*}
With $\lambda_{\rm min}(A)$ and $\lambda_{\rm max}(A)$ we denote the
smallest and largest real eigenvalue of a square matrix $A$. If $A \in
\Rqq$ has only real eigenvalues (e.g.\ if $A = A^\top$), then
$\lambda_1(A) \ge\lambda_2(A) \ge\cdots\ge\lambda_q(A)$ are its
ordered eigenvalues. The identity matrix in $\Rqq$ is denoted by $I_q$.

In the sequel we will introduce further notation and various
conditions. For the reader's convenience, these are listed once more at
the very end of this paper.


\section{Affine and linear equivariance}
\label{sec:Equivariance}

Affine and linear equivariance are key concepts in connection with
estimation of location and scatter. In what follows, let $\PP$ be a
family of probability distributions on $\R^q$.
For $P \in\PP$, a vector $a \in\R^q$ and a matrix $B \in\Rqqns$ let
\[
P^B \ := \ \LL(BX)
\quad\text{and}\quad
P^{a,B} \ := \ \LL(a + BX)
\quad\text{where} \ X \sim P.
\]

\begin{Definition}[Linear equivariance]
Suppose that $\PP$ is \textit{linear invariant} in the sense that
$P^B \in\PP$ for arbitrary $P \in\PP$ and $B \in\Rqqns$. A
scatter functional $\bSigma: \PP\to\Rqqsympsd$ is called
\textit{linear equivariant} if
\[
\bSigma(P^B) \ = \ B \bSigma(P) B^\top
\]
for arbitrary $P \in\PP$ and $B \in\Rqqns$.
\end{Definition}

\begin{Definition}[Affine equivariance]
Suppose that $\PP$ is \textit{affine invariant} in the sense that
$P^{a,B} \in\PP$ for arbitrary $P \in\PP$, $a \in\R^q$ and $B \in
\Rqqns$. Consider a location functional $\bmu: \PP\to\R^q$ and a
scatter functional $\bSigma: \PP\to\Rqqsympsd$. These functionals
are called \textit{affine equivariant} if
\[
\bmu(P^{a,B}) \ = \ a + B \bmu(P)
\quad\text{and}\quad
\bSigma(P^{a,B}) \ = \ B \bSigma(P) B^\top
\]
for arbitrary $P \in\PP$, $a \in\R^q$ and $B \in\Rqqns$.
\end{Definition}

These definitions are clearly motivated by the mean vector $\bmu(P)$
and covariance matrix $\bSigma(P)$, where $\PP$ consists of all
distributions $P$ with finite integral $\int\|x\|^2 \, P(dx)$.
Whenever we talk about affine or linear equivariant functionals on a
set $\PP$, we assume tacitly that $\PP$ is affine or linear invariant.

Obviously, affine equivariance of a scatter functional $\bSigma(\cdot
)$ implies its linear equivariance. Equivariance properties of location
and scatter functionals yield various desirable properties which are
summarized in two lemmas below. Let us first recall two symmetry
properties of a distribution $P$:

\begin{Definition}[Spherical and elliptical symmetry]
Let $X$ be a random vector with distribution $P$ on $\R^q$.

\vspace*{3pt}
\noindent
\textbf{(i)} \ The distribution $P$ is called spherically symmetric
(around $0$) if the distributions of $X$ and $UX$ coincide for any
orthogonal matrix $U \in\Rqq$.

\noindent
\textbf{(ii)} \ The distribution $P$ is called elliptically symmetric
with center $\mu\in\R^q$ and scatter matrix $\Sigma\in\Rqqsympd$,
if the distribution of $\Sigma^{-1/2}(X - \mu)$ is spherically symmetric.
\end{Definition}

If the distribution $P$ admits a density $f$, elliptical symmetry with
center $\mu$ and scatter matrix $\Sigma$ means that $f(x)$ is a
function of the squared Mahalanobis distance $(x - \mu)^\top\Sigma
^{-1} (x - \mu)$ only. In particular, if $P$ is spherically symmetric,
$f(x)$ depends only on the norm $\|x\|$.

Note that the scatter matrix $\Sigma$ of an elliptically symmetric
distribution is not unique. One could replace $\Sigma$ with $c\Sigma$
for any $c > 0$.

\begin{Lemma}[Some consequences of linear equivariance]
\label{lem: linear equivariance}
Let $\bSigma: \PP\to\Rqqsympsd$ be a linear equivariant functional
of scatter, and let $X$ be a random vector with distribution $P
\in\PP$.\vadjust{\goodbreak}

\noindent
\textbf{(i)} \ Let $J$ be a subset of $\{1,2,\ldots,q\}$ with two or
more elements. Suppose that the distributions of $X$ and $(X_{\pi
(i)})_{i=1}^q$ coincide for any permutation $\pi$ of $\{1,2,\ldots,q\}
$ such that $\pi(i) = i$ whenever $i \not\in J$. Then there exist
numbers $a = a(P)$ and $b = b(P)$ such that for arbitrary indices $j,k
\in J$,
\[
\bSigma(P)_{jk} \ = \
\begin{cases}
a & \text{if} \ j = k, \\
b & \text{if} \ j \ne k.
\end{cases}
\]

\noindent
\textbf{(ii)} \ Suppose that for a given sign vector $s \in\{-1,1\}
^q$, the distributions of $X$ and $(s_i X_i)_{i=1}^q$ coincide. Then
\[
\bSigma(P)_{ij} \ = \ 0
\quad
\text{whenever} \ \ s_i \ne s_j.
\]

\noindent
\textbf{(iii)} \ If $P$ is elliptically symmetric with center $0 \in
\R^q$ and scatter matrix $\Sigma\in\Rqqsympd$, then
\[
\bSigma(P) \ = \ c(P) \Sigma
\]
for some number $c(P) \ge0$.
\end{Lemma}

\begin{Lemma}[Some consequences of affine equivariance]
\label{lem: affine equivariance}
Let $\bmu: \PP\to\R^q$ and $\bSigma: \PP\to\Rqqsympsd$ be
affine equivariant functionals of location and scatter, respectively,
and let $X$ be a random vector with distribution $P \in\PP$.

\vspace*{3pt}\noindent
\textbf{(i)} \ Suppose that for a given vector $s \in\{-1,1\}^q$, the
distributions of $X$ and $(s_iX_i)_{i=1}^q$ coincide. Then
\[
\bmu(P)_i \ = \ 0
\quad
\text{whenever} \ \ s_i = -1.
\]

\noindent
\textbf{(ii)} \ If $P$ is elliptically symmetric with center $\mu\in
\R^q$ and scatter matrix $\Sigma\in\Rqqsympd$, then
\[
\bmu(P) \ = \ \mu
\quad\text{and}\quad
\bSigma(P) \ = \ c(P) \Sigma
\]
for some number $c(P) \ge0$.
\end{Lemma}

\begin{Remark}[Symmetrization and the block independence property]
Suppose that $X \sim P$ may be written as $X = [X_1^\top, X_2^\top
]^\top$ with two independent subvectors $X_i \in\R^{q(i)}$, $q(1) +
q(2) = q$. Let $X'$ be an independent copy of $X$. If $\bSigma: \PP
\to\Rqqsympsd$ is a linear equivariant scatter functional, and if
$\tilde{P} := \LL(X - X')$ belongs to $\PP$,
\[
\bSigma(\tilde{P}) \ = \
\begin{bmatrix}
\bSigma_1(\tilde{P}) & 0 \\
0 & \bSigma_2(\tilde{P})
\end{bmatrix}
\]
with $\bSigma_i(\tilde{P}) \in\R^{q(i)\times q(i)}$. This follows
from Lemma~\ref{lem: linear equivariance}~(ii), applied to $\tilde{X}
\sim\tilde{P}$ in place of $X \sim P$ and $s_i := 1_{[i \le q(1)]} -
1_{[i > q(1)]}$. If $\PP$ is even affine invariant and $\bmu: \PP\to
\R^q$ an affine equivariant location functional, then $\bmu(\tilde
{P}) = 0$ by Lemma~\ref{lem: affine equivariance}.
\end{Remark}


\section[From maximum-likelihood estimation to $M$-functionals]{From maximum-likelihood estimation to $\boldsymbol{M}$-functionals}
\label{sec:MLE to M}

In this section we describe various estimation problems and the
$M$-functionals which they lead to.

\subsection{Estimation in location-scatter families}
\label{subsec:LocationScatter}

Let $X_1, X_2, \ldots, X_n$ be independent random vectors with unknown
distribution~$P$. As a model for $P$ we consider a location-scatter
family constructed as follows: Let $\tilde{f} : [0,\infty) \to
[0,\infty)$ satisfy
\[
\tilde{c} := \int_{\R^q} \tilde{f}(\|x\|^2) \, dx \ \in\ (0,\infty
).
\]
For any location parameter $\mu\in\R^q$ and scatter parameter
$\Sigma\in\Rqqsympd$,
\[
f_{\mu,\Sigma}(x) \ := \ \tilde{c}^{-1} \det(\Sigma)^{-1/2}
\tilde{f} \bigl( (x - \mu)^\top\Sigma^{-1} (x - \mu) \bigr)
\]
defines a probability density $f_{\mu,\Sigma}$ on $\R^q$. Assuming
that $P$ has a density belonging to this family $(f_{\mu,\Sigma
})_{\mu,\Sigma}$, a maximum-likelihood estimator of $(\mu,\Sigma)$
is a maximizer $(\hat{\mu},\hat{\Sigma})$ of the likelihood function
\[
(\mu,\Sigma) \ \mapsto\ \prod_{i=1}^n f_{\mu,\Sigma}(X_i).
\]
In other words, $(\hat{\mu},\hat{\Sigma})$ minimizes
\[
\hat{L}(\mu,\Sigma)
\ := \ \frac{1}{n} \sum_{i=1}^n
\rho\bigl( (X_i - \mu)^\top\Sigma^{-1} (X_i - \mu) \bigr)
+ \log\det(\Sigma)
\]
with
\[
\rho(s) \ := \ - 2 \log\tilde{f}(s).
\]
The expected value of $\hat{L}(\mu,\Sigma)$ equals
\begin{equation}
\label{eq:def L 1}
L(\mu,\Sigma,P)
\ := \ \int\rho\bigl( (x - \mu)^\top\Sigma^{-1} (x - \mu) \bigr
) \, P(dx)
+ \log\det(\Sigma),
\end{equation}
provided this integral exists, and
\[
\hat{L}(\mu,\Sigma) \ = \ L(\mu,\Sigma,\hat{P})
\]
with $\hat{P}$ denoting the empirical distribution $n^{-1} \sum
_{i=1}^n \delta_{X_i}$ of the observations $X_i$. Consequently we
focus on $L(\mu,\Sigma,P)$ for arbitrary distributions $P$, keeping
in mind that $P$ could be a ``true'' or an empirical distribution.

Suppose that $P$ has a density $f$ which may but need not belong to the
model $(f_{\mu,\Sigma})_{\mu,\Sigma}$ and such that $\int f(x) \log
f(x) \, dx$ exists in $\R$. Then
\begin{align*}
L(\mu,\Sigma,P) - 2 \log\tilde{c} \
&= \ - 2 \int f(x) \log f_{\mu,\Sigma}(x) \, dx \\
&= \ - 2 \int f(x) \log f(x) \, dx + 2 D(f, f_{\mu,\Sigma})
\end{align*}
with the Kullback-Leibler divergence
\[
D(f,f_{\mu,\Sigma}) \ := \ \int f(x) \log\bigl( f(x) / f_{\mu
,\Sigma}(x) \bigr) \, dx.\vadjust{\goodbreak}
\]
It is well-known that $D(f, f_{\mu,\Sigma}) \ge0$ with equality if,
and only if, $f = f_{\mu,\Sigma}$ almost everywhere.
Thus minimizing $L(\mu,\Sigma,P)$ w.r.t.\ $(\mu,\Sigma)$ may be
viewed as approximating $P$ by one of the densities $f_{\mu,\Sigma}$
in terms of the Kullback-Leibler divergence.

\begin{Example}[Gaussian distributions]
Multivariate (nondegenerate) Gaussian distributions correspond to
$\tilde{f}(s) := \exp(- s/2)$ and $\tilde{c} := (2\pi)^{q/2}$,
i.e.\ $\rho(s) := s$. Suppose that $P$ has mean vector $\bmu(P)$,
finite integral $\int\|x\|^2 \, P(dx)$ and nonsingular covariance
matrix $\bSigma(P)$. Then
\begin{align*}
L(\mu,\Sigma,P) \
&= \ \int(x - \mu)^\top\Sigma^{-1} (x - \mu) \, P(dx) + \log\det
(\Sigma) \\
&= \ \int(x - \bmu(P))^\top\Sigma^{-1} (x - \bmu(P)) \, P(dx) +
\log\det(\Sigma)
\\
&\quad+ \ (\mu- \bmu(P))^\top\Sigma^{-1}(\mu- \bmu(P)).
\end{align*}
Hence for any fixed $\Sigma$, the unique minimizer of $\mu\mapsto
L(\mu,\Sigma,P)$ equals $\mu= \bmu(P)$. Moreover,
\begin{align*}
L(\bmu(P),\Sigma,P) \
&= \ \tr(\Sigma^{-1} \bSigma(P)) + \log\det(\Sigma) \\
&= \ \tr(\Sigma^{-1} \bSigma(P)) - \log\det(\Sigma^{-1} \bSigma(P))
+ \log\det(\bSigma(P)).
\end{align*}
Note that $\tr(\Sigma^{-1} \bSigma(P)) - \log\det(\Sigma^{-1}
\bSigma(P))$ equals $\tr(B) - \log\det(B)$ with the symmetric
matrix $B := \Sigma^{-1/2} \bSigma(P) \Sigma^{-1/2}$. If $\lambda_1
\ge\lambda_2 \ge\cdots\ge\lambda_q > 0$ denote the eigenvalues of
$B$, then
\[
\tr(B) - \log\det(B) \ = \ \sum_{i=1}^q (\lambda_i - \log\lambda
_i) \ \ge\ q
\]
with equality if, and only if, all eigenvalues $\lambda_i$ are equal
to one, i.e.\ if $\Sigma= \bSigma(P)$. Thus $(\bmu(P),\bSigma(P))$
is the unique minimizer of $L(\cdot,\cdot,P)$.
\end{Example}

The range of distributions $P$ for which $L(\mu,\Sigma,P)$ is
well-defined in $\R$ for arbitrary $(\mu,\Sigma)$ may become larger
if we replace the term $\rho\bigl( (x - \mu)^\top\Sigma^{-1} (x -
\mu) \bigr)$ with a difference
\[
\rho\bigl( (x - \mu)^\top\Sigma^{-1} (x - \mu) \bigr)
- \rho\bigl( (x - \mu_o)^\top\Sigma_o^{-1} (x - \mu_o) \bigr)
\]
for some $(\mu_o,\Sigma_o)$. The choice of the latter pair is
irrelevant, so we use $\mu_o = 0$ and $\Sigma_o = I_q$, where $I_q$
denotes the unit matrix in $\Rqq$.

\begin{Definition}[$M$-functionals of location and scatter]
\label{def:L}
Let $\rho: [0,\infty) \to\R$ be some continuous function. Further
let $\PP$ be the set of all probability distributions $P$ on $\R^q$
such that
\begin{equation}
\label{eq:def L 2}
L(\mu,\Sigma,P)
\ := \ \int\bigl[ \rho\bigl( (x - \mu)^\top\Sigma^{-1} (x - \mu
) \bigr)
- \rho(x^\top x) \bigr] \, P(dx)
+ \log\det(\Sigma)
\end{equation}
is well-defined in $\R$ for arbitrary $(\mu,\Sigma) \in\R^q \times
\Rqqsympd$.

With $\PP_\rho$ we denote the set of all distributions $P \in\PP$
such that $L(\cdot,\cdot,P)$ has a unique minimizer $(\bmu
(P),\bSigma(P))$. This defines an $M$-functional $\bmu: \PP_\rho\to
\R^q$ of location and an $M$-functional $\bSigma: \PP_\rho\to
\Rqqsympd$ of scatter.
\end{Definition}

\paragraph{Affine equivariance}
The set $\PP$ in Definition~\ref{def:L} is affine invariant. Indeed,
if $X \sim P \in\PP$ and $X' := a + BX \sim P^{a,B}$, then elementary
calculations show that
\begin{align*}
\rho\bigl(
& (X' - \mu')^\top{\Sigma'}^{-1} (X' - \mu') \bigr)
- \rho({X'}^\top X') \\
&= \ \Bigl[ \rho\bigl( (X - \mu)^\top\Sigma^{-1} (X - \mu) \bigr)
- \rho(X^\top X) \Bigr] \\
&\quad- \ \Bigl[ \rho\bigl( (X - \mu'')^\top{\Sigma''}^{-1} (X -
\mu'') \bigr)
- \rho(X^\top X) \Bigr],
\end{align*}
where $\mu' := a + B\mu$, $\Sigma' := B \Sigma B^\top$ and $\mu''
:= - B^{-1} a$, $\Sigma'' := (B^\top B)^{-1}$. Since $\log\det
(\Sigma') = \log\det(\Sigma) + 2 \log|\det(B)| = \log\det
(\Sigma) - \log\det(\Sigma'')$, we arrive at the key equation
\begin{equation}
\label{eq:equivariance L}
L \bigl( a + B\mu, B\Sigma B^\top, P^{a,B} \bigr)
\ = \ L(\mu,\Sigma,P) + c(a,B,P)
\end{equation}
with $c(a,B,P) := - L \bigl( -B^{-1}a, (B^\top B)^{-1}, P \bigr)$. In
particular, the set $\PP_\rho$ is affine invariant, and the
$M$-functionals $\bmu(\cdot)$, $\bSigma(\cdot)$ are affine equivariant.

\begin{Example}[Multivariate $t$-distributions]
The multivariate student-distrib\-utions are generated by $\tilde
{f}(s) := (\nu+ s)^{-(\nu+q)/2}$ for a fixed parameter $\nu> 0$, the
``degrees of freedom'',
and $\tilde{c} = \nu^{-\nu/2} \pi^{q/2} \Gamma(\nu/2) / \Gamma
((\nu+ q)/2)$. Here
\[
\rho(s) \ = \ (\nu+ q) \log(\nu+ s).
\]
With this choice of $\rho$, definition \eqref{eq:def L 2} yields
\begin{align}
\nonumber
L_\nu&(\mu,\Sigma,P) \\
\label{eq:def L_nu}
&= \ (\nu+ q) \int\log\Bigl(
\frac{\nu+ (x - \mu)^\top\Sigma^{-1} (x - \mu)}{\nu+ \|x\|^2}
\Bigr)
\, P(dx)
+ \log\det(\Sigma).
\end{align}
Since the integrand is continuous and bounded on $\R^q$ for any fixed
$(\mu,\Sigma)$, the set $\PP$ is just the set of \textit{all}
probability distributions on $\R^q$. In later sections we shall derive
a precise description of the corresponding subset $\PP_\rho$.
\end{Example}

\subsection{Tyler's (1987) $M$-functional of scatter and more}
\label{subsec:Tyler 1987 and more}

\paragraph{A maximum-likelihood estimator for directional data}
Tyler (\citeyear{Tyler_1987a,Tyler_1987b}) introduced a
particular $M$-estimator of scatter which may be motivated as follows:
Suppose that $X_1$, $X_2$, \ldots, $X_n$ are independent random
vectors with possibly different distributions $P_1$, $P_2$, \ldots,
$P_n$ on $\R^q$. However, suppose that each $P_i$ satisfies $P_i(\{0\}
) = 0$ and is elliptically symmetric with center $0$ and a common
scatter matrix $\Sigma$. This assumption means that $X_i = R_i B U_i$
with $B := \Sigma^{1/2}$ and $2n$ stochastically independent random
variables $R_1, R_2, \ldots, R_n > 0$ and $U_1, U_2, \ldots, U_n$
uniformly distributed on the unit sphere $\mathbb{S}^{q-1}$ of $\R
^q$. In particular, the directional vectors $V_i := \|X_i\|^{-1}X_i = \|
BU_i\|^{-1} BU_i$ are independent and identically distributed random vectors.
One can show that $V_i$ possesses a so called angular central Gaussian
distribution, i.e.
its distribution is absolutely continuous with respect to the uniform
distribution on $\mathbb{S}^{q-1}$ with density
\[
g_\Sigma(v) \ := \ \det(\Sigma)^{-1/2} (v^\top\Sigma^{-1} v)^{-q/2},
\]
see e.g. Watson (\citeyear{Watson_1983}). Consequently, a
maximum-likelihood estimator for $\Sigma$ is given by a maximizer of
the target function $L(\Sigma,\hat{P})$ over all matrices $\Sigma\in
\Rqqsympd$, where $\hat{P}$ is again the empirical distribution of
the $X_i$, and
\begin{equation}
\label{eq:def L_0}
L(\Sigma, P) \ := \ q \int\log\Bigl( \frac{x^\top\Sigma^{-1}
x}{x^\top x} \Bigr) \, P(dx)
+ \log\det(\Sigma)
\end{equation}
for any distribution $P$ on $\R^q$ with $P(\{0\}) = 0$. Note that
$L(\Sigma,P) = L(c\Sigma,P)$ for any $c > 0$. To achieve uniqueness
of a minimizer, we have to impose an additional constraint, e.g.
\[
\det(\Sigma) \ \stackrel{!}{=} \ 1,
\]
following Paindaveine's (\citeyear{Paindaveine_2008}) advice.

\paragraph{Estimation of proportional covariance matrices}
Suppose that one observes independent random matrices $S_1, S_2, \ldots
, S_K \in\Rqqsympsd$, where $S_i$ has a Wishart distribution
$\mathcal{W}_q(c_i \Sigma, m_i)$. The degrees of freedom, $m_1, m_2,
\ldots, m_K$, are given, while $c_1, c_2, \ldots, c_K > 0$ and
$\Sigma\in\Rqqsympd$ are unknown parameters.

As an explicit example, suppose that we observe independent random
vectors $X_{ij} \in\R^q$ for $1 \le i \le K$ and $1 \le j \le n_i$,
where $n_i = m_i + 1 \ge2$ and
\[
X_{ij}^{} \ \sim\ \NN_q(\mu_i^{}, c_i^{} \Sigma)
\]
with unknown means $\mu_i \in\R^q$. With $\bar{X}_i := n_i^{-1}
\sum_{j=1}^{n_i} X_{ij}$, the standard estimator of $\mu_i$, it is
well-known that
\[
S_i \ := \ \sum_{j=1}^{n_i} (X_{ij} - \bar{X}_i)(X_{ij} - \bar
{X}_i)^\top
\ \sim\ \mathcal{W}_q(c_i \Sigma, m_i).
\]

Recalling that $\mathcal{W}_q(\Gamma,m)$ stands for the distribution
of $\sum_{j=1}^m Y_i^{} Y_i^\top$ with independent random vectors
$Y_1, \ldots, Y_m \sim\NN_q(0,\Gamma)$, the log-likelihood function
times $-2$ may be written as
\begin{equation}
\label{eq:2NLL proportional covariances}
\sum_{i=1}^K \bigl(
c_i^{-1} \tr(\Sigma^{-1} S_i) + q m_i \log c_i + m_i \log\det
(\Sigma) \bigr).
\end{equation}
Minimization of this function was treated by Flury (\citeyear
{Flury_1986}), Eriksen (\citeyear{Eriksen_1987}) and Jensen and
Johansen (\citeyear{Jensen_Johansen_1987}). The proposed algorithms
rely on the fact that \eqref{eq:2NLL proportional covariances}, as a
function of the two arguments $\Sigma$ and $c = (c_i)_{i=1}^K$, is
easily minimized if one of the two arguments is fixed. For fixed
$\Sigma$, the unique minimizer is
\[
c(\Sigma) \ := \ q^{-1} \bigl( m_i^{-1} \tr(\Sigma^{-1} S_i) \bigr
)_{i=1}^K,
\]
whereas for fixed $c$, the unique minimizer is
\[
\Sigma(c) \ := \ m_+^{-1} \sum_{i=1}^K c_i^{-1} S_i
\]
with $m_+ := \sum_{i=1}^K m_i$. If focusing on the estimation of the
matrix parameter $\Sigma$, we may plug $c(\Sigma)$ into \eqref
{eq:2NLL proportional covariances} and try to minimize the resulting
function of $\Sigma$. Up to an additive term and a scaling factor
$m_+^{-1}$, the latter function equals
\begin{equation}
\label{eq:2NLL proportional covariances 2}
q \sum_{i=1}^K \frac{m_i}{m_+} \, \log\Bigl( \frac{\tr(\Sigma
^{-1} S_i)}{\tr(S_i)} \Bigr)
+ \log\det(\Sigma).
\end{equation}
Again one should impose some constraint such as $\det(\Sigma)
\stackrel{!}{=} 1$ to avoid non-uniqueness of the minimizer.

\paragraph{A generalized setting}
Note the similarity between \eqref{eq:def L_0} and \eqref{eq:2NLL
proportional covariances 2}. Consider the distribution $Q$ of the
random matrix $XX^\top$, where $X \sim P$.
Then $L(\Sigma,P)$ in \eqref{eq:def L_0} may be rewritten as
\[
q \int\log\Bigl( \frac{\tr(\Sigma^{-1}M)}{\tr(M)} \Bigr) \, Q(dM)
+ \log\det(\Sigma),
\]
where $M$ corresponds to $xx^\top$ with $x \in\R^q$. But \eqref
{eq:2NLL proportional covariances 2} is also of this form, this time
with the random distribution
\[
\hat{Q} \ := \ \sum_{i=1}^K \frac{m_i}{m_+} \, \delta_{S_i}^{}
\]
in place of $Q$. These considerations motivate the following definition.

\begin{Definition}[Generalized version of Tyler's $M$-functional of scatter]
\label{def:Tyler generalized}
For a distribution $Q$ on $\Rqqsympsd\setminus\{0\}$ and $\Sigma\in
\Rqqsympd$ we define
\[
L_0(\Sigma,Q)
\ := \ q \int\log\Bigl( \frac{\tr(\Sigma^{-1}M)}{\tr(M)} \Bigr)
\, Q(dM)
+ \log\det(\Sigma).
\]
If $L_0(\cdot,Q)$ has a unique minimizer $\Sigma$ satisfying $\det
(\Sigma) = 1$, then we denote it with $\Sigma_0(Q)$.
\end{Definition}

\subsection{Symmetrizations of arbitrary order}

For $k \ge2$ vectors $x_1, \ldots, x_k \in\R^q$ we define their
sample covariance matrix as
\[
S(x_1, \ldots, x_k) \ := \ \frac{1}{k-1} \sum_{i=1}^k (x_i - \bar
{x})(x_i - \bar{x})^\top
\]
with $\bar{x} := k^{-1} \sum_{i=1}^k x_i$. If $X_1, X_2, \ldots,
X_n$ are independent random vectors with distribution $P$ such that
$\int\|x\|^2 \, P(dx) < \infty$, then $S(X_1, X_2, \ldots, X_n)$ is
an unbiased estimator of the covariance matrix of $P$.\vadjust{\goodbreak} Elementary
calculations show that
\begin{align*}
S(X_1, X_2, \ldots, X_n) \
&= \ {\binom{n}{2}}^{-1} \sum_{1 \le i < j \le n}
2^{-1} (X_i - X_j)(X_i - X_j)^\top\\
&= \ {\binom{n}{2}}^{-1} \sum_{1 \le i < j \le n} S(X_i, X_j).
\end{align*}
More generally, for $2 \le k \le n$,
\[
S(X_1, X_2, \ldots, X_n)
\ = \ {\binom{n}{k}}^{-1} \sum_{1 \le i_1 < \cdots< i_k \le n}
S(X_{i_1}, \ldots, X_{i_k}).
\]
Instead of taking the average of all sample covariance matrices
$S(X_{i_1},\ldots,X_{i_k})$ one could apply Tyler's generalized
$M$-functional of scatter (Definition \ref{def:Tyler generalized}) or
other functionals of scatter to the random distribution
\[
{\binom{n}{k}}^{-1} \sum_{1 \le i_1 < \cdots< i_k \le n}
\delta_{S(X_{i_1}, \ldots, X_{i_k})}^{}
\]
on $\Rqqsympsd$, a measure-valued $U$-statistic (cf.\ Hoeffding,
\citeyear{Hoeffding_1948}). For $k = 2$ this approach was proposed by
D\"umbgen (\citeyear{Duembgen_1998}). Apart from the higher
computational complexity, trying $k \ge3$ is tempting.

\subsection{Simultaneous symmetrization in several samples}

Suppose we observe independent random vectors $X_{ij} \in\R^q$, where
$i=1,2,\ldots,K$ and $j = 1,2,\ldots,n_i$, $n_i \ge2$. Suppose that
$X_{ij}$ has an unknown elliptically symmetric distribution $P_i$ with
center $\mu_i \in\R^q$ and a common scatter matrix $\Sigma\in
\Rqqsympd$. In case of $P_i = \NN_q(\mu_i,\Sigma)$ one could
estimate $\Sigma$ by the usual pooled covariance matrix
\begin{align*}
\hat{\Sigma} \
&= \ \frac{1}{n_+ - K} \sum_{i=1}^K (n_i - 1) S(X_{i1},X_{i2},\ldots
,X_{in_i}) \\
&= \ \frac{2}{n_+ - K} \sum_{i=1}^K \frac{1}{n_i}
\sum_{1 \le j < \ell\le n_i} S(X_{ij},X_{i\ell}).
\end{align*}
Alternatively, one could estimate $\Sigma$ by a minimizer of \eqref
{eq:2NLL proportional covariances 2}. But in case of potentially
heavy-tailed distributions $P_i$, it might be
even better to apply Tyler's generalized $M$-functional of scatter
(Definition \ref{def:Tyler generalized}) or other functionals of
scatter to the random distribution
\[
\frac{2}{n_+ - K} \sum_{i=1}^K \frac{1}{n_i}
\sum_{1 \le j < \ell\le n_i} \delta_{S(X_{ij},X_{i\ell})}^{}
\]
on $\Rqqsympsd$.\vadjust{\goodbreak}

The resulting scatter estimator $\hat{\Sigma}$ could be used, for
instance, in the context of nearest-neighbor classification to define a
data-driven Mahalanobis distance $\hat{d}(x,y) := \bigl\| \hat
{\Sigma}^{-1/2}(x - y) \bigr\|$ between vectors $x,y \in\R^q$.


\section[$M$-functionals of scatter]{$\boldsymbol{M}$-functionals of scatter}
\label{sec:Scatter}

In this section we consider $M$-functionals of scatter only. That
means, when thinking about a distribution on $\R^q$, we assume that it
has a given center $\mu= 0$. In view of the considerations in the
preceding section, however, we consider distributions $Q$ on
$\Rqqsympsd$. Two particular examples for $Q$ are
\begin{align}
\label{eq:Example Q 1}
Q^1(P) \
&:= \ \LL(XX^\top)
\intertext{and}
\label{eq:Example Q 2}
Q^k(P) \
&:= \ \LL\bigl( S(X_1,X_2,\ldots,X_k) \bigr),
\quad k \ge2,
\end{align}
for independent, identically distributed random vectors $X, X_1$,
$X_2$, \ldots, $X_k$ with distribution $P$ on $\R^q$.

\subsection{Definitions and basic properties}

\begin{Definition}[A log-likelihood type criterion]
\label{def:MF_of_Scatter}
For a given ``loss function'' $\rho :[0,\infty)$ $\to\R$ we define
\[
L_\rho(\Sigma,Q)
\ := \ \int\bigl[ \rho(\tr(\Sigma^{-1}M)) - \rho(\tr(M)) \bigr]
\, Q(dM)
+ \log\det\Sigma
\]
for $\Sigma\in\Rqqsympd$, provided that the integral exists in $\R$.
\end{Definition}

\paragraph{Assumptions on $\rho$ and $Q$}
Throughout we assume that $\rho$ is continuously differentiable on
$(0,\infty)$ with derivative $\rho' > 0$. Moreover, we assume that
\[
\psi(s) \ := \ s \rho'(s).
\]
is non-decreasing in $s > 0$.

\paragraph{Case 0} For $s > 0$ let
\[
\rho(s) \ := \ q \log(s),
\]
so $\rho'(s) = q/s$ and $\psi(s) = q$. Here we assume that $Q(\{0\})
= 0$.

\paragraph{Case 1} We assume that $\psi$ is strictly increasing on
$(0,\infty)$ with limits $\psi(0) = 0$ and $\psi(\infty) \in
(q,\infty]$. Here we assume that
\begin{equation}
\label{eq:Finite moments}
\int\psi(\lambda\tr(M)) \, Q(dM) \ < \ \infty
\quad\text{for any} \ \lambda\ge1,
\end{equation}
which is obviously true in case of $\psi(\infty) < \infty$.\vadjust{\goodbreak}

\begin{Remark}
\label{Remark: Log-likelihood}
Note that Tyler's generalized $M$-functional of
scatter (Definition~\ref{def:Tyler generalized}) corresponds to Case~0 above. In Case~1,
if $Q= Q^1(P)$ as in \eqref{eq:Example Q 1}, then $L_\rho(\cdot, Q)$
corresponds to the log-likelihood function $L(0,\Sigma,P)$ for an
elliptical model with $\tilde{f}(s) := \exp(-\rho(s)/2)$. Note that
for $0 < s_o < s$,
\[
\rho(s) \ = \ \rho(s_o) + \int_{s_o}^s \psi(t) t^{-1} \, dt
\
\begin{cases}
\le\ \rho(s_o) + \psi(\infty) \log(s/s_o), \\
\ge\ \rho(s_o) + \psi(s_o) \log(s/s_o).
\end{cases}
\]
This implies that
\[
\int_{\R^q} \exp\bigl( - \rho(\|x\|^2)/2 \bigr) \, dx
\ = \ C_q \int_0^\infty\exp\bigl( - \rho(s)/2 + (q/2 - 1) \log(s)
\bigr) \, ds
\]
is finite if, and only if, $\psi(\infty) > q$.
\end{Remark}

\begin{Remark}
Several authors require in addition $\rho'$ to be non-increasing on
$(0,\infty)$. Then
\[
\psi(\lambda s)
\ = \ \lambda s \rho'(\lambda s)
\ \le\ \lambda\psi(s)
\]
for any $s > 0$ and $\lambda\ge1$, whence \eqref{eq:Finite moments}
is equivalent to
\[
\int\psi(\tr(M)) \, Q(dM) \ < \ \infty.
\]
\end{Remark}

\begin{Example}[Multivariate $t$-distributions]
For $\nu\ge0$ let
\[
\rho(s) = \rho_{\nu,q}(s) \ := \ (\nu+ q) \log(\nu+ s).
\]
In case of $\nu> 0$, $L_\rho(\Sigma,Q)$ in Definition \ref
{def:MF_of_Scatter} may be viewed as a generalization of $L_\nu
(0,\Sigma,P)$ in \eqref{eq:def L_nu}. Here $\rho'(s) = (\nu+
q)/(\nu+ s)$ is strictly decreasing and $\psi(s) = (\nu+ q) s/(\nu+
s)$ is strictly increasing in $s \ge0$. Moreover, $\psi(0) = 0$ and
$\psi(\infty) = \nu+ q$.
\end{Example}

\begin{Example}[Multivariate elliptical Weibull-distributions]
For a fixed $\gamma> 0$ and $s \ge0$ let $\rho(s) := s^\gamma$.
Then $\rho'(s) = \gamma s^{\gamma- 1}$ and $\psi(s) := \gamma
s^\gamma$. Here $L_\rho(\Sigma,Q)$ corresponds to elliptically
symmetric distributions with center $0$ that are generated by $\tilde
{f}(s) := \exp(- s^\gamma/2)$. In this situation \eqref{eq:Finite
moments} means that
\[
\int\tr(M)^\gamma\, Q(dM) \ < \ \infty,
\]
and in setting~\eqref{eq:Example Q 1} this is equivalent to
\[
\int\|x\|^{2\gamma} \, P(dx) \ < \ \infty.
\]
\end{Example}

\begin{Example}
Another example, suggested to us by David Tyler, is given by
\[
\rho(s) \ := \ (\nu+ q) \log(1 + s^2)/2
\]
for $s \ge0$ with some parameter $\nu> 0$. Here $\rho'(s) = (\nu+
q) s / (1 + s^2)$, and $\psi(s) = (\nu+ q) s^2/(1 + s^2)$ is strictly
increasing in $s \ge0$ with $\psi(0) = 0$ and $\psi(\infty) = \nu+ q$.
\end{Example}

\paragraph{Existence of $L_\rho$}
The functional $L_\rho(\cdot,P) : \Rqqsympd\to\R$ is well-defined
in Cases~0 and 1.
This will be derived from the following two elementary inequalities
which will be used several times:

\begin{Lemma}
\label{lem:Trace inequalities}
For $M \in\Rqqsympsd$ and $A \in\Rqqsym$,
\[
\lambda_{\rm min}(A) \tr(M) \ \le\ \tr(AM) \ \le\ \lambda_{\rm
max}(A) \tr(M).
\]
\end{Lemma}

\begin{Lemma}
\label{lem:Expansion of rho}
For arbitrary $s, t > 0$,
\[
\psi(s) \log(t/s) \ \le\ \rho(t) - \rho(s) \ \le\ \psi(t) \log
(t/s).
\]
If $\rho'$ is non-increasing on $(0,\infty)$, then
\[
\rho'(t) (t - s) \ \le\ \rho(t) - \rho(s) \ \le\ \rho'(s) (t - s).
\]
\end{Lemma}

It follows from Lemma~\ref{lem:Trace inequalities} that for arbitrary
$M \in\Rqqsympsd$ and $\Sigma\in\Rqqsympd$,
\[
\lambda_{\rm max}(\Sigma)^{-1} \tr(M)
\ \le\ \tr(\Sigma^{-1} M)
\ \le\ \lambda_{\rm min}(\Sigma)^{-1} \tr(M).
\]
Combining these inequalities in case of $M \ne0$ with Lemma~\ref
{lem:Expansion of rho}, applied to $\{s,t\} = \bigl\{ \tr(M),\tr
(\Sigma^{-1} M) \bigr\}$, yields the inequality
\[
\bigl| \rho(\tr(\Sigma^{-1}M)) - \rho(\tr(M)) \bigr|
\ \le\ \psi\bigl( \lambda_*(\Sigma) \tr(M) \bigr) \log(\lambda
_*(\Sigma))
\]
with $\lambda_*(\Sigma) = \max\bigl\{ \lambda_{\rm min}(\Sigma
)^{-1}, \lambda_{\rm max}(\Sigma) \bigr\}$, and the right hand side
is integrable with respect to $Q$ by assumption \eqref{eq:Finite moments}.

\paragraph{Linear equivariance}
For a nonsingular matrix $B \in\Rqq$ let
\[
Q^B \ := \ \LL(BSB^\top)
\quad\text{and}\quad
Q_B \ := \ \LL(B^{-1}SB^{-\top})
\quad\text{with} \ \ S \sim Q,
\]
where $B^{-\top} := (B^{-1})^\top= (B^\top)^{-1}$. Then one can
easily verify that for arbitrary $\Sigma\in\Rqqsympd$,
\begin{align}
\nonumber
L_\rho(B\Sigma B^\top, Q^B) - L_\rho(BB^\top, Q^B) \
&= \ L_\rho(\Sigma, Q), \\
\label{eq:equivariance L_rho}
L_\rho(B\Sigma B^\top, Q) - L_\rho(BB^\top, Q) \
&= \ L_\rho(\Sigma, Q_B).
\end{align}
Let $\QQ_\rho$ denote the set of all distributions $Q$ as described
in Cases~0 and 1 such that $L_\rho(\cdot,Q)$ has a unique minimizer in
\[
\begin{cases}
\bigl\{ \Sigma\in\Rqqsympd: \det(\Sigma) = 1 \bigr\} & \text{in
Case~0}, \\
\Rqqsympd& \text{in Case~1}.
\end{cases}
\]
This minimizer is denoted by $\bSigma_\rho(Q)$. Then $\QQ_\rho$ is
linear invariant and $\bSigma_\rho$ is linear equivariant in the
sense that $Q^B \in\QQ_\rho$ and
\[
\bSigma(Q^B) \ = \
\begin{cases}
\det(BB^\top)^{-1/q} B \bSigma_\rho(Q) B^\top& \text{in Case~0} \\
B \bSigma_\rho(Q) B^\top& \text{in Case~1}
\end{cases}
\]
for all $Q \in\QQ_\rho$ and $B \in\Rqqns$.

\subsection{Existence and uniqueness of an optimizer}

The question of existence and uniqueness of minimizers of $L_\rho
(\cdot,Q)$ is closely related to the mass which $Q$ puts on special
linear subspaces of $\Rqqsym$. We define
\[
\VV_q \ := \ \{\V: \V\ \text{is a linear subspace of} \ \R^q\}.
\]
Then for $\V\in\VV_q$, we consider
\[
\M(\V) \ := \ \bigl\{ M \in\Rqqsympsd: M \R^q \subset\V\bigr\},
\]
a linear subspace of $\Rqqsym$ with dimension $\dim(\M(\V)) = \dim
(\V)(\dim(\V)+1)/2$. Another object of interest is the matrix
\begin{align*}
\Psi_\rho(\Sigma,Q) \
&:= \ \int\rho'(\tr(\Sigma^{-1} M)) M \, Q(dM) \\
&= \ \int\psi(\tr(\Sigma^{-1} M)) \tr(\Sigma^{-1} M)^{-1} M \,
Q(dM),
\end{align*}
where the integrands are interpreted as $0 \in\Rqq$ if $M = 0$.
It will turn out that the following conditions play the key role for the
existence of a unique minimizer~$\Sigma_\rho(Q)$.

\paragraph{Condition~0}
We assume that
\begin{equation}
\label{eq:Uniqueness Case 0}
Q(\M(\V)) \ < \ \frac{\dim(\V)}{q}
\quad\text{for all} \ \V\in\VV_q \ \text{with} \ 1 \le\dim(\V)
< q.
\end{equation}

\paragraph{Condition~1}
We assume that
\begin{equation}
\label{eq:Uniqueness Case 1}
Q(\M(\V)) \ < \ \frac{\psi(\infty) - q + \dim(\V)}{\psi(\infty)}
\quad\text{for all} \ \V\in\VV_q \ \text{with} \ 0 \le\dim(\V)
< q.
\end{equation}
In case of $\psi(\infty) = \infty$ the fraction on the right hand
side of \eqref{eq:Uniqueness Case 1} is interpreted as $1$.

\begin{Theorem}
\label{thm:Uniqueness}
A matrix $\Sigma\in\Rqqsympd$ minimizes $L_\rho(\cdot,Q)$ if, and
only if,
\begin{equation}
\label{eq:Fixed point}
\Psi_\rho(\Sigma,Q) \ = \ \Sigma.
\end{equation}

In Case~0, $L_\rho(\cdot,Q)$ possesses a unique minimizer with
determinant $1$ if, and only if, Condition~0 is satisfied.

In Case~1, $L_\rho(\cdot,P)$ possesses a unique minimizer if, and
only if, Condition~1 is satisfied.
\end{Theorem}

Our proof of Theorem~\ref{thm:Uniqueness} is based on an in-depth
analysis of the mapping $L_\rho(\cdot, Q)$ in Section~\ref
{sec:AnalysisL}. In particular it will turn out that the fixed-point
equation \eqref{eq:Fixed point} is equivalent to $L_\rho(\cdot,Q)$
having gradient $0$ at $\Sigma$. With Theorem~\ref{thm:Uniqueness} at
hand we may redefine the family $\QQ_\rho$ as follows:

\smallskip
\noindent
In Case~0, $\QQ_\rho$ consists of all probability distributions $Q$
on $\Rqqsympsd$ satisfying Condition~0 and $Q(\{0\}) = 0$.

\smallskip
\noindent
In Case~1, $\QQ_\rho$ consists of all probability distributions $Q$
on $\Rqqsympsd$ satisfying Condition~1 and $\int\psi(\lambda\tr
(M)) \, Q(dM) < \infty$ for any $\lambda\ge1$.\vadjust{\goodbreak}

Let us comment now on these conditions in two special settings.

\paragraph{The setting \eqref{eq:Example Q 1}}
If $Q = Q^1(P) = \LL(XX^\top)$ with a random vector $X \sim P$, then
$Q(\{0\}) = P(\{0\})$, and $\int\psi(\lambda\tr(M)) \, Q(dM) = \int
\psi(\lambda\|x\|^2) \, P(dx)$. Moreover,
\[
Q(\M(\V)) \ = \ P(\V).
\]
Hence Conditions~0 and 1 coincide with the known conditions from the
literature on $M$-estimation of scatter. In particular, a unique
minimizer $\bSigma_\rho(Q)$ is well-defined if $P$ is smooth in the
sense that
\begin{equation}
\label{eq:PV=0}
P(\V) \ = \ 0
\quad\text{for any} \ \V\in\VV_q \ \text{with} \ \dim(\V) < q
\end{equation}
and satisfies $\int\psi(\lambda\|x\|^2) \, P(dx) < \infty$ for any
$\lambda\ge1$.

Now consider the empirical distribution
\[
\hat{Q}^1 \ := \ n^{-1} \sum_{i=1}^n \delta_{X_i^{}X_i^\top}^{}
\]
with $n \ge q$ independent random vectors $X_1, X_2, \ldots, X_n \sim
P$. This is an unbiased estimator of $Q^1(P)$. In Section~\ref{sec:Proofs}
we will apply Theorem~\ref{thm:Uniqueness} to $\hat{Q}^1$ and prove
the following result:

\begin{Lemma}
\label{lem:Example Q 1}
Suppose that $P$ is smooth in the sense of \eqref{eq:PV=0}. Then
$\bSigma(\hat{Q}^1)$ is well-defined with probability one, provided that
\[
n \ \ge\
\begin{cases}
q+1 & \text{in Case 0}, \\
q & \text{in Case 1}.
\end{cases}
\]
\end{Lemma}

This result is based on the fact that in case of \eqref{eq:PV=0}, $q$
independent random vectors with distribution $P$ are linearly
independent almost surely.

\paragraph{The setting \eqref{eq:Example Q 2}}
Let $Q = Q^k(P) = \LL\bigl( S(X_1,X_2,\ldots,X_k)\bigr)$ with $k
\ge2$ independent random vectors $X_1, X_2, \ldots, X_k \sim P$. Here
$Q(\{0\}) = 0$ if, and only if, $P$ has no atoms, i.e.
\[
P(\{x\}) \ = \ 0
\quad\text{for all} \ x \in\R^q.
\]
Note also that $\tr(S(X_1,X_2,\ldots,X_k)) \le(k-1)^{-1} \sum
_{i=1}^k \|X_i\|^2$, so
\begin{align*}
\psi\bigl( \lambda\tr(S(X_1,X_2,\ldots,X_k)) \bigr) \
&\le\ \psi\Bigl( \lambda(1-1/k)^{-1} \max_{1\le i\le k} \|X_i\|^2
\Bigr) \\
&\le\ \sum_{i=1}^k \psi\bigl( \lambda(1 - 1/k)^{-1} \|X_i\|^2
\bigr)
\end{align*}
and
\begin{equation}
\label{eq:From.P.to.Q}
\int\psi(\lambda\tr(M)) \, Q(dM)
\ \le\ k \int\psi\bigl( \lambda(1 - 1/k)^{-1} \|x\|^2 \bigr) \,
P(dx).
\end{equation}
Moreover, according to Lemma~\ref{lem:Sample covariances} in
Section~\ref{sec:Proofs},
\[
S(X_1, X_2, \ldots, X_k) \, \R^q
\ = \ \mathrm{span}(X_2 - X_1, \ldots, X_k - X_1).
\]
Hence
\begin{align*}
Q(\M(\V)) \
&= \ \Pr\bigl( \mathrm{span}(X_2-X_1, \ldots, X_k-X_1) \subset\V
\bigr) \\
&= \ \Pr(X_2-X_1, \ldots, X_k-X_1 \in\V) \\
&= \ \int P(x + \V)^{k-1} \, P(dx) \\
&= \ \sum_{w \in\V^\perp} P(w + \V)^k.
\end{align*}
In particular, $\bSigma_\rho(Q)$ is well-defined if $P$ is smooth in
the sense that
\begin{equation}
\label{eq:PH=0}
P(H) \ = \ 0
\quad\text{for any hyperplane} \ H \subset\R^q,
\end{equation}
and if $\int\psi(\lambda\|x\|^2) \, P(dx) < \infty$ for arbitrary
$\lambda\ge1$. (A hyperplane is a set of the form $w + \V$ with $w
\in\R^q$, $\V\in\VV_q$, $\dim(\V) = q-1$.)

Now consider the empirical distribution
\[
\hat{Q}^k \ := \ {\binom{n}{k}}^{-1} \sum_{1 \le i_1 < \cdots< i_k
\le n}
\delta_{S(X_{i_1}, \ldots, X_{i_k})}^{}
\]
for some $k \ge2$ and $n \ge k$ independent random vectors $X_1, X_2,
\ldots, X_n \sim P$. Note that $\hat{Q}^k$ is an unbiased estimator
of $Q^k(P)$. In Section~\ref{sec:Proofs} we'll prove the following result:

\begin{Lemma}
\label{lem:Example Q 2}
Suppose that $P$ is smooth in the sense of \eqref{eq:PH=0}. Then
$\bSigma(\hat{Q}^k)$ is well-defined almost surely, provided that $n
\ge q+1$.
\end{Lemma}

\paragraph{Estimation of proportional covariance matrices}
As in Section~\ref{subsec:Tyler 1987 and more} consider
\[
\hat{Q} \ = \ \sum_{i=1}^K \frac{m_i}{m_+} \, \delta_{S_i}
\]
with independent random matrices $S_i \sim\mathcal{W}_q(c_i \Sigma,
m_i)$. Let $S_i = c_i \sum_{j=1}^{m_i} Y_{ij}^{} Y_{ij}^\top$ with
independent random vectors $Y_{ij} \sim\NN_q(0,\Sigma)$, $1 \le i
\le K$, $1 \le j \le m_i$. Then one can easily show that
\[
S_i \, \R^q \ = \ \mathrm{span}(Y_{ij} : 1 \le j \le m_i).
\]
Thus with similar arguments as in the proof of Lemma~\ref{lem:Example
Q 1} one can show that with probability one,
\[
\hat{Q}(\M(\V))
\ \le\ \frac{1}{m_+} \sum_{i=1}^K \sum_{j=1}^{m_i} 1_{[Y_{ij} \in
\V]}
\ \le\ \frac{\dim(\V)}{m_+}
\]
for arbitrary $\V\in\VV_q$ with $\dim(\V) < q$. Hence $\bSigma
_\rho(\hat{Q})$ is well-defined in Case~0 almost surely, provided that
\[
m_+ \ \ge\ q + 1.
\]

\subsection{A fixed-point algorithm}

Suppose that $\rho$ satisfies the additional constraint that $\rho'$
is non-increasing on $(0,\infty)$. In this case one can use the
fixed-point equation \eqref{eq:Fixed point} to calculate $\bSigma
_\rho(Q)$ numerically. Recall that $\Sigma_* \in\Rqqsympd$
minimizes $L_\rho(\cdot,Q)$ if, and only if, $\Psi_\rho(\Sigma
_*,Q) = \Sigma_*$, according to Theorem~\ref{thm:Uniqueness}. This
fixed-point equation implies that
\[
\Psi_\rho(\Sigma,Q) \ \in\ \Rqqsympd
\quad\text{for arbitrary} \ \Sigma\in\Rqqsympd.
\]
For otherwise we could find a vector $v \in\R^q \setminus\{0\}$ such that
\[
0 \ = \ v^\top\Psi_\rho(\Sigma,Q) v
\ = \ \int\rho'(\tr(\Sigma^{-1}M)) v^\top M v \, Q(dM).
\]
But then $v^\top M v = 0$ for almost all $M$ w.r.t.\ $Q$, i.e.\
$Q(\mathbb{M}(v^\top)) = 1$. This would yield the contradiction $0 <
v^\top\Sigma_* v = v^\top\Psi_\rho(\Sigma_*,Q) v = 0$. It would
also contradict Condition~0 and 1.

Iterating the mapping $\Psi_\rho(\cdot,Q)$ yields a sequence
converging to a positive multiple of $\bSigma_\rho(Q)$ in Case~0 and
to $\bSigma_\rho(Q)$ in Case~1:

\begin{Lemma}[Convergence of a fixed-point algorithm]
\label{lem:Fixed-point algorithm}
Suppose that $Q$ fulfills Condition~0 in Case~0 and Condition~1 in
Case~1, and let $\rho'$ be non-increasing on $(0,\infty)$. For any
starting point $\Sigma_0 \in\Rqqsympd$, define inductively
\[
\Sigma_k \ := \ \Psi_\rho(\Sigma_{k-1},Q)
\]
for $k = 1, 2, 3, \dots$. Then the sequence $(\Sigma_k)_{k \ge0}$
converges to a solution of the fixed-point equation \eqref{eq:Fixed point}.
\end{Lemma}

A key ingredient for proving this lemma is the following inequality. It
may be viewed as a special case of a wellkown inequality for the EM
algorithm by Dempster et al.\ (\citeyear{Dempster_etal_1977}). For the
precise connection between variations of the present fixed-point
algorithm and the EM algorithm we refer to Arslan et al.\ (\citeyear
{Arslan_etal_1995}) and Arslan and Kent (\citeyear{Arslan_Kent_1998}).

\begin{Lemma}
\label{lem:PsiQ.better.than.I}
Suppose that $\rho'$ is non-increasing on $(0,\infty)$. Let $Q$ be a
probability distribution on $\Rqqsympsd$ such that $Q(\mathbb
{M}(v^\top)) < 1$ for any $v \in\R^q \setminus\{0\}$ and $\int\psi
(\tr(M)) \, Q(dM) < \infty$. Then for any $\Sigma\in\Rqqsympd$,
\[
L_\rho(\Psi_\rho(\Sigma,Q),Q) \ < \ L_\rho(\Sigma,Q)
\]
unless $\Psi_\rho(\Sigma,Q) = \Sigma$.
\end{Lemma}


\section{Analytical properties of the criterion function}
\label{sec:AnalysisL}

The results in the previous section can be derived from an in-depth
analysis of the function $L_\rho(\cdot,Q)$. As mentioned in the
introduction, we utilize matrix exponentials which are reviewed in the
next subsection. Then we derive differentiability, a convexity property
and coercivity of $L_\rho(\cdot,Q)$ under certain conditions. In the
last subsection we derive second order Taylor expansions of $L_\rho
(\cdot,Q)$ which are needed later on.

\subsection{The exponential transform of matrices}

\paragraph{The exponential transform on $\Rqq$}
For an arbitrary matrix $A \in\Rqq$, its exponential transform
\[
\exp(A) \ := \ \sum_{k=0}^\infty\frac{A^k}{k!}
\]
is well-defined in $\Rqq$, satisfying the inequalities $\|\exp(A)\|
\le e^{\|A\|}$ and
\[
\Bigl\| \sum_{k=\ell}^\infty\frac{A^k}{k!} \Bigr\|
\ \le\ e_{}^{\|A\|} \|A\|^\ell/ \ell!
\quad\text{for} \ \ell\ge1.
\]
If $A, B \in\Rqq$ are interchangeable in the sense that $AB = BA$,
the familiar equation $\exp(A+B) = \exp(A)\exp(B) = \exp(B)\exp
(A)$ is valid. In particular, $\exp(A)$ is always nonsingular with inverse
\[
\exp(A)^{-1} \ = \ \exp(-A).
\]

In general the expansion of $\exp(A + B)$ is somewhat more
complicated. From the following result only the very first inequality
is needed later, but the full result may be of interest for curious
readers and illustrates why treating $L_\rho(\Sigma,Q)$ as a function
of $\log(\Sigma)$ is not that straightforward.

\begin{Lemma}[Taylor expansions of $\exp(\cdot)$]
\label{lem:Taylor of exp}
For matrices $A, \Delta\in\Rqq$,
\begin{align*}
\exp(A + \Delta) \
&= \ \exp(A) + R_0(A,\Delta) \\
&= \ \exp(A) + \int_0^1 \exp((1-u)A) \Delta\exp(uA) \, du +
R_1(A,\Delta)
\end{align*}
with
\[
\|R_0(A,\Delta)\| \ \le\ e^{\|A\| + \|\Delta\|} \|\Delta\|
\quad\text{and}\quad
\|R_1(A,\Delta)\| \ \le\ e^{\|A\| + \|\Delta\|} \|\Delta\|^2/2.
\]
Moreover,
\[
\exp(A + \Delta) \ = \ \exp(A) +
\sum_{k=1}^\infty\frac{1}{k!} \, \Ex\bigl[
\exp(U_{k0}A) \Delta\exp(U_{k1}A) \cdots\Delta\exp(U_{kk}A) \bigr
],
\]
where $U_{k0} = 1 - \sum_{j=1}^k U_{kj}$, and $(U_{kj})_{j=1}^k$ is
uniformly distributed on the convex polytope $\bigl\{ u \in[0,1]^k :
\sum_{j=1}^k u_j \le1 \bigr\}$.
\end{Lemma}

\paragraph{The exponential transform on $\Rqqsym$}
Any matrix $A \in\Rqqsym$ may be written as
\[
A \ = \ \sum_{i=1}^q \lambda_i(A) \, u_i^{} u_i^\top
\ = \ U \diag\bigl( (\lambda_i(A))_{i=1}^q \bigr) U^\top
\]
with the ordered eigenvalues $\lambda_i(A)$ of $A$, an orthonormal
basis $u_1, u_2, \ldots, u_q$ of corresponding eigenvectors, and the
orthogonal matrix $U = [u_1\, u_2\, \ldots\, u_q]$. Then one can
easily verify that
\[
\exp(A) \ = \ \sum_{i=1}^q \exp(\lambda_i(A)) \, u_i^{} u_i^\top
\ \in\ \Rqqsympd.
\]
As a mapping from $\Rqqsym$ to $\Rqqsympd$, the exponential function
is bijective with inverse
\[
\log(A) \ := \ \sum_{i=1}^q \log(\lambda_i(A)) \, u_i^{}u_i^\top.
\]
Moreover,
\[
\det(\exp(A)) \ = \ \exp(\tr(A)).
\]

\paragraph{Local parametrizations of $\Rqqsympd$}
Unfortunately, for $\Sigma, \Sigma' \in\Rqqsympd$, the equation
$\Sigma' = \exp(\log(\Sigma) + A)$ with $A := \log(\Sigma') -
\log(\Sigma)$ is not very helpful, because the Taylor expansion of
$\exp(\log(\Sigma) + A)$ is somewhat awkward, unless $\log(\Sigma
)$ and $A$ are interchangeable. In view of our considerations on linear
equivariance, we consider a different approach: Let $\Sigma\in
\Rqqsympd$, and fix an arbitrary $B \in\Rqqns$ such that
\[
\Sigma\ = \ BB^\top,
\]
e.g.\ $B = \Sigma^{1/2}$. Then
\[
\Rqqsympd\ = \ \bigl\{ B \exp(A) B^\top: A \in\Rqqsym\bigr\}.
\]
Indeed, any matrix $\Sigma' \in\Rqqsympd$ may be written as $B \exp
(A) B^\top$ with $A := \log(B^{-1} \Sigma' B^{-\top})$. Note that
the matrix $A$ depends on both $B$ and $\Sigma'$, but its eigenvalues
are simply $\lambda_i(A) = \log\lambda_i(\Sigma^{-1} \Sigma')$.
Moreover, if $\det(\Sigma) = 1$, then
\[
\bigl\{ \Sigma' \in\Rqqsympd: \det(\Sigma') = 1 \bigr\}
\ = \ \bigl\{ B \exp(A) B^\top: A \in\Rqqsym, \tr(A) = 0 \bigr\}.
\]

\subsection{First-order smoothness of the criterion function}

We start with an expansion of $L_\rho(\cdot,Q)$ in small
neighborhoods of $I_q$. To this end we need the matrix
\[
\Psi_\rho(Q)
\ := \ \Psi_\rho(I_q,Q) = \int\rho'(\tr(M)) M \, Q(dM)
\ \in\ \Rqqsym.
\]

\begin{Proposition}[1st order Taylor expansion]
\label{prop:Expansion 1}
For $A \in\Rqqsym$,
\[
L_\rho(\exp(A),Q)
\ = \ \langle A, G_\rho(Q)\rangle+ R_\rho(A,Q)
\]
with the gradient
\[
G_\rho(Q) \ := \ I_q - \Psi_\rho(Q)
\ \in\ \Rqqsym
\]
and a remainder $R_\rho(A,Q)$ satisfying the following inequalities:
\begin{align*}
\bigl| \langle A, G_\rho(Q)\rangle\bigr| \
&\le\ (q + J_\rho(Q)) \|A\|, \\
|R_\rho(A,Q)| \
&\le\ \bigl( J_\rho(e^{\|A\|},Q) - J_\rho(e^{-\|A\|},Q) \bigr) \|
A\|
+ J_\rho(Q) \|A\|^2/2,
\end{align*}
where $J_\rho(Q) := J_\rho(1,Q)$ and
\[
J_\rho(\lambda,Q) \ := \ \int\psi(\lambda\tr(M)) \, Q(dM).
\]
\end{Proposition}

Note that $J_\rho(\cdot,Q) \equiv q$ in Case~0. In Case~1, $J_\rho
(\lambda,Q)$ is continuous and monotone increasing in $\lambda> 0$
with values in $[0, \psi(\infty))$. Thus in both cases,
\[
|R_\rho(A,Q)| \ = \ o(\|A\|) \quad\text{as} \ A \to0.
\]

Proposition~\ref{prop:Expansion 1} carries over to expansions in other
neighborhoods via linear equivariance: For any fixed $B \in\Rqqns$
and $A \in\Rqqsym$ we have by \eqref{eq:equivariance L_rho},
\begin{align*}
L_\rho(B \exp(A) B^\top, Q) - L_\rho(BB^\top, Q) \
&= \ L_\rho(\exp(A), Q_B) \\
&= \ \langle A, G_\rho(Q_B) \rangle+ R_\rho(A, Q_B),
\end{align*}
where
\begin{align*}
\bigl| \langle A, G_\rho(Q_B)\rangle\bigr| \
&\le\ (q + J_\rho(Q_B)) \|A\|, \\
|R_\rho(A,Q_B)| \
&\le\ \bigl( J_\rho(e^{\|A\|},Q_B) - J_\rho(e^{-\|A\|},Q_B) \bigr)
\|A\|
+ J_\rho(Q_B) e^{\|A\|} \|A\|^2/2.
\end{align*}
Moreover, with $\Sigma:= BB^\top$, Lemma~\ref{lem:Trace
inequalities} and monotonicity of $\psi$ yield
\begin{equation}
\label{ineq:JQ}
J_\rho(\lambda,Q_B)
\ = \ \int\psi(\lambda\tr(\Sigma^{-1} M)) \, Q(dM)
\ \le\ J_\rho(\lambda/\lambda_{\rm min}(\Sigma),Q).
\end{equation}
Note also that
\[
G_\rho(Q_B) \ = \ B^{-1} \bigl( \Sigma- \Psi_\rho(\Sigma,Q) \bigr
) B^{-\top},
\]
so the fixed-point equation \eqref{eq:Fixed point} in Theorem~\ref
{thm:Uniqueness} is satisfied if, and only if, $G_\rho(Q_B) = 0$.

Proposition~\ref{prop:Expansion 1} implies also that $L_\rho(\cdot
,Q)$ is a continuously differentiable and locally Lipschitz-continuous
function on $\Rqqsympd$ in the usual sense:

\begin{Corollary}[Smoothness]
\label{cor:Diff.and.Lipschitz}
The function $L_\rho(\cdot,Q)$ is continuously differentiable on
$\Rqqsympd$ with gradient
\begin{align*}
\nabla L_\rho(\Sigma,Q) \
&= \ \Sigma^{-1} - \int\rho'(\tr(\Sigma^{-1}M)) \Sigma^{-1} M
\Sigma^{-1} \, Q(dM) \\
&= \ B^{-1} G_\rho(Q_B) B^{-1}
\end{align*}
with $B := \Sigma^{1/2}$. Moreover, let $K$ be a convex subset of
$\Rqqsympd$ with $\lambda_{\rm min}(K) := \inf_{\Sigma\in K}
\lambda_{\rm min}(\Sigma) > 0$. Then for $\Sigma_0, \Sigma_1 \in K$,
\[
\bigl| L_\rho(\Sigma_1,Q) - L_\rho(\Sigma_0,Q) \bigr|
\ \le\ \bigl( q + J_\rho(\lambda_{\rm min}(K)^{-1},Q) \bigr)
\lambda_{\rm min}(K)^{-1} \|\Sigma_1 - \Sigma_0\|.
\]
\end{Corollary}

\subsection{Convexity and coercivity}

Theorem~\ref{thm:Uniqueness} follows essentially from the next two
results. The first one provides a surrogate for the simpler claim that
$L_\rho(\Sigma,Q)$ is a convex function of $\log(\Sigma)$. The
second one deals with the behavior of $L_\rho(\Sigma,Q)$ as $\|\log
(\Sigma)\| \to\infty$.

\begin{Proposition}[Convexity]
\label{prop:Convexity}
For any fixed $B \in\Rqqns$ and $A \in\Rqqsym$,
\[
\R\ni t \ \mapsto\ L_\rho(B \exp(tA) B^\top, Q)
\]
is a convex function. This convexity is strict if, and only if,
\[
\begin{cases}
Q \bigl( \bigcup_{i=1}^\ell\M(B\V_i) \bigr) \ < \ 1
& \text{in Case~0}, \\[0.5ex]
Q \bigl( \M(B \V_0) \bigr) \ < \ 1
& \text{in Case~1},
\end{cases}
\]
where $\V_1, \ldots, \V_\ell$ are the eigenspaces of $A$, and $\V
_0 := \{x \in\R^q : Ax = 0\}$.
\end{Proposition}

\begin{Proposition}[Coercivity]
\label{prop:Coercivity}
Let $B$ be an arbitrary fixed matrix in $\Rqqns$.  In Case~0,
\[
\lim_{\|A\| \to\infty, \, \tr(A) = 0} \, L_\rho(B\exp(A)B^\top,
Q) \ = \ \infty
\]
if, and only if, Condition~0 is true.  In Case~1,
\[
\lim_{\|A\| \to\infty} \, L_\rho(B\exp(A)B^\top, Q) \ = \ \infty
\]
if, and only if, Condition~1 is true.
\end{Proposition}

The convexity property in Proposition~\ref{prop:Convexity} is
sometimes called ``geodesic convexity'' (cf.\ Wiesel, \citeyear
{Wiesel_2012}). This name stems from the fact that for arbitrary
matrices $\Sigma_0 = BB^\top$ and $\Sigma_1 = B \exp(A) B^\top$ in
$\Rqqsympd$, the path
\[
[0,1] \ni t \ \mapsto\ \Gamma(t) := B \exp(tA) B^\top
\]
minimizes the ``length''
\[
\int_0^1 \bigl\| \Gamma(t)^{-1/2} \Gamma'(t) \Gamma(t)^{-1/2}
\bigr\|_F \, dt
\]
over all continuously differentiable functions $\Gamma: [0,1] \to
\Rqqsympd$ with $\Gamma(0) = \Sigma_0$ and $\Gamma(1) = \Sigma_1$;
see Bhatia (\citeyear{Bhatia_2007}, Chapter~6).

\subsection{Second-order smoothness of the criterion function}

In order to prove differentiability of $\bSigma_\rho(\cdot)$, we
need second order Taylor expansions of $L_\rho(\cdot,Q)$. These are
also useful to replace the fixed-point algorithm described earlier by
faster methods, see D\"umbgen et al.\ (\citeyear{Duembgen_etal_2013}).

From now on we assume that $\rho$ is twice continuously differentiable
on $(0,\infty)$. In addition to $\psi(s) = s \rho'(s)$ we consider
\[
\psi_2(s) \ := \ s \psi'(s) \ = \ \psi(s) + s^2 \rho''(s).
\]
In Case~0, $\psi\equiv q$, so $\psi' \equiv\psi_2 \equiv0$. Case~1
is modified as follows:

\paragraph{Case 1'}
We assume that $\psi' > 0$ and that $\psi$ has limits $\psi(0) = 0$
and $\psi(\infty) \in(q,\infty]$. Moreover we assume that
\begin{equation}
\label{eq:Finite moments'}
\int\psi(\tr(M)) \, Q(dM) \ < \ \infty
\end{equation}
and that there exists a constant $\kappa> 0$ such that
\begin{equation}
\label{ineq:psi2}
\psi_2(s) \ \le\ \kappa\psi(s)
\quad\text{for all} \ s > 0.
\end{equation}

\begin{Remark}
Inequality \eqref{ineq:psi2} is mainly for convenience and to avoid
additional integrability conditions for $\psi_2$. It also allows to
replace \eqref{eq:Finite moments} with the simpler condition \eqref
{eq:Finite moments'}, see Lemma~\ref{lem:psi.and.kappa} below.
It follows from $\psi_2(s) = \psi(s) + s^2 \rho''(s)$ and $\psi,
\psi' > 0$ that $s^2 \rho''(s) = \psi_2(s) - \psi(s) \in\bigl(
-\psi(s), \psi_2(s) \bigr)$. Hence inequality \eqref{ineq:psi2} is
equivalent to the existence of a constant $\tilde{\kappa}$ such that
\begin{equation}
\label{ineq:psi2'}
s^2 |\rho''(s)| \ \le\ \tilde{\kappa} \psi(s)
\quad\text{for all} \ s > 0.
\end{equation}
\end{Remark}

\begin{Remark}
Suppose that $\rho'$ is non-increasing, i.e.\ $\rho'' \le0$. Then $0
< \psi_2(s) \le\psi(s)$ and $-\psi(s) < s^2 \rho''(s) \le0$.
Hence \eqref{ineq:psi2} and \eqref{ineq:psi2'} are satisfied with
\mbox{$\kappa= \tilde{\kappa} = 1$}.
\end{Remark}

\begin{Example}[Multivariate elliptical Weibull-distributions]
In case of $\rho(s) := s^\gamma$ for a constant $\gamma> 0$, we have
$\psi(s) = \gamma s^\gamma$ and
\[
s^2 \rho''(s) \ = \ (\gamma- 1) \psi(s),
\quad
\psi_2(s) \ = \ \gamma\psi(s),
\]
so \eqref{ineq:psi2} and \eqref{ineq:psi2'} are satisfied with
$\kappa= \gamma$ and $\tilde{\kappa} = |\gamma- 1|$.
\end{Example}

\begin{Example}
In case of $\rho(s) := (\nu+ q) \log(1 + s^2)/2$ for a constant $\nu
> 0$, we have $\psi(s) = (\nu+ q) s^2/(1 + s^2)$ and
\[
s^2 \rho''(s) \ = \ \bigl( 1 - 2 \psi(s)/\psi(\infty) \bigr) \psi
(s),
\quad
\psi_2(s) \ = \ 2 \bigl( 1 - \psi(s)/\psi(\infty) \bigr) \psi(s),
\]
so \eqref{ineq:psi2} and \eqref{ineq:psi2'} are satisfied with
$\kappa= 2$ and $\tilde{\kappa} = 1$.
\end{Example}

\begin{Lemma}
\label{lem:psi.and.kappa}
Let $\phi: (0,\infty) \to(0,\infty)$ be a differentiable function.
For any $\kappa\in\R$ the following two statements are equivalent:
\begin{align}
\label{ineq:psi.and.kappa.1}
s \phi'(s) \
&\le\ \kappa\phi(s)
\quad\text{for all} \ s > 0; \\
\label{ineq:psi.and.kappa.2}
\phi(\lambda s) \
&\le\ \lambda^\kappa\phi(s)
\quad\text{for all} \ s > 0 \ \text{and} \ \lambda> 1.
\end{align}
\end{Lemma}

Now we are ready to extend the expansion of $L_\rho(\cdot,Q)$ around
$I_q$ from Proposition~\ref{prop:Expansion 1}:

\begin{Proposition}[2nd order Taylor expansion]
\label{prop:Expansion 2}
In Case~0 and Case~1', for arbitrary $A \in\Rqqsym$,
\begin{equation}
\label{basic expansion}
L_\rho(\exp(A),Q)
\ = \ \langle A, G_\rho(Q) \rangle+ 2^{-1} H_\rho(A,Q) + R_{\rho,2}(A,Q)\vadjust{\goodbreak}
\end{equation}
with the gradient $G_\rho(Q)$ as in Proposition~\ref{prop:Expansion
1}, the quadratic term
\begin{align*}
H_\rho(A,Q) \
&:= \ \int\bigl( \rho'(\tr(M)) \tr(A^2M) + \rho''(\tr(M)) \tr
(AM)^2 \bigr) \, Q(dM) \\
&= \ \tr(A^2 \Psi_\rho(Q)) + \int\rho''(\tr(M)) \tr(AM)^2 \, Q(dM)
\end{align*}
and a remainder term $R_{\rho,2}(A,Q)$ satisfying the following inequalities:
\begin{align}
\label{eq:Expansion 2a}
H_\rho(A,Q) \
&\in\ \bigl[ 0, (1 + \kappa) J_\rho(Q) \|A\|^2 \bigr], \\
\label{eq:Expansion 2b}
|R_{\rho,2}(A,Q)| \
&\le\ \Omega(\|A\|,Q) \|A\|^2/2
+ (\kappa+ 1/7) J_\rho(Q) \|A\|^3
\end{align}
with
\[
\Omega(\delta,Q) \ := \ \int\sup_{z \in[-\delta,\delta]}
\bigl| \psi_2(e^z \tr(M)) - \psi_2(\tr(M)) \bigr| \, Q(dM).
\]
Moreover,
\begin{equation}
\label{eq:Expansion 2c}
H_\rho(A,Q) \ > \ 0 \quad\text{if} \
\begin{cases}
Q \bigl( \bigcup_{i=1}^\ell\M(\V_i) \bigr) \ < \ 1
& \text{in Case 0}, \\
Q(\M(\V_0)) \ < \ 1
& \text{in Case 1'},
\end{cases}
\end{equation}
where $\V_1,\ldots,\V_\ell$ are the eigenspaces of $A$, and $\V_0
:= \{x \in\R^q : Ax = 0\}$.
\end{Proposition}

Note that $\Omega(\delta,Q)$ is continuous in $\delta\ge0$ with
$\Omega(0,Q) = 0$. This follows from the fact that
\[
\sup_{z \in[-\delta,\delta]}
\bigl| \psi_2(e^z \tr(M)) - \psi_2(\tr(M)) \bigr|
\]
is continuous in $\delta\ge0$ and not greater than $\kappa\psi
(e^\delta\tr(M)) \le\kappa e^{\kappa\delta} \psi(\tr(M))$. In particular,
\[
R_{\rho,2}(A,Q) \ = \ o(\|A\|^2) \quad\text{as} \ A \to0.
\]

Again Proposition~\ref{prop:Expansion 2} carries over to expansions in
other neighborhoods via linear equivariance: For any fixed $B \in
\Rqqns$ and $A \in\Rqqsym$,
\begin{align*}
L_\rho(&B \exp(A) B^\top, Q) - L_\rho(BB^\top, Q) \\
&= \ L_\rho(\exp(A), Q_B)
\ = \ \langle A, G_\rho(Q_B) \rangle+ 2^{-1} H_\rho(A,Q_B) + R_{\rho
,2}(A, Q_B),
\end{align*}
where $R_{\rho,2}(A,Q_B) = o(\|A\|^2)$ as $A \to0$.

\paragraph{The Hessian operator}
The quadratic term $H_\rho(A,Q)$ in Proposition~\ref{prop:Expansion
2} may be written as
\[
H_\rho(A,Q) \ = \ \langle A, H_\rho(Q) A \rangle
\]
with the linear operator $H_\rho(Q) : \Rqqsym\to\Rqqsym$ given by
\begin{align*}
H_\rho(Q) A \
&:= \ \int\bigl( \rho'(\tr(M)) 2^{-1}(AM + MA)
+ \rho''(\tr(M)) \tr(AM) M \bigr) \, Q(dM) \\
&= \ 2_{}^{-1} \bigl( A \Psi_\rho(Q) + \Psi_\rho(Q) A \bigr)
+ \int\rho''(\tr(M)) \tr(AM) M \, Q(dM).
\end{align*}
This operator is self-adjoint, that means, $\langle A, H_\rho(Q) B
\rangle= \langle B, H_\rho(Q)A \rangle$ for arbitrary $A,B \in
\Rqqsym$.

\paragraph{Invertibility in Case~1'}
Under Condition~1 it follows from the last part of Proposition~\ref
{prop:Expansion 2} that $H_\rho(Q)$ is positive definite and thus invertible.

\paragraph{Invertibility in Case~0}
The gradient $G_\rho(Q) = I_q - q \int\tr(M)^{-1} M \, Q(dM)$ is
contained in the linear subspace
\[
\W_0 \ := \ \{A \in\Rqqsym: \tr(A) = 0\},
\]
and for any $A \in\Rqqsym$,
\[
H_\rho(Q) A
\ = \ q \int\bigl( \tr(M)^{-1} 2^{-1}(AM + MA)
- \tr(M)^{-2} \tr(AM) M \bigr) \, Q(dM)
\]
belongs to $\W_0$, too. Hence we view $H_\rho(Q)$ as a linear
operator from $\W_0$ to $\W_0$. Under Condition~0, the last part of
Proposition~\ref{prop:Expansion 2} implies that this operator is
positive definite und thus invertible.


\section{Continuity, consistency and differentiability}
\label{sec:Scatter 2}

In this section we derive various properties of $\bSigma_\rho(\cdot
)$ and related limit theorems. The arguments we use are adaptations of
standard arguments in the statistical literature, e.g.\ the monographs
mentioned in the introduction. Related are also the papers by Haberman
(\citeyear{Haberman_1989}) and Niemiro (\citeyear{Niemiro_1992})
about $M$-estimation with convex criterion functions.

Throughout this section let $Q$ be a distribution in $\QQ_\rho$ and define
\[
\mathbb{Y} \ := \
\begin{cases}
\Rqqsympsd\setminus\{0\} & \text{in Case~0}, \\
\Rqqsympsd & \text{in Case~1}.
\end{cases}
\]
Moreover we consider the linear space
\[
\W\ := \
\begin{cases}
\{A \in\Rqqsym: \tr(A) = 0\} & \text{in Case~0}, \\
\Rqqsym & \text{in Case~1}.
\end{cases}
\]
Recall that in Case~1', $H_\rho(Q) : \W\to\W$ is an invertible
linear operator.

Unless stated otherwise, all subsequent asymptotic statements refer to
the sequence index $n$ tending to $\infty$. Furthermore, ``$\to_p$''
and ``$\to_w$'' stand for convergence in probability and weak
convergence, respectively.

\subsection{Continuity}
\label{subsec:Continuity}

Our first result establishes a certain continuity property of $\bSigma
_\rho(\cdot)$.

\begin{Theorem}[Continuity I]
\label{thm:Continuity}
Let $(Q_n)_n$ be a sequence of probability distributions on $\mathbb
{Y}$ converging weakly to $Q$. In Case~1 suppose in addition that all
$Q_n$ satisfy \eqref{eq:Finite moments} and that
\begin{equation}
\label{eq:Continuity}
\int\psi\bigl( \lambda_o \tr(\bSigma_\rho(Q)^{-1} M) \bigr) \, Q_n(dM)
\ \to\ \int\psi\bigl( \lambda_o \tr(\bSigma_\rho(Q)^{-1} M)
\bigr) \, Q(dM)
\end{equation}
for some $\lambda_o > 1$. Then $Q_n \in\QQ_\rho$ for sufficiently
large $n$, and
\[
\bSigma_\rho(Q_n) \ \to\ \bSigma(Q).
\]
\end{Theorem}

\begin{Remark}[Weak Continuity]
In case of $\psi(\infty) < \infty$, Condition~\eqref{eq:Continuity}
is satisfied for any $\Sigma_o \in\Rqqsympd$ because $Q_n \to_w Q$.
Thus Theorem~\ref{thm:Continuity} shows that the set $\QQ_\rho$ is
open in the topology of weak convergence of probability measures on
$\mathbb{Y}$, and that the functional $\bSigma_\rho$ is weakly
continuous on $\QQ_\rho$.
\end{Remark}

Our proof of Theorem~\ref{thm:Continuity} covers also the situation of
random distributions $\hat{Q}_n$ in place of $Q_n$. Indeed the
following result is true:

\begin{Theorem}[Continuity II]
\label{thm:Consistency}
Let $\hat{Q}_1, \hat{Q}_2, \hat{Q}_3, \ldots$ be random
distributions on $\mathbb{Y}$ such that for any bounded and continuous
function $f : \mathbb{Y} \to\R$,
\begin{equation}
\label{eq:Consistency 1}
\int f \, d\hat{Q}_n
\ \to_p \ \int f \, dQ.
\end{equation}
In Case~1 suppose further that $\hat{Q}_n$ satisfies \eqref{eq:Finite
moments} almost surely and that
\begin{equation}
\label{eq:Consistency 2}
\int\psi\bigl( \lambda_o \tr(\bSigma_\rho(Q)^{-1} M) \bigr) \,
\hat{Q}_n(dM)
\ \to_p \ \int\psi\bigl( \lambda_o \tr(\bSigma_\rho(Q)^{-1} M)
\bigr) \, Q(dM)
\end{equation}
for some $\lambda_o > 1$. Then $\Pr(\hat{Q}_n \in\QQ_\rho) \to1$ and
\[
\bSigma_\rho(\hat{Q}_n) \ \to_p \ \bSigma_\rho(Q).
\]
\end{Theorem}

In case of $\psi(\infty) < \infty$, one could derive Theorem~\ref
{thm:Consistency} easily from Theorem~\ref{thm:Continuity} by means of
metrics for weak convergence as described in by Dudley (\citeyear
{Dudley_2002}, Section~11.3). In the general setting, however, it is
easier to prove Theorem~\ref{thm:Consistency} directly and realize
that Theorem~\ref{thm:Continuity} is just a special case of it.

\subsection{Differentiability}
\label{subsec:Diffability}

In this subsection we refine Theorem~\ref{thm:Consistency} with an
asymptotic linear expansion of $\bSigma_\rho(\cdot)$ in Cases~0 and
1'. By linear equivariance it suffices to consider the case
\[
\bSigma_\rho(Q) \ = \ I_q.
\]

\begin{Theorem}[Differentiability]
\label{thm:Diffability}
Let $\hat{Q}_1$, $\hat{Q}_2$, $\hat{Q}_3$, \ldots be random
distributions on $\mathbb{Y}$ satisfying Condition~\eqref
{eq:Consistency 1}. In Case~1' suppose further that for all $n$, $\int
\psi(\tr(M)) \, \hat{Q}_n(dM) < \infty$ almost surely, and
\begin{equation}
\label{eq:Consistency 2'}
\int\psi(\tr(M)) \, \hat{Q}_n(dM)
\ \to_p \ \int\psi(\tr(M)) \, Q(dM).
\end{equation}
Then in Cases~0 and 1',
\[
G_\rho(\hat{Q}_n) \ \to_p \ 0,
\quad\text{and}\quad
H_\rho(\hat{Q}_n) \ \to_p \ H_\rho(Q).
\]
Moreover, $\Pr(\hat{Q}_n \in\QQ_\rho) \to1$ and
\begin{equation}
\label{eq:Linear approximation}
\log(\bSigma_\rho(\hat{Q}_n))
\ = \ - H_\rho(Q)^{-1}G_\rho(\hat{Q}_n)
+ o_p \bigl( \|G_\rho(\hat{Q}_n)\| \bigr).
\end{equation}
\end{Theorem}

\begin{Remark}
Condition~\eqref{eq:Consistency 2'} seems to be weaker than \eqref
{eq:Consistency 2} at first glance. But in Case~1',
\[
\psi\bigl( \lambda_o \tr(M) \bigr)
\ \le\ \lambda_o^\kappa\psi(\tr(M))
\]
for any $\lambda_o > 1$ and $M \in\mathbb{Y}$ by Lemma~\ref
{lem:psi.and.kappa}. Consequently \eqref{eq:Consistency 2} follows
from \eqref{eq:Consistency 1} and \eqref{eq:Consistency 2'} by virtue
of Lemma~\ref{lem:Mallows convergence} in Section~\ref{sec:Proofs}.
\end{Remark}

\begin{Remark}\label{rem: equiv.exp}
Note that the asymptotic expansion \eqref{eq:Linear approximation} is
equivalent to the expansion
\[
\bSigma_\rho(\hat{Q}_n) \ = \ I_q - H_\rho(Q)^{-1}G_\rho(\hat{Q}_n)
+ o_p \bigl( \|G_\rho(\hat{Q}_n)\| \bigr).
\]
\end{Remark}

\begin{Remark}[Weak Differentiability]
In Cases~0 and 1' with $\psi(\infty) < \infty$, Theorem~\ref
{thm:Diffability} shows that the functional $\bSigma_\rho$ is weakly
differentiable on $\QQ_\rho$ in the following sense: Let $Q \in\QQ
_\rho$ and $B := \bSigma_\rho(Q)^{1/2}$. Further let $(Q_n)_n$ be a
sequence of probability distributions in $\QQ_\rho$ converging weakly
to $Q$. Then $G_\rho((Q_n)_B) \to0$ and
\[
\log(B^{-1} \bSigma_\rho(Q_n) B^{-1})
\ = \ - H_\rho(Q_B)^{-1} G_\rho((Q_n)_B)
+ o \bigl( \bigl\| G_\rho((Q_n)_B) \bigr\| \bigr).
\]
\end{Remark}

\subsection{Orthogonally invariant distributions}

The previous differentiability results involve the operator $H_\rho
(Q)$. The latter turns out to have a special structure under a certain
symmetry condition on $Q$:

\begin{Definition}[Orthogonal symmetry]
The distribution $Q$ of a random matrix $M \in\Rqqsympsd$ is called
orthogonally invariant if
\[
\LL(VMV^\top) \ = \ \LL(M)
\quad\text{for any orthogonal matrix} \ V \in\Rqq.
\]
\end{Definition}

This property is closely related to spherically symmetric distributions
on $\R^q$. For instance, let $Q = \LL(XX^\top)$ with a random vector
$X$ with spherically symmetric distribution on $\R^q$. Then $Q$ is
orthogonally invariant. Another example is given by $Q = \LL
(S(X_1,X_2,\ldots,X_k))$ with independent, identically distributed
random vectors $X_1, X_2, \ldots, X_k \in\R^q$ such that $\LL(X_1 -
\mu)$ is spherically symmetric for some $\mu\in\R^q$.

By linear equivariance of $\bSigma_\rho(\cdot)$, orthogonal
invariance of $Q$ implies that $\bSigma_\rho(Q)$ is a positive
multiple of $I_q$. As shown in the subsequent lemma, the operator
$H_\rho(Q)$ has a rather simple form here. It will be convenient to
decompose $\Rqqsym$ as
\[
\Rqqsym\ = \ \W_0 + \W_1
\]
with $\W_0 = \{A \in\Rqqsym: \tr(A) = 0\}$ and $\W_1 := \{s I_q :
s \in\R\}$. Any matrix $A \in\Rqqsym$ has the unique decomposition
\[
A \ = \ A_0 + A_1
\]
with $A_0 := A - q^{-1} \tr(A) I_q \in\W_0$ and $A_1 := q^{-1} \tr
(A) I_q \in\W_1$.

\begin{Lemma}
\label{lem:OrthInvQ}
Suppose that $Q$ is orthogonally invariant, and let $\bSigma_\rho(Q)
= I_q$. Then for $A = A_0 + A_1$ with $A_0 \in\W_0, A_1 \in\W_1$,
\[
H_\rho(Q) A
\ = \ d_0(Q) A_0 + d_1(Q) A_1,
\]
where
\begin{align*}
d_0(Q) \
&:= \ 1 + \frac{2}{q(q+2)}
\int\rho''(\tr(M))
\Bigl( \|M\|_F^2 + \frac{\tr(M)^2 - \|M\|_F^2}{q - 1} \Bigr)
\, Q(dM), \\
d_1(Q) \
&:= \ 1 + \frac{1}{q}
\int\rho''(\tr(M)) \tr(M)^2 \, Q(dM).
\end{align*}
\end{Lemma}

\paragraph{Implications for rank one distributions}
Suppose that a random matrix $M \sim Q$ satisfies $\mathrm{rank}(M)
\le1$ almost surely. This is true in settings \eqref{eq:Example Q 1}
and \eqref{eq:Example Q 2} with $k = 2$. Then $\|M\|_F = \tr(M)$
almost surely, so
\begin{align*}
d_0(Q) \
&= \ 1 + \frac{2}{q(q+2)}
\int\rho''(\tr(M)) \tr(M)^2 \, Q(dM), \\
d_1(Q) \
&= \ 1 + \frac{1}{q}
\int\rho''(\tr(M)) \tr(M)^2 \, Q(dM).
\end{align*}

\paragraph{Implications for Case~0}
Recall that in Case~0, $\rho(s) = q \log(s)$, so $\rho'(s) = q/s$
and $\rho''(s) = -q/s^2$. Thus $d_1(Q) = 0$, and for $A = A_0 + A_1$
with \mbox{$A_0 \in\V_0, A_1 \in\V_1$},
\[
H_\rho(Q) A
\ = \ d_0(Q) A_0
\]
with
\[
d_0(Q) \ = \ 1 - \frac{2}{q+2}
\int\frac{(q-2) \|M\|_F^2/\tr(M)^2 + 1}{q-1} \, Q(dM).
\]
In particular, if $\mathrm{rank}(M) = 1$ almost surely, then
\[
H_\rho(Q) A \ = \ \frac{q}{q+2} \, A_0.
\]

\subsection{Consistency and Central Limit Theorems}

In this section we apply the previous results to particular empirical
distributions related to Settings~\eqref{eq:Example Q 1} and \eqref
{eq:Example Q 2}. For convenience we restrict our attention to Cases~0
and 1'.

For some fixed integer $k \ge1$ and arbitrary integers $n \ge k$ we
consider distributions
\[
Q \ := \ Q^k(P)
\quad\text{and}\quad
Q_n \ := \ Q^k(P_n)
\]
in $\QQ_\rho$ with distributions $P, P_n$ on $\R^q$ such that
\[
\bSigma_\rho(Q)
\ = \ I_q
\ = \ \bSigma_\rho(Q_n)
\quad\text{for all} \ n \ge k.
\]
Recall that in Case~0, $\tilde{Q} = Q^k(\tilde{P}) \in\QQ_\rho$
implies that
\[
\begin{cases}
\tilde{P}(\{0\}) = 0
&\text{if} \ k = 1, \\
\tilde{P}(\{x\}) = 0 \ \ \text{for all} \ x \in\R^q
&\text{if} \ k \ge2.
\end{cases}
\]

\paragraph{Additional assumptions}
We assume that
\[
P_n \ \to_w \ P.
\]
Further, for a certain exponent $m \ge1$ we assume that
\[
\int\psi(\|x\|^2)^m \, P_n(dx)
\ \to\ \int\psi(\|x\|^2)^m \, P_n(dx),
\]
where all integrals on the left and right hand side are finite.

\smallskip

Note that for any exponent $m \ge1$, the second part of the additional
assumptions is a consequence of the first part whenever $\psi(\infty)
< \infty$.

Now we consider for $n \ge k$ independent random vectors $X_{n1},
X_{n2}, \ldots, X_{nn}$ with distribution $P_n$ and define
\[
\hat{Q}_n \ := \
\begin{cases}
\displaystyle
\frac{1}{n} \sum_{i=1}^n
\delta_{X_{ni}^{} X_{ni}^\top}^{}
& \text{if} \ k = 1, \\[2.5ex]
\displaystyle
\binom{n}{k}^{-1} \sum_{1\le i_1 < \cdots< i_k \le n}
\delta_{S(X_{ni_{1}}, \ldots, X_{ni_{k}})}^{}
& \text{if} \ k \ge2.
\end{cases}
\]

Our first result proves consistency of $\bSigma_\rho(\hat{Q}_n)$ as
an estimator for \mbox{$\bSigma_\rho(Q_n) = I_q$}. It is essentially a
corollary to Theorem~\ref{thm:Consistency}:

\begin{Theorem}[Consistency]
\label{thm:Consistency S12}
In the setting just described, suppose that the additional assumptions
hold with $m = 1$. Then $\Pr(\hat{Q}_n \in\QQ_\rho) \to1$ and
\[
\bSigma_\rho(\hat{Q}_n) \ \to_p \ I_q.
\]
\end{Theorem}

Our second result provides a precise linear expansion for $\bSigma
_\rho(\hat{Q}_n)$ and is based on Theorem~\ref{thm:Diffability}:

\begin{Theorem}[Linear expansion]
\label{thm:CLT1}
Let \ $\mathbb{X} := \R^q\setminus\{0\}$ in Case~0 with $k = 1$, and
$\mathbb{X} := \R^q$ otherwise. In the just described setting,
suppose that the additional assumptions hold with $m = 2$.
Then $\Pr(\hat{Q}_n \in\QQ_\rho) \to1$ and
\[
\sqrt{n} \log(\bSigma_\rho(\hat{Q}_n))
\ = \ \frac{1}{\sqrt{n}} \sum_{i=1}^n \bigl( Z(X_{ni}) - \Ex
Z(X_{n1}) \bigr)
+ o_p(1)
\]
for some continuous function $Z : \mathbb{X} \to\Rqqsym$ depending
only on $P$ such that
\[
\sup_{x \in\mathbb{X}} \, \frac{\|Z(x)\|}{1 + \psi(\|x\|^2)} \ < \
\infty
\quad\text{and}\quad
\int Z \, dP \ = \ 0.
\]
Precisely, if $k = 1$, then
\[
Z(x) \ := \ H_\rho(Q)_{}^{-1} \bigl( \rho'(\|x\|^2) x x^\top- I_q
\bigr)
\quad\text{and}\quad
\Ex Z(X_{n1}) \ = \ 0.
\]
If $k \ge2$, then
\[
Z(x) \ = \ k H_\rho(Q)^{-1} \Bigl(
\Ex\bigl[ \rho' \bigl( \tr(S(x,X_2,\ldots,X_k)) \bigr)
S(x,X_2,\ldots,X_k) \bigr]
- I_q \Bigr)
\]
with independent random vectors $X_2, \ldots, X_k \sim P$.
\end{Theorem}

\begin{Remark}[Central Limit Theorem]
By virtue of the multivariate version of Lindeberg's Central Limit
Theorem, the expansion in Theorem~\ref{thm:CLT1} implies a Central
Limit Theorem for the estimator $\bSigma_\rho(\hat{Q}_n)$. Namely,
\[
\LL\bigl( \sqrt{n} \log(\bSigma_\rho(\hat{Q}_n)) \bigr)
\ \to_w \ \NN_{q\times q} \bigr( 0, \Cov(Z(X)) \bigr)
\]
with $X \sim P$. This means, that for any matrix $A \in\Rqqsym$,
\[
\bigl\langle\sqrt{n} \log(\bSigma_\rho(\hat{Q}_n)), A \bigr
\rangle
\ \to_w \ \NN\bigr( 0, \Var(\langle Z(X), A\rangle) \bigr).
\]
\end{Remark}

\begin{Remark}[Spherical symmetry I]
\label{rem:CLT1.R1}
Let $P$ be spherically symmetric around $0 \in\R^q$. Then the
matrix-valued function $Z$ in Theorem~\ref{thm:CLT1} may be written as
\[
Z(x) \ = \ z_0(\|x\|^2) xx^\top+ z_1(\|x\|^2) I_q
\]
with certain functions $z_0, z_1 : [0,\infty) \to\R$, where $z_1(s)
= - q^{-1} s z_0(s)$ in Case~0.
\end{Remark}

\begin{Remark}[Spherical symmetry II]
\label{rem:CLT1.R2}
Let $P$ be spherically symmetric around $0 \in\R^q$, and let $k = 1$.
Further let
\[
\rho(s) \ = \ (\nu+ q) \log(\nu+ s)
\]
with $\nu= 0$ (Case~0) or $\nu> 0$ (Case~1'). For $x \in\R^q$ we write
\[
xx^\top\ = \ A_0(x) + a(x) I_q + I_q
\]
with $a(x) := q^{-1} \|x\|^2 - 1$, so that $\tr(A_0(x)) = 0$. Then the
matrix-valued function $Z$ in Theorem~\ref{thm:CLT1} is given by
\[
Z(x) \ = \ (\nu+ \|x\|^2)^{-1} \bigl( c_0 A_0(x) + c_1 a(x) I_q \bigr)
\]
with
\[
c_0 \ := \ \frac{(q + \nu)(q + 2)}{q + 2(1 - \beta)\nu/q},
\quad
c_1 \ := \ 1_{[\nu> 0]} \frac{q}{1 - \beta}
\]
and
\[
\beta= \beta(P,\nu) \ := \ \int\frac{(\nu+ q)\nu}{(\nu+ \|x\|
^2)^2} \, P(dx).
\]
\end{Remark}


\section[$M$-functionals of location and scatter]{$\boldsymbol{M}$-functionals of location and scatter}
\label{sec:Location and Scatter}

Now we return to the estimation of location and scatter as in
Section~\ref{subsec:LocationScatter}. We restrict our attention to
$M$-functionals derived from multivariate $t$-distributions with $\nu
\ge1$ degrees of freedom. That means, for an arbitrary distribution
$P$ on $\R^q$ we consider
\[
L(\mu,\Sigma,P)
\ := \ \int\bigl[ \rho\bigl( (x - \mu)^\top\Sigma^{-1} (x - \mu
) \bigr)
- \rho(x^\top x) \bigr] \, P(dx)
+ \log\det(\Sigma)
\]
as in \eqref{eq:def L 2}, where
\[
\rho(s) \ = \ \rho_{\nu,q}(s) := (\nu+ q) \log(\nu+ s).
\]

The reason for the restriction to $\rho_{\nu,q}$ with $\nu\ge1$ is
a nice trick by Kent and Tyler (\citeyear{Kent_Tyler_1991}) to
reduce the location-scatter problem in dimension $q$ to the
scatter-only problem in dimension $q+1$ with $\nu-1$ in place of $\nu$.
As shown by Kent et al.\ (\citeyear{Kent_etal_1994}), the
particular loss functions $\rho_{\nu,q}$ are the only ones for which
this trick works.

For more details about and generalizations of multivariate
$t$-distributions we refer to Lange et al.\
(\citeyear{Lange_etal_1989}) and the monograph by Kotz and
Nadarajah (\citeyear{Kotz_Nadarajah_2004}). An alternative approach
to the location-scatter problem which is closely related to
Tyler's (\citeyear{Tyler_1987a}) scatter functional is presented by
Hettmansperger and Randles (\citeyear{Hettmansperger_Randles_2002}).

\subsection{Existence and uniqueness}

The first question is under what conditions on $P$ the functional
$L(\cdot,\cdot,P)$ admits a unique minimizer $(\bmu(P),\bSigma
(P))$. To this end let
\[
y = y(x)
\ := \
\begin{bmatrix} x \\ 1
\end{bmatrix}
\quad\text{and}\quad
\Gamma
\ := \
\begin{bmatrix}
\Sigma+ \mu\mu^\top& \mu\\ \mu^\top& 1
\end{bmatrix}
=
\begin{bmatrix}
I_q & \mu\\ 0 & 1
\end{bmatrix}
\begin{bmatrix}
\Sigma& 0 \\ 0 & 1
\end{bmatrix}
\begin{bmatrix}
I_q & \mu\\ 0 & 1
\end{bmatrix}
^\top
\]
for $x \in\R^q$ and $(\mu,\Sigma) \in\R^q \times\Rqqsympd$.
Then one can easily verify that
\[
\det(\Gamma) \ = \ \det(\Sigma),
\quad
\Gamma^{-1} \ = \
\begin{bmatrix}
I_q & -\mu\\ 0 & 1
\end{bmatrix}
^\top
\begin{bmatrix}
\Sigma^{-1} & 0 \\ 0 & 1
\end{bmatrix}
\begin{bmatrix}
I_q & -\mu\\ 0 & 1
\end{bmatrix}
\]
and
\[
y^\top\Gamma^{-1} y \ = \ (x - \mu)^\top\Sigma^{-1} (x - \mu) + 1.
\]
Consequently, with
\[
\tilde{P} \ := \ \LL(y(X)), \quad X \sim P,
\]
and
\[
\tilde{\rho}(s) \ := \ \rho(s - 1) \ = \ \rho_{\nu-1,q+1}(s)
\]
we may write
\[
L(\mu,\Sigma,P)
\ = \ \tilde{L}(\Gamma,\tilde{P})
\ := \ \int\bigl[ \tilde{\rho}(y^\top\Gamma^{-1}y) - \tilde{\rho
}(y^\top y) \bigr]
\, \tilde{P}(dy)
+ \log\det(\Gamma).
\]
If a matrix $\Gamma\in\R^{(q+1)\times(q+1)}_{{\rm sym},>0}$
minimizes $\tilde{L}(\cdot,\tilde{P})$, and if
\[
\Gamma_{q+1,q+1} \ = \ 1,
\]
then we may write
\[
\Gamma\ = \
\begin{bmatrix}
\bs{\Sigma}(P) + \bs{\mu}(P)\bs{\mu}(P)^\top& \bs{\mu}(P) \\
\bs{\mu}(P)^\top& 1
\end{bmatrix}
,
\]
and $(\bs{\mu}(P),\bs{\Sigma}(P)) \in\R^q \times\Rqqsympd$
solves the original minimization problem. It will turn out that the
additional constraint $\Gamma_{q+1,q+1} = 1$ poses no problem here.

Concerning the minimization of $\tilde{L}(\cdot,\tilde{P})$ over $\R
^{(q+1)\times(q+1)}_{{\rm sym},>0}$, one can deduce from Theorem~\ref
{thm:Uniqueness} that the following condition on $P$ plays a crucial role:
\begin{align}
\label{eq:ExistenceMVT}
P(a + \mathbb{V}) \ < \ \frac{\dim(\mathbb{V}) + \nu}{q + \nu}
\quad
& \text{for arbitrary} \ a \in\R^q \ \text{and linear} \\[-1.5ex]
\nonumber
& \text{subspaces} \ \mathbb{V} \subset\R^q \
\text{with} \ 0 \le\dim(\mathbb{V}) < q.
\end{align}
Here is the main result:

\begin{Theorem}
\label{thm:MTLS}
In case of $\nu= 1$, the functional $\tilde{L}(\cdot,\tilde{P})$
has a unique minimizer $\Gamma$ with $\Gamma_{q+1,q+1} = 1$ if, and
only if, \eqref{eq:ExistenceMVT} holds true. Moreover, if $\tilde
{\Gamma}$ is some minimizer of $\tilde{L}(\cdot,\tilde{P})$, then
$\Gamma= (\tilde{\Gamma}_{q+1,q+1})^{-1} \tilde{\Gamma}$.

In case of $\nu> 1$, the functional $\tilde{L}(\cdot,\tilde{P})$
has a unique minimizer $\Gamma$ if, and only if, \eqref
{eq:ExistenceMVT} holds true. This minimizer satisfies automatically
$\Gamma_{q+1,q+1} = 1$.
\end{Theorem}

Consequently, Condition \eqref{eq:ExistenceMVT} is both necessary and
sufficient for $L(\cdot,\cdot,P)$ to have a unique minimizer $(\bmu
(P),\bSigma(P))$. In that case, we have to minimize $\tilde{L}(\cdot
,\tilde{P})$, which is equivalent to finding a solution $\Gamma\in\R
^{(q+1)\times(q+1)}_{{\rm sym},>0}$ of the fixed point equation
\[
\Gamma
\ = \ \int\rho'(\|y\|^2-1) \, yy^\top\, \tilde{P}(dy)
= \int\rho'(\|x\|^2) y(x)y(x)^\top\, P(dx).
\]
If we write such a matrix $\Gamma$ as
\[
\Gamma\ = \
\begin{bmatrix}
A & b \\
b^\top& c
\end{bmatrix}
\]
with $A \in\Rqqsym$, $b \in\R^q$ and $c = \Gamma_{q+1,q+1} > 0$, then
\[
\bmu(P) \ = \ c^{-1} b
\quad\text{and}\quad
\bSigma(P) \ = \ c^{-1} A - \bmu(P)\bmu(P)^\top.
\]
Moreover, $c = 1$ in case of $\nu> 1$.

\subsection{Weak differentiability and linear expansions}

The results for weak continuity and differentiability of scatter-only
functionals imply analogous results for the location-scatter problem.
Let $(P_n)_n$ be a sequence of probability distributions on $\R^q$
converging weakly to a distribution $P$ such that $(\bmu(P),\bSigma
(P))$ is well-defined. Then for sufficiently large $n$, $(\bmu
(P_n),\bSigma(P_n))$ is well-defined, too, and
\[
(\bmu(P_n),\bSigma(P_n)) \ \to\ (\bmu(P),\bSigma(P)).
\]
(Again asymptotic statements are meant as $n \to\infty$.) This
follows from Theorem~\ref{thm:Continuity}, applied to $Q_{(n)} := \LL
\bigl( y(X)y(X)^\top\bigr)$, $X \sim P_{(n)}$. Theorem~\ref
{thm:Diffability} yields the following expansion:

\begin{Theorem}
\label{thm:WeakDiffability.t}
Let $P$ be a probability distribution on $\R^q$ such that $\bmu(P) =
0$ and $\bSigma(P) = I_q$. Then there exists a bounded and continuous function
\[
\tilde{Z} : \R^q \to\R^{(q+1)\times(q+1)}_{{\rm sym}}\vadjust{\goodbreak}
\]
depending only on $P$ such that $\int\tilde{Z} \, dP = 0$ with the
following property: Let $\hat{P}_1$, $\hat{P}_2$, $\hat{P}_3$,
\ldots be random distributions on $\R^q$ such that for any bounded and
continuous function $f : \R^q \to\R$,
\[
\int f \, d\hat{P}_n \ \to_p \ \int f \, dP.
\]
Then $\bigl( \bmu(\hat{P}_n), \bSigma(\hat{P}_n) \bigr)$ is
well-defined with asymptotic probability one, and
\[
\begin{bmatrix}
\bSigma(\hat{P}_n) - I_q & \bmu(\hat{P}_n) \\
\bmu(\hat{P}_n)^\top& 0
\end{bmatrix}
\ = \ \int\bigl( \tilde{Z} - \tilde{Z}_{q+1,q+1} I_{q+1} \bigr) \,
d\hat{P}_n
+ o_p \Bigl( \Bigl\| \int\tilde{Z} \, d\hat{P}_n \Bigr\| \Bigr).
\]

The precise definition of $\tilde{Z}$ is
\[
\tilde{Z}(x) \ := \ \tilde{H}(P)^{-1} \bigl( \rho'(\|x\|^2) y(x)
y(x)^\top- I_{q+1} \bigr),
\]
where $\tilde{H}(P) : \tilde{\mathbb{M}} \to\tilde{\mathbb{M}}$
is the linear operator given by
\[
\tilde{H}(P) M
\ := \ M + \int\rho''(\|x\|^2) \, y(x)^\top M y(x) \, y(x)y(x)^\top
\, P(dx)
\]
for matrices $M$ in
\[
\tilde{\mathbb{M}} \ := \
\begin{cases}
\bigl\{ M \in\R^{(q+1)\times(q+1)}_{\rm sym} : \tr(M) = 0 \bigr\}
& \text{if} \ \nu= 1, \\
\R^{(q+1)\times(q+1)}_{\rm sym}
& \text{if} \ \nu> 1.
\end{cases}
\]
Moreover, in case of $\nu> 1$,
\[
\tilde{Z}_{q+1,q+1} \ \equiv\ 0.
\]
\end{Theorem}

\begin{Remark}[Empirical distributions]
\label{rem:WeakDiffability.t}
Let $P_1, P_2, P_3, \ldots$ and $P$ be distributions on $\R^q$ such
that $P_n \to_w P$ and $\bmu(P) = 0 = \bmu(P_n)$ and $\bSigma(P) =
I_q = \bSigma(P_n)$ for all $n$. Further let $\hat{P}_n$ be the
empirical distribution of independent random vectors $X_{n1}, X_{n2},
\ldots, X_{nn}$ with distribution $P_n$. As in the proof of
Theorem~\ref{thm:Consistency S12} one can show that these random
distributions $\hat{P}_n$ satisfy the assumptions of Theorem~\ref
{thm:WeakDiffability.t}. This implies that $(\bmu(\hat{P}_n), \bSigma
(\hat{P}_n))$ is well-defined with asymptotic probability one, and
\[
\sqrt{n}
\begin{bmatrix}
\bSigma(\hat{P}_n) - I_q & \bmu(\hat{P}_n) \\
\bmu(\hat{P}_n)^\top& 0
\end{bmatrix}
\ = \ \frac{1}{\sqrt{n}} \sum_{i=1}^n
\bigl( \tilde{Z}(X_{ni}) - \tilde{Z}(X_{ni})_{q+1,q+1} I_q \bigr)
+ o_p(1)
\]
with $\tilde{Z} : \R^q \to\R^{(q+1)\times(q+1)}_{{\rm sym}}$ as in
Theorem~\ref{thm:WeakDiffability.t}. In particular, $\Ex\tilde
{Z}(X_{n1}) = 0$ for all $n$, and the random matrix in the previous
display converges in distribution to a random matrix with a centered
Gaussian distribution on $\R^{(q+1)\times(q+1)}_{{\rm sym}}$.
\end{Remark}

\begin{Remark}[Symmetry]
\label{rem:WeakDiffability.t.R1}
Suppose that $P$ is symmetric in the sense that $\LL(-X) = \LL(X)$
for $X \sim P$. Then the function $\tilde{Z}$ in Theorem~\ref
{thm:WeakDiffability.t} may be written as
\[
\tilde{Z}(x) \ = \
\begin{bmatrix} Z(xx^\top) & 0 \\ 0 & z(\|x\|^2)
\end{bmatrix}
+
\rho'(\|x\|^2)
\begin{bmatrix} 0 & Bx \\ x^\top B & 0
\end{bmatrix}\vadjust{\goodbreak}
\]
with bounded and continuous functions $Z : \Rqqsympsd\to\Rqqsym$, $z
: [0,\infty) \to\R$ and a nonsingular matrix $B \in\Rqqsym$. In
particular, the random variables $\sqrt{n}(\bSigma(\hat{P}_n) -
I_q)$ and $\sqrt{n} \bmu(\hat{P}_n)$ in Remark~\ref
{rem:WeakDiffability.t} are asymptotically independent.
\end{Remark}

\begin{Remark}[Spherical symmetry]
\label{rem:WeakDiffability.t.R2}
Suppose that $P$ is spherically symmetric around $0$. Let $\beta=
\beta(P,\nu)$, $A_0(\cdot)$ and $a(\cdot)$ be defined as in
Remark~\ref{rem:CLT1.R2}. Then the function $\tilde{Z} - \tilde
{Z}_{q+1,q+1} I_{q+1}$ in Theorem~\ref{thm:WeakDiffability.t} may be
written as follows:
\[
\tilde{Z}(x) - \tilde{Z}(x)_{q+1,q+1} I_{q+1}
\ = \ (\nu+ \|x\|^2)^{-1}
\begin{bmatrix}
c_0 A_0(x) + c_1 a(x) I_q & c_2 x \\
c_2 x^\top& 0
\end{bmatrix}
\]
where
\[
c_0 \ := \ \frac{(q + \nu)(q + 2)}{q + 2(1 - \beta)\nu/q}, \quad
c_1 \ := \ \frac{q}{1 - \beta}
\quad\text{and}\quad
c_2 \ := \ \frac{q}{q - 2(1 - \beta)}.
\]
Comparing this with Remark~\ref{rem:CLT1.R2}, we see that the
estimator $\bSigma(\hat{P}_n)$ has the same asymptotic behaviour as
the corresponding estimator in the scatter-only problem.
\end{Remark}


\section{Auxiliary results and proofs}
\label{sec:Proofs}

\subsection{Proofs for Section~\ref{sec:Equivariance}}

\begin{proof}[\bf Proof of Lemma~\ref{lem: linear equivariance}]
Note that $(X_{\pi(i)})_{i=1}^q = BX$ with the permutation matrix $B =
(1_{[\pi(i) = j]})_{i,j=1}^q$. Thus our assumption on $X$ in part~(i)
and linear equivariance of $\bSigma(\cdot)$ imply that
\[
\bSigma(P) \ = \ B \bSigma(P) B^\top
\ = \ \bigl( \bSigma(P)_{\pi(i),\pi(j)} \bigr)_{i,j=1}^q
\]
for any permutation $\pi$ of $\{1,2,\ldots,q\}$ such that $\pi(i) =
i$ whenever $i \not\in J$. Let $j_1 := \min(J)$ and $j_2 := \max
(J)$. For arbitrary indices $j \ne k$ in $J$, choose $\pi$ such that
$\pi(j_1) = j$ and $\pi(j_2) = k$. Then we realize that $\bSigma
(P)_{j,j} = a(P) := \bSigma(P)_{j_1,j_1}$ and $\bSigma(P)_{j,k} =
b(P) := \bSigma(P)_{j_1,j_2}$. This proves part~(i).

To verify part~(ii) we write $(s_i X_i)_{i=1}^q = BX$ with $B := \diag
(s)$. Then
\[
\bSigma(P) \ = \ B \bSigma(P) B^\top
\ = \ \bigl( s_i s_j \bSigma(P)_{i,j} \bigr)_{i,j=1}^q.
\]
Consequently, $\bSigma(P)_{ij} = 0$ whenever $s_is_j = -1$, i.e.\ $s_i
\ne s_j$.

As for part~(iii), suppose first that $P$ is spherically symmetric.
This implies that $X \sim P$ satisfies the assumptions of part~(i) with
the full index set $J = \{1,2,\ldots,q\}$ and of part~(ii) for any
sign vector $s \in\{-1,1\}^q$. Hence $\bSigma(P) = c(P) I_q$ for some
$c(P) \ge0$. Now suppose that $P$ is elliptically symmetric with
center $0$ and scatter matrix $\Sigma\in\Rqqsympd$. Then the
distribution $P'$ of $X' := \Sigma^{-1/2} X$ is spherically symmetric,
and $P = {P'}^B$ with $B := \Sigma^{1/2}$. Thus $\bSigma(P) = B
\bSigma(P') B^\top= c(P') \Sigma$.
\end{proof}

\begin{proof}[\bf Proof of Lemma~\ref{lem: affine equivariance}]
Under the assumption of part~(i),
\[
\bmu(P) \ = \ \diag(s) \bmu(P)
\ = \ \bigl( s_i \bmu(P)_i \bigr)_{i=1}^q.
\]
Consequently, $\bmu(P)_i = 0$ whenever $s_i = -1$.\vadjust{\goodbreak}

If $P$ is elliptically symmetric with center $\mu$ and scatter matrix
$\Sigma$, then the distribution $P'$ of $X' := \Sigma^{-1/2}(X - \mu
)$ is spherically symmetric, and $P = {P'}^{\mu,B}$ with $B := \Sigma
^{1/2}$. But $X'$ satisfies the assumptions of part~(i) for any sign
vector $s \in\{-1,1\}^q$. Hence $\bmu(P') = 0$, and $\bmu(P) = \mu+
B \bmu(P') = \mu$. Moreover, $\bSigma(P) = B \bSigma(P') B^\top=
c(P') \Sigma$, according to Lemma~\ref{lem: linear equivariance},
applied to $P'$.
\end{proof}


\subsection{Proofs for Section~\ref{sec:Scatter}}

\begin{proof}[\bf Proof of Lemma~\ref{lem:Trace inequalities}]
Let $M = \sum_{i=1}^q \lambda_i(M) u_i^{} u_i^\top$ with eigenvalues
$\lambda_i(M) \ge0$ and an orthonormal basis $u_1$, $u_2$, \ldots,
$u_q$ of $\R^q$. Then $\tr(M) = \sum_{i=1}^q \lambda_i(M)$ and
\[
\tr(AM)
\ = \ \sum_{i=1}^q \lambda_i(M) u_i^\top A u_i^{}
\
\begin{cases}
\le\ \lambda_{\rm max}(A) \sum_{i=1}^q \lambda_i(M)
\ = \ \lambda_{\rm max}(A) \tr(M), \\
\ge\ \lambda_{\rm min}(A) \sum_{i=1}^q \lambda_i(M)
\ = \ \lambda_{\rm min}(A) \tr(M).
\end{cases}
\]
\\[-5ex]
\end{proof}

\begin{proof}[\bf Proof of Lemma~\ref{lem:Expansion of rho}]
For fixed $s>0$ and $x \in\R$ define $f(x) := \rho(e^x s)$. Then
$f'(x) = \rho'(e^x s) e^x s = \psi(e^x s)$. Consequently by the mean
value theorem,
\[
\rho(t) - \rho(s) \ = \ f(\log(t/s)) - f(0)
\ = \ f'(\xi) \log(t/s)
\ = \ \psi(e^\xi s) \log(t/s)
\]
with some number $\xi$ between $0$ and $\log(t/s)$. Since $\psi$ is
non-decreasing on $(0,\infty)$, either $\log(t/s) > 0$ and $\psi(s)
\le\psi(e^\xi s) \le\psi(t)$, or $\log(t/s) < 0$ and $\psi(t) \le
\psi(e^\xi s) \le\psi(s)$. In both cases, $\psi(s) \log(t/s) \le
\rho(t) - \rho(s) \le\psi(t) \log(t/s)$.

Note also that
\[
\rho(t) - \rho(s) \ = \ \rho'(\xi) (t - s)
\]
for some $\xi$ between $a$ and $b$. Hence if $\rho'$ is
non-increasing, the asserted inequalities follow from the fact that
either $t - s \ge0$ and $\rho'(t) \le\rho'(\xi) \le\rho'(s)$, or
$t - s < 0$ and $\rho'(s) \le\rho'(\xi) \le\rho'(t)$.
\end{proof}

\begin{proof}[\bf Proof of Lemma~\ref{lem:Example Q 1}]
It follows from \eqref{eq:PV=0} that
\[
\Pr\bigl( X_1, X_2, \ldots, X_k \ \text{are linearly independent}
\bigr)
\ = \ 1
\quad\text{for} \ k = 1, 2, \ldots, q.
\]
Indeed, $\Pr(X_1 \ne0) = 1$, and for $2 \le k \le q$,
\[
\Pr\bigl( X_k \not\in\mathrm{span}(X_1,\ldots,X_{k-1}) \,\big|\,
X_1,\ldots,X_{k-1} \bigr)
\ = \ 1.
\]
This implies that with probability one,
\[
\hat{Q}^1(\M(\V)) = \hat{P}(\V) \ \le\ \frac{\dim(\V)}{n}
\quad\text{for all} \ \V\in\VV_q \ \text{with} \ \dim(\V) < q.
\]
Consequently, according to Theorem~\ref{thm:Uniqueness}, $\bSigma
_\rho(\hat{Q}^1)$ is well-defined with probability one, provided that
\[
\frac{d}{n} \ < \
\begin{cases}
\displaystyle
\frac{d}{q} \quad\text{for} \ 1 \le d < q, & \text{in Case 0}, \\[2ex]
\displaystyle
\frac{\psi(\infty) - q + d}{\psi(\infty)} \quad\text{for} \ 0
\le d < q, & \text{in Case 1}.
\end{cases}
\]
But this can be shown to be equivalent to $n \ge q+1$ in Case~0 and $n
\ge q$ in Case~1.
\end{proof}

To understand setting \eqref{eq:Example Q 2} thoroughly, the following
two results about linear subspaces of $\R^q$ and sample covariance
matrices are useful:

\begin{Lemma}
\label{lem:Sample covariances}
For arbitrary integers $k \ge1$ and points $x_1, x_2, \ldots, x_k \in
\R^q$ with sample mean $\bar{x} = k^{-1} \sum_{i=1}^k x_i$,
\begin{align*}
\W(x_1,x_2,\ldots,x_k) \
&:= \ \mathrm{span}(x_i - x_j \,: \, i,j = 1,2,\ldots,k) \\
&\,= \ \mathrm{span}(x_1 - \bar{x}, x_2 - \bar{x}, \ldots, x_k -
\bar{x}) \\
&\,= \ \mathrm{span}(x_1 - x_a, x_2 - x_a, \ldots, x_k - x_a)
\end{align*}
for any $a \in\{1,2,\ldots,k\}$. Moreover, in case of $k \ge2$,
\[
S(x_1, x_2, \ldots, x_k) \, \R^q \ = \ \W(x_1, x_2, \ldots, x_k).
\]
\end{Lemma}

\begin{Corollary}
\label{cor:Sample covariances}
Let $x_1, x_2, \ldots, x_k$ and $y_1, y_2, \ldots, y_\ell$ be
arbitrary point in $\R^q$. Suppose that both $\W(x_1, x_2, \ldots,
x_k)$ and $\W(y_1, y_2, \ldots, y_\ell)$ are contained in a given
space $\V\in\VV_q$. If $\{x_1, x_2, \ldots, x_k\}$ and $\{y_1, y_2,
\ldots, y_\ell\}$ have at least one point in common, then
\[
\W(x_1, x_2, \ldots, x_m, \, y_1, y_2, \ldots, y_\ell) \ \subset\
\V.
\]
\end{Corollary}

\begin{proof}[\bf Proof of Lemma~\ref{lem:Sample covariances}]
For arbitrary indices $a, j \in\{1,2,\ldots,k\}$ we may write $x_j -
\bar{x} = (x_j - x_a) - k^{-1} \sum_{i=1}^k (x_i - x_a)$, so
\begin{align*}
\mathrm{span}(
&x_1 - \bar{x}, x_2 - \bar{x}, \ldots, x_k - \bar{x}) \\
&\subset\ \mathrm{span}(x_1 - x_a, x_2 - x_a, \ldots, x_k - x_a) \\
&\subset\ \mathrm{span}(x_i - x_j \,:\, i,j = 1,2,\ldots,k) \\
&= \ \mathrm{span} \bigl( (x_i - \bar{x}) - (x_j - \bar{x}) \,:\,
i,j = 1,2,\ldots,k) \\
&\subset\ \mathrm{span}(x_1 - \bar{x}, x_2 - \bar{x}, \ldots, x_k
- \bar{x}).
\end{align*}
Hence the preceding three inclusions are equalities.

Now suppose that $k \ge2$. Since $S := S(x_1,x_2,\ldots,x_k)$ is
positive semidefinite, it follows from its spectral representation that
a vector $w \in\R^q$ is perpendicular to the column space $S \, \R
^q$ if, and only if,
\[
0 \ = \ w^\top S w
\ = \ (k - 1)^{-1} \sum_{i=1}^k (w^\top(x_i - \bar{x}))^2,
\]
i.e.\ $w$ is perpendicular to $\mathrm{span}(x_1 - \bar{x}, x_2 -
\bar{x}, \ldots, x_k - \bar{x}) = \W(x_1,x_2,\ldots,x_k)$. Hence
the column space of $S$ is equal to $\W(x_1,x_2,\ldots,x_k)$.
\end{proof}

\begin{proof}[\bf Proof of Lemma~\ref{lem:Example Q 2}]
For any nonvoid index set $M \subset\{1,2,\ldots,n\}$ define $\W(M)
:= \W(X_i : i \in M)$; in particular, $\W(\{i\}) = \{0\}$. Then it
follows from Lemma~\ref{lem:Sample covariances} that for any $\V\in
\VV_q$,
\[
\hat{Q}^k(\M(\V))
\ = \ {\binom{n}{k}}^{-1} \sum_{J \in\mathcal{J}_k}
1_{[\W(J) \subset\V]},
\]
where $\mathcal{J}_k$ stands for the set of all subsets of $\{
1,2,\ldots,n\}$ with $k$ elements. Moreover, Corollary~\ref
{cor:Sample covariances} implies that for two nonvoid index sets $M,M'$,
\[
\W(M \cup M') \ \subset\V
\quad\text{if} \ \ \W(M) \subset\V, \W(M') \subset\V\ \text
{and} \ M \cap M' \ne\emptyset.
\]
Consequently, if we partition $\{1,2,\ldots,n\}$ into pairwise
disjoint and maximal subsets $M_1, M_2, \ldots, M_L$ such that $\W
(M_\ell) \subset\V$ for $\ell= 1, 2, \ldots, L$, then
\[
\hat{Q}^k(\M(\V))
\ = \ {\binom{n}{k}}^{-1} \sum_{\ell=1}^L \binom{\# M_\ell}{k}
\]
with the usual convention that $\binom{a}{k} := 0$ for integers $0 \le
a < k$.

For any fixed index set $M$ with $1 \le\# M \le q$ and an additional
index $j \not\in M$, it follows from \eqref{eq:PH=0} and Lemma~\ref
{lem:Sample covariances} that
\begin{align*}
\Pr\bigl( \W(M \cup\{j\}) \ne\W(M) \,\big|\, (X_i)_{i \ne j}
\bigr) \
&= \ \Pr\bigl( X_j - X_a \not\in\W(M) \,\big|\, (X_i)_{i \ne j}
\bigr) \\
&= \ 1,
\end{align*}
where $a$ is any index in $M$. This implies that with probability one,
for any given partition $M_1, M_2, \ldots, M_L$ of $\{1,2,\ldots,n\}$
into nonvoid subsets $M_\ell$,
\[
\dim\Bigl( \bigcup_{\ell=1}^L \W(M_\ell) \Bigr)
\ = \ \min\Bigl( \sum_{\ell=1}^L (\# M_\ell- 1), q \Bigr).
\]
In particular, for any $\V\in\VV_q$ with $d := \dim(\V) < q$, the
value of $\hat{Q}^k(\M(\V))$ is no larger than the maximum of
\begin{equation}
\label{eq:Strange}
{\binom{n}{k}}^{-1} \sum_{\ell=1}^L \binom{m_\ell+1}{k}
\end{equation}
over all integers $L \ge1$ and $m_1, m_2, \ldots, m_L \ge k-1$ such
that $\sum_{\ell=1} m_\ell\le d$. It will be shown later that this
maximum equals
\[
{\binom{n}{k}}^{-1} \binom{d+1}{k}.
\]
Since $(\psi(\infty) - q + d)/\psi(\infty) = 1 - (q - d)/\psi
(\infty) > 1 - (q - d)/q = d/q$ in Case~1, we conclude that $\bSigma
_\rho(\hat{Q}^k)$ is well-defined almost surely, provided that
\[
{\binom{n}{k}}^{-1} \binom{d+1}{k} \ < \ \frac{d}{q}
\quad\text{for} \ k-1 \le d < q.
\]
Since $\binom{d+1}{k} \big/ d$ is increasing in $d \ge k-1$, this
condition is equivalent to
\[
{\binom{n}{k}}^{-1} \binom{q}{k} \ < \ \frac{q-1}{q}.
\]
But this holds in case of $n \ge q+1$, since the left hand side equals
\[
{\binom{n}{k}}^{-1} \binom{q}{k}
\ \le\ {\binom{q+1}{k}}^{-1} \binom{q}{k}
\ = \ \frac{q-k+1}{q+1}
\ \le\ \frac{q-1}{q+1}
\ < \ \frac{q-1}{q}.\vadjust{\goodbreak}
\]

It remains to be shown that the sum $\sum_{\ell=1}^L \binom{m_\ell
+1}{k}$ in \eqref{eq:Strange} is not larger than $\binom{d+1}{k}$.
For this purpose, let $N_1, N_2, \ldots, N_L$ be disjoint subsets of
$\{1,2,\ldots,d\}$ with $\# N_\ell= m_\ell$, and let $M_\ell:=
N_\ell\cup\{d+1\}$. Then for $\ell, \ell' \in\{1,2,\ldots,L\}$
with $\ell\ne\ell'$, a subset of $M_\ell$ with $k$ elements is
different from any subset of $M_{\ell'}$ with $k$ elements.
Consequently,\vspace*{-3pt}
\begin{align*}
\sum_{\ell=1}^L \binom{m_\ell+1}{k} \
&= \ \sum_{\ell=1}^L \# \bigl\{ \text{subsets of $M_\ell$ with $k$
elements} \bigr\} \\
&\le\ \# \bigl\{ \text{subsets of $\{1,2,\ldots,d+1\}$ with $k$
elements} \bigr\} \\
&= \ \binom{d+1}{k}.
\end{align*}
\\[-6.5ex]
\end{proof}

\begin{proof}[\bf Proof of Theorem~\ref{thm:Uniqueness}]
The first part, i.e.\ the equivalence of the fixed-point equation $\Psi
_\rho(\Sigma,Q) = \Sigma$ and $\Sigma$ being a minimizer of $L_\rho
(\cdot,Q)$, follows from Propositions~\ref{prop:Expansion 1} and \ref
{prop:Convexity}: Recall that with $B := \Sigma^{1/2}$ we may write
\begin{align*}
L_\rho(\Sigma^{1/2} \exp(A) \Sigma^{1/2}, Q) - L_\rho(\Sigma,Q) \
&= \ L_\rho(\exp(A),Q_B) \\
&= \ \bigl\langle A, G_\rho(Q_B) \bigr\rangle
+ o(\|A\|)
\end{align*}
as $\Rqqsym\ni A \to0$, and
\[
G_\rho(Q_B)
\ = \ B^{-1} (\Sigma- \Psi_\rho(\Sigma,Q)) B^{-1}.
\]
If $\Sigma$ minimizes $L_\rho(\cdot,Q)$, then $G_\rho(Q_B) = 0$,
which is equivalent to $\Psi_\rho(\Sigma,Q) = \Sigma$. On the other
hand, if $\Sigma$ is not a minimizer of $L_\rho(\cdot,Q)$, then
there exists a matrix $A \in\Rqqsym$ such that $L_\rho(\exp(A),
Q_B) < 0$. But convexity of $\R\ni t \mapsto h(t) := L_\rho(\exp
(tA), Q_B)$ implies that
\[
0 \ > \ L_\rho(\exp(A),Q_B)
\ = \ h(1) - h(0)
\ \ge\ h'(0)
\ = \ \langle A, G_\rho(Q_B)\rangle,
\]
i.e.\ $G_\rho(Q_B) \ne0$ and thus $\Psi_\rho(\Sigma,Q) \ne\Sigma$.

In Case~1, suppose that Condition~1 holds true. According to
Proposition~\ref{prop:Coercivity}, $L(\cdot,Q)$ is a continuous
function on $\Rqqsympd$ which is coercive in that $L_\rho(\Sigma,Q)
\to\infty$ as $\|\log(\Sigma)\| \to\infty$. Consequently there
exists a minimizer $\Sigma_o$ of $L_\rho(\cdot,Q)$. But Condition~1
and Proposition~\ref{prop:Convexity} imply that $L_\rho(\Sigma
_o^{1/2} \exp(tA) \Sigma_o^{1/2}, Q)$ is strictly convex for any $A
\in\Rqqsym\setminus\{0\}$. Consequently, $\Sigma_o$ is the unique
minimizer of $L_\rho(\cdot,Q)$.

Still in Case~1, suppose that $\Sigma_o \in\Rqqsympd$ is a unique
minimizer of $L_\rho(\cdot,Q)$. Then $L_\rho(\Sigma_o^{1/2} \exp
(A) \Sigma_o^{1/2}, Q)$ is a coercive function of $A \in\Rqqsym$:
For if $\|A\| \ge1$ and $A' := \|A\|^{-1} A$, then by Proposition~\ref
{prop:Convexity},
\begin{align*}
L_\rho(
&\Sigma_o^{1/2} \exp(A) \Sigma_o^{1/2}, Q) - L_\rho(\Sigma_o,Q) \\
&= \ L_\rho(\Sigma_o^{1/2} \exp(\|A\| A') \Sigma_o^{1/2}, Q)
- L_\rho(\Sigma_o,Q) \\
&\ge\ \|A\|
\bigl( L_\rho(\Sigma_o^{1/2} \exp(A') \Sigma_o^{1/2}, Q)
- L_\rho(\Sigma_o,Q) \bigr) \\
&\ge\ \|A\| \, \min_{A'' \in\Rqqsym\,:\, \|A''\| = 1}
\bigl( L_\rho(\Sigma_o^{1/2} \exp(A'') \Sigma_o^{1/2}, Q)
- L_\rho(\Sigma_o,Q) \bigr),\vadjust{\vspace*{-6pt}\goodbreak}
\end{align*}
and the minimum on the right hand side is strictly positive by
uniqueness of the minimizer $\Sigma_o$. But coercivity of $L_\rho
(\Sigma_o^{1/2} \exp(\cdot) \Sigma_o^{1/2},Q)$ is equivalent to
Condition~1, according to Proposition~\ref{prop:Coercivity}.

In Case~0 one can argue in the same way, this time with $\{\Sigma\in
\Rqqsympd: \det(\Sigma) = 1\}$ and $\{A \in\Rqqsym: \tr(A) = 0\}$
in place of $\Rqqsympd$ and $\Rqqsym$, respectively.
\end{proof}

\begin{proof}[\bf Proof of Lemma~\ref{lem:PsiQ.better.than.I}]
Writing $\Psi(\tilde{Q}) = \Psi_\rho(I,\tilde{Q})$ for arbitrary
distributions $\tilde{Q}$ and $B := \Sigma^{1/2}$, note first that
\begin{align*}
L_\rho(\Psi_\rho(\Sigma,Q), Q) - L_\rho(\Sigma,Q) \
&= \ L_\rho( B \Psi(Q_B) B^\top, Q)
- L_\rho(BB^\top, Q) \\
&= \ L_\rho( \Psi(Q_B), Q_B).
\end{align*}
Hence it suffices to show that
\[
L_\rho(\Psi(Q_B),Q_B) \ < \ 0
\]
unless $\Psi(Q_B) = I_q$. It follows from the second part of
Lemma~\ref{lem:Expansion of rho} that for $\Gamma\in\Rqqsympd$,
\begin{align*}
L_\rho(\Gamma,Q_B) \
&= \ \int\bigl[ \rho(\tr(\Gamma^{-1} M)) - \rho(\tr(M)) \bigr]
\, Q_B(dM)
+ \log\det(\Gamma) \\
&\le\ \int\rho'(\tr(M)) \bigl[ \tr(\Gamma^{-1} M) - \tr(M)
\bigr] \, Q_B(dM)
+ \log\det(\Gamma) \\
&= \ \tr\bigl( (\Gamma^{-1} - I_q) \Psi(Q_B) \bigr)
+ \log\det(\Gamma).
\end{align*}
Hence
\begin{align*}
L_\rho(\Psi(Q_B),Q_B) \
&\le\ \tr( I_q - \Psi(Q_B)) + \log\det\Psi(Q_B) \\
&= \ \sum_{i=1}^q \bigl[ 1 - \lambda_i(\Psi(Q_B)) + \log\lambda
_i(\Psi(Q_B)) \bigr].
\end{align*}
Since $1 - x + \log x < 0$ for $0 < x \ne1$, the latter sum is
strictly negative unless $\lambda_i(\Psi(Q_B)) = 1$ for $1 \le i \le
q$, which is equivalent to $\Psi(Q_B) = I_q$, i.e.\ $\Psi_\rho
(\Sigma,Q) = \Sigma$.
\end{proof}

\begin{proof}[\bf Proof of Lemma~\ref{lem:Fixed-point algorithm}]
Under the stated conditions on the distribution $Q$, the function
$L_\rho(\cdot,Q)$ has a minimizer $\Sigma_o$, that means, $\Psi
_\rho(\Sigma_o,Q) = \Sigma_o$. Note that
\[
\Sigma_o^{-1/2}\Sigma_k^{}\Sigma_o^{-1/2}
\ = \ \Sigma_o^{-1/2}\Psi_\rho(\Sigma_{k-1}^{},Q)\Sigma_o^{-1/2}
\ = \ \Psi_\rho(\Sigma_o^{-1/2}\Sigma_{k-1}^{}\Sigma_o^{-1/2},
Q_{\Sigma_o^{1/2}}).
\]
Hence we may assume w.l.o.g.\ that $\Sigma_o = I_q$ and $\Psi_\rho
(I_q,Q) = I_q$. Again we write $\Psi(\cdot)$ instead of $\Psi_\rho
(\cdot,Q)$.

The equation $\Psi(I_q) = I_q$ implies that the mapping $\Psi$ has
the following properties, as shown below: For any $\Sigma\in\Rqqsympd$,
\begin{align*}
\lambda_{\rm min}(\Psi(\Sigma)) \
&\ge\ a :=
\begin{cases}
\lambda_{\rm min}(\Sigma) & \text{in Case~0}, \\
\min\{\lambda_{\rm min}(\Sigma), 1\} & \text{in Case~1},
\end{cases}
\\
\lambda_{\rm max}(\Psi(\Sigma)) \
&\le\ b :=
\begin{cases}
\lambda_{\rm max}(\Sigma) & \text{in Case~0}, \\
\max\{\lambda_{\rm max}(\Sigma), 1\} & \text{in Case~1}.
\end{cases}
\end{align*}
This follows from Lemma~\ref{lem:Trace inequalities} and various
properties of $\rho$: For any $M \in\Rqqsympsd$,
\[
\lambda_{\rm max}(\Sigma)^{-1} \tr(M)
\ \le\ \tr(\Sigma^{-1} M)
\ \le\ \lambda_{\rm min}(\Sigma)^{-1} \tr(M).
\]
Hence for any unit vector $v \in\R^q$,
\begin{align*}
v^\top\Psi(\Sigma) v&
\ = \ \int\rho'(\tr(\Sigma^{-1}M)) \, v^\top Mv \, Q(dM) \\
&
\begin{cases}
\displaystyle
\ge\ a \int\rho'(\tr(M)) \, v^\top Mv \, Q(dM)
\ = \ a \, v^\top\Psi(I_q) v
\ = \ a, \\
\displaystyle
\le\ b \int\rho'(\tr(M)) \, v^\top Mv \, Q(dM)
\ = \ b \, v^\top\Psi(I_q) v
\ = \ b,
\end{cases}
\end{align*}
because for $M \ne0$,
\[
\rho'(\tr(\Sigma^{-1} M))
\begin{cases}
\displaystyle
\ge\rho'(\tr(M)/a)
= \frac{\psi(\tr(M)/a)}{\tr(M)/a}
\ge\frac{a \psi(\tr(M))}{\tr(M)}
= a \rho'(\tr(M)), \\[2ex]
\displaystyle
\le\rho'(\tr(M)/b)
= \frac{\psi(\tr(M)/b)}{\tr(M)/b}
\le\frac{b \psi(\tr(M))}{\tr(M)}
= b \rho'(\tr(M)),
\end{cases}
\]
due to $\rho'$ being non-increasing and $\psi$ being constant in
Case~0 and increasing on $(0,\infty)$ in Case~1.

Now we define
\[
[a_k,b_k] \ := \
\begin{cases}
\bigl[ \lambda_{\rm min}(\Sigma_k), \lambda_{\rm max}(\Sigma_k)
\bigr]
& \text{in Case~0}, \\
\bigl[ \min\{\lambda_{\rm min}(\Sigma_k),1\}, \max\{\lambda_{\rm
max}(\Sigma_k),1\} \bigr]
& \text{in Case~1}.
\end{cases}
\]
Then $(a_k)_k$ and $(b_k)_k$ are non-decreasing and non-increasing,
respectively, with corresponding limits $a_* \le b_*$. In Case~0 we
have to show that $a_* = b_*$, because then $\Sigma_k \to a_* I_q$. In
Case~1 we have to show that $a_* = b_* = 1$, because then $\Sigma_k
\to I_q$. To this end, note that the set $\bigl\{ \Sigma\in\Rqqsym:
\lambda(\Sigma) \in[a_0,b_0]^q \bigr\}$ is compact. Hence there
exist indices $k(1) < k(2) < k(3) < \cdots$ such that $\Sigma_{k(\ell
)} \to\Sigma_*$ as $\ell\to\infty$, where $\lambda(\Sigma_*) \in
[a_0,b_0]^q$. Lemma~\ref{lem:PsiQ.better.than.I} entails that the
sequence $\bigl( L_\rho(\Sigma_k,Q) \bigr)_{k \ge0}$ is
non-increasing. Consequently, since $L_\rho(\cdot,Q)$ and $\Psi
(\cdot)$ are continuous,
\begin{align*}
L_\rho(\Sigma_*,Q) \
&= \ \lim_{\ell\to\infty} L_\rho(\Sigma_{k(\ell)},Q) \\
&= \ \lim_{\ell\to\infty} L_\rho(\Sigma_{k(\ell)+1},Q)
\ = \ \lim_{\ell\to\infty} L_\rho(\Psi(\Sigma_{k(\ell)},Q)
\ = \ L_\rho(\Psi(\Sigma_*),Q).
\end{align*}
Hence Lemma~\ref{lem:PsiQ.better.than.I} implies that $\Psi(\Sigma
_*) = \Psi_\rho(\Sigma_*,Q) = \Sigma_*$. Thus $\Sigma_*$ is a
minimizer of $L_\rho(\cdot,Q)$. In Case~0 this implies that $\Sigma
_*$ is a positive multiple of $I_q$, whence $a_* = \lambda_{\rm
min}(\Sigma_*) = \lambda_{\rm max}(\Sigma_*) = b_*$. In Case~1 this
implies that $\Sigma_* = I_q$, whence $a_* = \min\{\lambda_{\rm
min}(\Sigma_*),1\} = 1$ and $b_* = \max\{\lambda_{\rm max}(\Sigma
_*),1\} = 1$.
\end{proof}


\subsection{Proofs for Section~\ref{sec:AnalysisL}}
\begingroup\abovedisplayskip=7pt\belowdisplayskip=7pt
\begin{proof}[\bf Proof of Lemma~\ref{lem:Taylor of exp}]
By definition,
\[
\exp(A+\Delta) \ = \ \sum_{\ell=0}^\infty\frac{(A + \Delta)^\ell
}{\ell!},
\]
and for $\ell\ge1$, the expansion of $(A + \Delta)^\ell$ is the sum
of $A^\ell$ and all matrices of the form $A^{s_0}\Delta A^{s_1} \cdots
\Delta A^{s_k}$ with $k \in\{1, \ldots, \ell\}$ times the factor
$\Delta$ and exponents $s_0, \ldots, s_k \ge0$ such that $s_+ :=\sum
_{j=0}^k s_j$ equals $\ell-k$. Consequently,
\[
\exp(A + \Delta)
\ = \ \exp(A) + \sum_{k=1}^\infty T_k(A,\Delta)
\]
with
\[
T_k(A,\Delta)
\ := \ \sum_{s_0, \ldots, s_k \ge0}
\frac{A^{s_0} \Delta A^{s_1} \cdots\Delta A^{s_k}}{(s_+ + k)!}.
\]
Note that for given $\ell\ge k$ there are $\binom{\ell}{k}$ tupels
$(s_0,\ldots,s_k)$ of integers $s_j \ge0$ with $s_+ = \ell-k$. Thus
\[
\|T_k(A,\Delta)\|
\ \le\ \sum_{\ell=k}^\infty
\binom{\ell}{k} \frac{\|A\|^{\ell-k} \|\Delta\|^k}{\ell!}
\ = \ e_{}^{\|A\|} \frac{\|\Delta\|^k}{k!}.
\]
In particular,
\[
\exp(A + \Delta)
\ = \ \exp(A) + R_0(A,\Delta)
\ = \ \exp(A) + T_1(A,\Delta) + R_1(A,\Delta)
\]
with
\[
\|R_m(A,\Delta)\|
\ \le\ \sum_{k=m+1}^\infty e_{}^{\|A\|} \frac{\|\Delta\|^k}{k!}
\ \le\ e_{}^{\|A\| + \|\Delta\|} \frac{\|\Delta\|^{m+1}}{(m+1)!}
\]
for $m = 0,1$.

It remains to derive alternative expressions for $T_k(A,\Delta)$.
First of all, it follows from a well-known identity for the beta
function that
\begin{align*}
T_1(A,\Delta) \
&= \ \sum_{s_0,s_1 \ge0} \frac{A^{s_0} \Delta A^{s_1}}{(s_0 + s_1 +
1)!} \\
&= \ \sum_{s_0,s_1 \ge0} \frac{s_0! s_1!}{(s_0 + s_1 + 1)!}
\frac{A^{s_0}}{s_0!} \Delta\frac{A^{s_1}}{s_1!} \\
&= \ \sum_{s_0,s_1 \ge0} \int_0^1 (1 - u)^{s_0} u^{s_1} \, du \
\frac{A^{s_0}}{s_0!} \Delta\frac{A^{s_1}}{s_1!} \\
&= \ \int_0^1 \sum_{s_0,s_1 \ge0}
\frac{((1 - u)A)^{s_0}}{s_0!} \Delta\frac{(uA)^{s_1}}{s_1!} \ du \\
&= \ \int_0^1 \exp((1-u)A) \Delta\exp(uA) \, du.
\end{align*}
For general $k \ge1$ we utilize a special construction of the random
tupel $(U_{kj})_{j=0}^k$ which is well-known from uniform order
statistics: If $E_0, E_1, E_2, \ldots$ are independent standard
exponential random variables, then the random variable
$(U_{kj})_{j=0}^k := (E_j/F)_{j=0}^k$ with $F := \sum_{j=0}^k E_j$ has
the desired distribution. Moreover, $(U_{kj})_{j=0}^k$ and $F$ are
stochastically independent, where $F$ has distribution $\mathrm
{Gamma}(k+1,1)$. From these facts one can derive that
\begin{align*}
\Ex\bigl[ U_{k0}^{s_0} U_{k1}^{s_1} \cdots U_{kk}^{s_k} \bigr] \
&= \ \Ex\bigl[ F^{s_+} U_{k0}^{s_0} U_{k1}^{s_1} \cdots U_{kk}^{s_k}
\bigr]
\big/ \Ex(F^{s_+}) \\
&= \ \Ex\bigl[ E_{0}^{s_0} E_{1}^{s_1} \cdots E_{k}^{s_k} \bigr]
\big/ \Ex(F^{s_+})
\ = \ \frac{s_0! s_1! \cdots s_k!}{(s_+ + k)!/k!},
\end{align*}
so
\begin{align*}
T_k(A,\Delta) \
&= \ \frac{1}{k!} \sum_{s_0, \ldots, s_k \ge0}
\Ex\Bigl[ \frac{(U_{k0}A)^{s_0}}{s_0!} B \frac{(U_{k1}A)^{s_1}}{s_1!}
\cdots B \frac{(U_{kk}A)^{s_k}}{s_k!} \Bigr] \\
&= \ \frac{1}{k!}
\Ex\bigl[ \exp(U_{k0}A) B \exp(U_{k1}A) \cdots B \exp(U_{kk}A)
\bigr].
\end{align*}
\\[-5ex]
\end{proof}\endgroup

In our proofs of Propositions~\ref{prop:Expansion 1} and \ref
{prop:Expansion 2} we utilize two elementary bounds for random
variables with bounded support. The first one is well-known, but we
haven't seen the second one elsewhere.

\begin{Lemma}
\label{lem:Moments}
Let $Y$ be a random variable with values in $[a,b]$. Then
\[
\Var(Y) \ \le\ (b-a)^2/4
\quad\text{and}\quad
\bigl| \Ex\bigl( (Y-\Ex(Y))^3 \bigr) \bigr| \ \le\ (b-a)^3 / (6
\sqrt{3}).
\]
\end{Lemma}

In addition we need several properties of an auxiliary function:

\begin{Lemma}
\label{lem:Expansion 0}
Let $A \in\Rqqsym$ and $M \in\Rqqsympsd\setminus\{0\}$. For $t \in
\R$ let
\[
g(t) = g(t,A,M) \ := \ \log\tr(\exp(-tA)M).
\]
This defines a smooth convex function $g$ on $\R$ with the following
properties:
\begin{align*}
|g'| \
&\le\ \|A\| \
\quad\text{with}\quad
g'(0) \ = \ - \tr(AM)/\tr(M), \\
0 \ \le\ g'' \
&\le\ \|A\|^2
\quad\text{with}\quad
g''(0) \ = \ \tr(A^2M)/\tr(M) - \tr(AM)^2/\tr(M)^2, \\
|g'''| \
&\le\ \|A\|^3 \, 4/\sqrt{27}.
\end{align*}
Furthermore, either $g'' > 0$ on $\R$, or there exists an eigenvalue
$\lambda$ of $A$ such that
\[
M \ \in\ \M(\{ x \in\R^q : Ax = \lambda x\}),
\quad
g' \ \equiv\ - \lambda
\quad\text{and}\quad
g'' \ \equiv\ 0.
\]
\end{Lemma}

\begin{proof}[\bf Proof of Lemma~\ref{lem:Moments}]
It suffices to consider the case $[a,b] = [0,1]$, because otherwise one
could just replace $Y$ with $(Y - a)/(b - a)$. Then
\[
\Var(Y) \ = \ \Ex(Y^2) - \Ex(Y)^2
\ \le\ \Ex(Y) - \Ex(Y)^2 \ \le\ 1/4
\]
with equality if, and only if, $Y \in\{0,1\}$ almost surely and $\Ex
(Y) = 1/2$.\vadjust{\goodbreak}

As to the central third moment, with $\mu:= \Ex(Y)$ it suffices to
prove that
\begin{equation}
\label{ineq:3rd.moment}
\Ex((Y - \mu)^3) \ \le\ 1/(6 \sqrt{3}),
\end{equation}
because $- (Y - \mu)^3 = ((1 - Y) - (1 - \mu))^3$. We only have to
consider the situation that $0 < \mu< 1$ with strictly positive
probabilities $p_0 := \Pr(Y < \mu)$ and $p_1 := \Pr(Y \ge\mu)$,
because otherwise $Y = \mu$ almost surely. Note that $h(x) := (x - \mu
)^3$ is concave on $[0,\mu]$ and convex on $[\mu,1]$. Hence with
\[
x_0 \ := \ \Ex(Y \,|\, Y < \mu)
\quad\text{and}\quad
x_1 \ := \ \Ex(Y \,|\, Y \ge\mu)
\]
we may conclude from Jensen's inequality that
\begin{align*}
\Ex((Y - \mu)^3) \
&= \ p_0 \Ex(h(Y) \,|\ Y < \mu) + p_1 \Ex(h(Y) \,|\, Y \ge\mu) \\
&\le\ p_0 (x_0 - \mu)^3 + p_1 \Ex(h(Y) \,|\, Y \ge\mu) \\
&\le\ p_0 (x_0 - \mu)^3
+ p_1 \Ex\Bigl( \frac{1 - Y}{1 - \mu} h(\mu) + \frac{Y - \mu}{1
- \mu} h(1)
\,\Big|\, Y \ge\mu\Bigr) \\
&= \ p_0 (x_0 - \mu)^3
+ p_1 \Ex\Bigl( \frac{Y - \mu}{1 - \mu} (1 - \mu)^3
\,\Big|\, Y \ge\mu\Bigr) \\
&= \ p_0 (x_0 - \mu)^3
+ p_1 (x_1 - \mu) (1 - \mu)^2.
\end{align*}
Equality holds if
\[
Y \ \sim\ p_0 \delta_{x_0}^{}
+ \frac{p_1 (1 - x_1)}{1 - \mu} \delta_\mu^{}
+ \frac{p_1 (x_1 - \mu)}{1 - \mu} \delta_1^{}.
\]
Note that in the latter case, $\Ex(Y)$ is still equal to $\mu$,
because $p_0x_0 + p_1x_1 = \mu$. If we replace $\LL(Y)$ with $\LL(Y
\,|\, Y \ne\mu)$, the mean does not change, but $\Ex((Y - \mu)^3)$
increases by the factor $1/\Pr(Y \ne\mu)$. Thus it even suffices to
consider distributions $\LL(Y)$ which are concentrated on two points
$x_0 \in[0,1)$ and $1$. Finally, in case of $x_0 > 0$ we could replace
$Y$ and $\mu$ with $(Y - x_0)/(1 - x_0)$ and $(\mu- x_0)/(1 - x_0) =
\Pr(Y = 1)$, respectively. This would increase $\Ex((Y - \mu)^3)$ by
a factor $(1 - x_0)^{-3}$ and lead to a random variable with values in
$\{0,1\}$.

Finally we have to maximize
\[
(1 - \mu) (0 - \mu)^3 + \mu(1 - \mu)^3
\ = \ \mu(1 - \mu) (1 - 2\mu)
\]
over all $\mu\in(0,1)$. With $u := 1 - 2\mu\in(-1,1)$ one may write
\[
\mu(1 - \mu) (1 - 2\mu) \ = \ 4^{-1} (1 - u^2) u \ \le\ 1/(6\sqrt{3})
\]
with equality for $u = 1/\sqrt{3}$.
\end{proof}

\begin{proof}[\bf Proof of Lemma~\ref{lem:Expansion 0}]
Let $A = \sum_{i=1}^q \lambda_i(A) u_i^{} u_i^\top$ with an
orthonormal basis $u_1$, $u_2$, \ldots, $u_q$ of $\R^q$. Then $\tr
(M) = \sum_{i=1}^q u_i^\top M u_i^{}$ and
\[
g(t)
\ = \ \log\Bigl( \sum_{i=1}^q e_{}^{- t \lambda_i(A)} u_i^\top M
u_i^{} \Bigr)
\ = \ \log\tr(M) + \log\Ex(e^{tY}),
\]
where $Y \sim\sum_{i=1}^q p_i \delta_{- \lambda_i(A)}^{}$ with
$p_i^{} := u_i^\top M u_i^{} / \tr(M)$. Elementary calculations show that
\begin{align*}
g'(t) \
= \ & \Ex(e^{tY} Y) / \Ex(e^{tY}), \\
g''(t) \
= \ & \Ex(e^{tY} Y^2) / \Ex(e^{tY})
- \Ex(e^{tY} Y)^2 / \Ex(e^{tY})^2, \\
g'''(t) \
= \ & \Ex(e^{tY} Y^3) / \Ex(e^{tY})
- 3 \Ex(e^{tY} Y^2) \Ex(e^{tY} Y) / \Ex(e^{tY})^2 \\
&+ \ 2 \Ex(e^{tY} Y)^3 / \Ex(e^{tY})^3.
\end{align*}
Defining the modified distribution $\Pr_t$ via $\Pr_t(B) := \Ex
(e^{tY} 1_B) / \Ex(e^{tY})$, we may rewrite this as
\[
g'(t) \ = \ \Ex_t(Y),
\quad
g''(t) = \Var_t(Y)
\quad\text{and}\quad
g'''(t) \ = \ \Ex_t \bigl( (Y - \Ex_t(Y))^3 \bigr).
\]
In particular, $g'(0) = \Ex(Y)$ equals $- \tr(AM)/\tr(M)$, and
$g''(0) = \Var(Y)$ equals $\tr(A^2M)/\tr(M) - \tr(AM)^2/\tr(M)^2$.

Note that $|Y| \le\|A\|$, so $|g'| \le\|A\|$. Further it follows from
Lemma~\ref{lem:Moments} with $[a,b] = \bigl[ -\|A\|, \|A\| \bigr]$
that $0 \le g''(0) \le\|A\|^2$, and $|g'''| \le\|A\|^3 4/\sqrt{27}$.

Finally, for any $t_o \in\R$ the equation $g''(t_o) = 0$ is
equivalent to $Y$ being constant almost surely with respect to $\Pr
_{t_o}$. But this means that for some eigenvalue $\lambda$ of $A$,
\[
u_i^\top M u_i^{} \ = \ 0
\quad\text{whenever}\quad
\lambda_i(A) \ \ne\ \lambda,
\]
so $M \in\M(\{x \in\R^q : Ax = \lambda x\})$. This implies that
$g(t) = g(0) - \lambda t$ for all $t \in\R$, whence $g' \equiv
-\lambda$ and $g'' \equiv0$.
\end{proof}

\begin{proof}[\bf Proof of Proposition~\ref{prop:Expansion 1}]
Note first that
\[
L_\rho(\exp(A),Q)
\ = \ \tr(A) + \int\bigl[ \rho(\tr(\exp(-A)M)) - \rho(\tr(M))
\bigr] \, Q(dM).
\]
For fixed $M \in\Rqqsympsd\setminus\{0\}$ let $a := \tr(M) > 0$ and
$b := \tr(\exp(-A)M)$. Then $b/a \in\bigl[ \lambda_{\rm min}(\exp
(-A)), \lambda_{\rm max}(\exp(-A)) \bigr] \subset[e^{-\|A\|}, e^{\|
A\|}]$ by Lemma~\ref{lem:Trace inequalities}. Hence Lemma~\ref
{lem:Expansion of rho} implies that $\rho(\tr(\exp(-A)M)) - \rho
(\tr(M))$ equals
\[
\rho(b) - \rho(a)
\ = \ \psi(a) \log(b/a) + r_1(a,b)
\]
with
\begin{align*}
|r_1(a,b)| \
&\le\ \bigl( \psi(\max\{a,b\}) - \psi(\min\{a,b\}) \bigr) |\log
(b/a)| \\
&\le\ \bigl( \psi\bigl( e^{\|A\|} \tr(M) \bigr)
- \psi\bigl( e^{-\|A\|} \tr(M) \bigr) \bigr) \|A\|.
\end{align*}
Moreover, $\log(b/a) = g(1) - g(0)$ with $g = g(\cdot,A,M)$ as in
Lemma~\ref{lem:Expansion 0}. Hence for a suitable number $\xi\in(0,1)$,
\[
g(1) - g(0) \ = \ g'(0) + g''(\xi)/2\vadjust{\goodbreak}
\]
where $g'(0) = - \tr(AM)/\tr(M)$ and $0 \le g''(\xi) \le\|A\|^2$.
All in all we obtain the expansion
\begin{align*}
\rho(b) - \rho(a) \
&= \ \psi(a) g'(0) + \psi(a) g''(\xi)/2 + r_1(a,b) \\
&= \ - \rho'(\tr(M)) \tr(AM) + \psi(\tr(M)) g''(\xi)/2 + r_1(a,b).
\end{align*}
Consequently
\[
L_\rho(\exp(A),Q)
\ = \ \tr(A) - \int\rho'(\tr(M)) \tr(AM) \, Q(dM) + R_\rho(A,Q),
\]
where
\[
|R_\rho(A,Q)| \ \le\ \bigl( J_\rho(e^{\|A\|},Q) - J_\rho(e^{-\|A\|
},Q) \bigr) \|A\|
+ J_\rho(Q) \|A\|^2/2.
\]
Moreover,
\[
\tr(A) - \int\rho'(\tr(M)) \tr(AM) \, Q(dM) \ = \ \langle A,
G_\rho(Q)\rangle
\]
with $G_\rho(Q) = I_q - \int\rho'(\tr(M)) M \, Q(dM) = I_q - \Psi
_\rho(Q)$, and the inequalities $|\tr(A)| \le q \|A\|$ and $|\tr
(AM)\| \le\|A\| \tr(M)$ imply that $\bigl| \langle A, G_\rho
(Q)\rangle\bigr|$ is bounded by $(q + J_\rho(Q)) \|A\|$.
\end{proof}

\begin{proof}[\bf Proof of Corollary~\ref{cor:Diff.and.Lipschitz}]
For fixed $\Sigma\in\Rqqsympd$ let $B := \Sigma^{1/2}$. If $\Delta
\in\Rqqsym$ with $\|\Delta\| < \lambda_{\rm min}(\Sigma)$, then
$\Sigma+ \Delta\in\Rqqsympd$, too, and we may write
\[
\Sigma+ \Delta
\ = \ B (I_q + B^{-1}\Delta B^{-1}) B
\ = \ B \exp(A(\Delta)) B
\]
with $A(\Delta) := \log(I_q + B^{-1}\Delta B^{-1})$, whence
\[
L_\rho(\Sigma+ \Delta, Q) - L_\rho(\Sigma,Q)
\ = \ L_\rho(\exp(A(\Delta)),Q_B).
\]
As $\Delta\to0$,
\[
A(\Delta) \ = \ B^{-1} \Delta B^{-1} + O(\|\Delta\|^2),
\]
so it follows from Proposition~\ref{prop:Expansion 1} that
\begin{align*}
L_\rho(\exp(A(\Delta)),Q_B) \
&= \ \langle B^{-1} \Delta B^{-1}, G_\rho(Q_B) \rangle+ o(\|\Delta\|
) \\
&= \ \langle\Delta, B^{-1} G_\rho(Q_B) B^{-1} \rangle+ o(\|\Delta\|
).
\end{align*}
Consequently, $\nabla L_\rho(\Sigma,Q)$ equals
\[
B^{-1} G_\rho(Q_B) B^{-1}
\ = \ \Sigma^{-1} - \int\rho'(\tr(\Sigma^{-1}M)) \Sigma
^{-1}M\Sigma^{-1} \, Q(dM).
\]
By dominated convergence, this is continuous in $\Sigma$, because
$\Sigma\mapsto\Sigma^{-1}$ is continuous, $\rho'$ is continuous on
$(0,\infty)$, and the norm of the integrand on the right hand side is
not greater than $\lambda_{\rm min}(\Sigma)^{-1} \psi(\lambda_{\rm
min}(\Sigma)^{-1} \tr(M))$.

For a compact convex set $K \subset\Rqqsympd$ and $\Sigma_0,\Sigma
_1 \in K$ define the convex combination $\Sigma_t := (1 - t)\Sigma_0
+ t \Sigma_1$ for $t \in[0,1]$. Then $L_\rho(\Sigma_t,Q)$ is
differentiable in $t$ with derivative $\langle\Sigma_1 - \Sigma_0,
\nabla L_\rho(\Sigma_t,Q) \rangle$. Hence for a suitable point $\xi
\in(0,1)$ and $B := \Sigma_\xi^{1/2}$ it follows from the bounds in
Proposition~\ref{prop:Expansion 1} and inequality \eqref{ineq:JQ} that
\begin{align*}
\bigl| L_\rho(\Sigma_1,Q) - L_\rho(\Sigma_0) \bigr| \
&= \ \bigl| \langle\Sigma_1 - \Sigma_0,
\nabla L_\rho(\Sigma_\xi,Q) \rangle\bigr| \\
&= \ \bigl| \bigl\langle B^{-1}(\Sigma_1 - \Sigma_0) B^{-1},
G_\rho(Q_B) \bigr\rangle\bigr| \\
&\le\ (q + J_\rho(Q_B)) \bigl\| B^{-1}(\Sigma_1 - \Sigma_0) B^{-1}
\bigr\| \\
&\le\ \bigl( q + J_\rho(\lambda_{\rm min}(\Sigma_\xi)^{-1},Q)
\bigr)
\lambda_{\rm min}(\Sigma)^{-1} \bigl\| \Sigma_1 - \Sigma_0 \bigr
\| \\
&\le\ (q + J_\rho(\Lambda_K,Q) \bigr)
\Lambda_K \bigl\| \Sigma_1 - \Sigma_0 \bigr\|.
\end{align*}
\\[-5ex]
\end{proof}

\begin{proof}[\bf Proof of Proposition~\ref{prop:Convexity}]
Note first that by \eqref{eq:equivariance L_rho},
\begin{align*}
L_\rho( & B \exp(tA) B^\top, Q) - L_\rho(BB^\top, Q) \\
&= \ L_\rho(\exp(tA), Q_B) \\
&= \ t \cdot\tr(A)
+ \int\bigl[ \rho(\tr(\exp(-tA)M)) - \rho(\tr(M)) \bigr] \,
Q(dM).
\end{align*}
Thus we consider a fixed matrix $M \in\Rqqsympsd\setminus\{0\}$ and
verify convexity of
\[
h(t) = h(t,A,M) \ := \ \rho(e^{g(t)})
\]
with $g(t) = \log\tr(\exp(-tA)M))$ as in Lemma~\ref{lem:Expansion
0}. Indeed,
\[
h'(t) \ = \ \rho'(e^{g(t)}) e^{g(t)} g'(t)
\ = \ \psi(e^{g(t)}) g'(t)
\]
is monotone increasing in $t \in\R$. For if $s < t$, then
\begin{align}
\nonumber
\psi(e^{g(t)}) g'(t)
& - \psi(e^{g(s)}) g'(s) \\
\nonumber
&= \
\begin{cases}
\bigl( \psi(e^{g(t)}) - \psi(e^{g(s)}) \bigr) g'(s)
+ \psi(e^{g(t)}) \bigl( g'(t) - g'(s) \bigr) \\
\bigl( \psi(e^{g(t)}) - \psi(e^{g(s)}) \bigr) g'(t)
+ \psi(e^{g(s)}) \bigl( g'(t) - g'(s) \bigr)
\end{cases}
\\
\label{ineq:convexity1}
&\ge\
\begin{cases}
\bigl( \psi(e^{g(t)}) - \psi(e^{g(s)}) \bigr) g'(s) \\
\bigl( \psi(e^{g(t)}) - \psi(e^{g(s)}) \bigr) g'(t)
\end{cases}
\\
\label{ineq:convexity2}
&\ge\ 0.
\end{align}
Inequality \eqref{ineq:convexity1} follows from $\psi$ being positive
and $g'$ being non-decreasing. Inequality \eqref{ineq:convexity2}
follows from $\psi$ being non-decreasing and $g$ being convex. For if
$\psi(e^{g(t)}) - \psi(e^{g(s)}) > 0$, then $g(t) - g(s) > 0$ and
thus $g'(t) > 0$. Likewise $\psi(e^{g(t)}) - \psi(e^{g(s)}) < 0$
implies that $g(t) - g(s) < 0$ whence $g'(s) < 0$.

Concerning strict convexity, recall from Lemma~\ref{lem:Expansion 0}
that either $g'' > 0$ on $\R$, or $g'' \equiv0$ and $M \in\bigcup
_{i=1}^\ell\M(\V_i)$. Hence, in Case~0, $\psi(e^g) g' = q g'$ is
strictly increasing if, and only if, $M \not\in\bigcup_{i=1}^\ell\M
(\V_i)$. Consequently, $t \mapsto L_\rho(B\exp(tA)B^\top, Q)$ is
strictly convex if, and only if, $Q_B \bigl( \bigcup_{i=1}^\ell\M
(\V_i) \bigr) = Q \bigl( \bigcup_{i=1}^\ell\M(B\V_i) \bigr) < 1$.

In Case~1, inequality \eqref{ineq:convexity1} is strict, unless $g''
\equiv0$. But in the latter case, $g(t) = g(0) + g'(0) t$ and $g'(t) =
g'(0)$, so inequality \eqref{ineq:convexity2} is strict, unless \mbox{$g'(0)
= 0$}. Hence $h$ is strictly convex unless $g$ is constant. But this is
equivalent to saying that $M \in\M(\V_0)$. Consequently, $t \mapsto
L_\rho(B\exp(tA)B^\top, Q)$ is strictly convex, unless $Q_B(\M(\V
_0)) = Q(\M(B\V_0)) = 1$.
\end{proof}

\begin{proof}[\bf Proof of Proposition~\ref{prop:Coercivity}]
Since Conditions~0 and 1 are not affected by replacing $Q$ with $Q_B$,
we may restrict our attention to $B = I_q$. Let $\W:= \{A \in\Rqqsym
: \tr(A) = 0\}$ in Case~0 and $\W:= \Rqqsym$ in Case~1. For $A \in
\W$ and $t \in\R$ let
\[
h(t,A) \ := \ L_\rho(\exp(tA),Q).
\]
We know from Proposition~\ref{prop:Convexity} that $h$ is convex in
the first argument. Moreover, the derivative $h'(t,A) = \partial h(t,A)
/ \partial t$ is given by
\[
h'(t,A) \ = \ \tr(A) + \int
\rho' \bigl( \tr(\exp(-tA)M \bigr) \tr(- A \exp(-tA) M) \, Q(dM).
\]
This could be verified directly or derived from Proposition~\ref
{prop:Expansion 1}, because $h(t+s,\break A)  - h(t,A) = L_\rho(\exp(sA),
Q_{\exp(tA/2)})$. The derivative $h'(t,A)$ is continuous in~$A$, which
implies the following equivalence:
\begin{equation}
\label{eq:Coercivity 1}
\lim_{\|B\| \to\infty, B \in\W} \, L_\rho(\exp(B),Q)
\ = \ \infty
\end{equation}
if, and only if,
\begin{equation}
\label{eq:Coercivity 2}
h'(A) := \lim_{t \to\infty} \, h'(t,A) \ > \ 0
\quad\text{for any fixed} \ A \in\W\setminus\{0\}.
\end{equation}
To see this, note first that $h'(A) \le0$ is equivalent to $h(\cdot
,A)$ being non-increasing. Thus a violation of \eqref{eq:Coercivity 2}
would imply a violation of \eqref{eq:Coercivity 1}. Now suppose that
\eqref{eq:Coercivity 2} holds true.
Since $h'(t, A)$ is non-decreasing in $t \ge0$ and continuous in $A
\in\mathbb{S}(\W):=\{A\in\W: \|A\|=1\}$,
\[
U(t) \ := \ \bigl\{ A \in\mathbb{S}(\W) : h'(t,A) > 0 \bigr\}
\]
is an open subset of $\mathbb{S}(\W)$ with $U(s) \subset U(t)$
whenever $s < t$.
Moreover, \eqref{eq:Coercivity 2} entails that $\bigcup_{t \ge0}
U(t) = \mathbb{S}(\W)$.
But the latter set is compact, so $U(t_o) = \mathbb{S}(\W)$ for some
$t_o \ge0$.
Now for $t \ge t_o$ we have by the convexity of $h$ in the first argument,
\begin{align*}
\min_{B \in\W\,:\, \|B\| = t} \, L_\rho(\exp(B),Q) \
&= \ \min_{A \in\mathbb{S}(\W)} \, h(t,A) \\
&\ge\ \min_{A \in\mathbb{S}(\W)} \, h(t_o,A)
+ (t - t_o) \min_{A \in\mathbb{S}(\W)} \, h'(t_o,A) \\
&\to\ \infty \quad\text{as} \ t \to\infty,
\end{align*}
i.e.\ \eqref{eq:Coercivity 1} is satisfied, too.

Now we determine the limit $h'(A)$ for fixed $A \in\W\setminus\{0\}
$. To this end we write $A = - \sum_{i=1}^q \beta_i^{} u_i^{}u_i^\top
$ with $\beta_i := - \lambda_i(A)$ and an orthonormal basis $u_1,
u_2, \ldots, u_q$ of $\R^q$. Then
\[
h'(t,A) \ = \ - \sum_{i=1}^q \beta_i
+ \int\psi\Bigl( \sum_{i=1}^q u_i^\top M u_i^{} \, e^{t\beta_i}
\Bigr)
\frac{\sum_{i=1}^q \beta_i u_i^\top M u_i^{} \, e^{t\beta_i}}
{\sum_{i=1}^q u_i^\top M u_i^{} \, e^{t\beta_i}}
\, Q(dM)
\]
with $\psi(0) \cdot0/0 := 0$. As shown in the proof of
Proposition~\ref{prop:Convexity}, the integrand on the right hand side
is non-decreasing in $t \ge0$. Let $\V_0 := \{0\}$ and $\V_j :=
\mathrm{span}(u_1,\ldots,u_j)$ for $1 \le j \le q$. If $M \in\M(\V
_j) \setminus\M(\V_{j-1})$, then $u_j^\top M u_j^{} > 0 = u_k^\top M
u_k^{}$ for $j < k \le q$, and one can easily derive from $\beta_1 \le
\beta_2 \le\cdots\le\beta_q$ that
\[
\lim_{t \to\infty} \,
\psi\Bigl( \sum_{i=1}^q u_i^\top M u_i^{} \, e^{t\beta_i} \Bigr)
\frac{\sum_{i=1}^q \beta_i u_i^\top M u_i^{} \, e^{t\beta_i}}
{\sum_{i=1}^q u_i^\top M u_i^{} \, e^{t\beta_i}}
\ = \
\begin{cases}
q \beta_j & \text{in Case~0} \\
\psi(\infty) \beta_j^+ & \text{in Case~1}
\end{cases}
\]
with the usual notation $a^\pm= \max(\pm a, 0)$ for real numbers $a$.
Thus it follows from monotone convergence that
\[
h'(A) \ = \ - \sum_{i=1}^q \beta_i +
\begin{cases}
\displaystyle
q \sum_{j=1}^q \beta_j Q \bigl( \M(\V_j)\setminus\M(\V_{j-1})
\bigr)
& \text{in Case~0}, \\
\displaystyle
\psi(\infty)
\sum_{j=1}^q \beta_j^+ Q \bigl( \M(\V_j)\setminus\M(\V_{j-1})
\bigr)
& \text{in Case~1}.
\end{cases}
\]

In Case~0, define $\gamma_d := \beta_{d+1} - \beta_d$ for $d =
1,\ldots,q-1$. Then
\begin{align*}
h'(A) \
&= \ q \sum_{j=1}^q \beta_j
\bigl[ -1/q + Q(\M(\V_j)) - Q(\M(\V_{j-1})) \bigr] \\
&= \ q \sum_{j=1}^q \beta_j
\bigl[ Q(\M(\V_j)) - j/q - Q(\M(\V_{j-1})) + (j-1)/q \bigr] \\
&= \ q \sum_{j=1}^{q-1} \beta_j \bigl[ Q(\M(\V_j)) - j/q \bigr]
+ q \sum_{j=2}^{q} \beta_j \bigl[ (j-1)/q - Q(\M(\V_{j-1})) \bigr
] \\
&= \ q \sum_{d=1}^{q-1} \gamma_d \bigl[ d/q - Q(\M(\V_d)) \bigr],
\end{align*}
where we utilized that $Q(\M(\V_0)) = Q(\{0\}) = 0$ and $Q(\M(\V
_q)) = Q(\Rqqsym) = 1$. Since all $\gamma_d$ are non-negative with
$\sum_{d=1}^{q-1} \gamma_d = \beta_q - \beta_1 > 0$, Condition~0
implies clearly that $h'(A) > 0$. On the other hand, if $Q(\M(\V))
\ge j/q$ for some $\V\in\VV_q$ with $d := \dim(\V) \in[1,q)$, we
may choose the basis $u_1, u_2, \ldots, u_q$ such that $\V= \V_d$,
and with $\beta_i := 1_{[i > d]} - (q-d)/q$, the matrix $A = - \sum
_{i=1}^q \beta_i^{} u_i^{} u_i^\top$ satisfies $h'(A) = q \bigl[ d/q
- Q(\M(\V_d)) \bigr] \le0$. Consequently, \eqref{eq:Coercivity 2} and
Condition~0 are equivalent in Case~0.

In Case 1, let $\gamma_d := \beta_{d+1}^+ - \beta_d^+$ for $d = 0,
1, \ldots, q-1$, where $\beta_0^+ := 0$. Then $- \sum_{i=1}^q \beta
_i$ is equal to
\[
\sum_{i=1}^q \beta_i^- - \sum_{i=1}^q (\beta_i^+ - \beta_0^+)
\ = \ \sum_{i=1}^q \beta_i^- - \sum_{i=1}^q \sum_{d=0}^{i-1} \gamma_d
\ = \ \sum_{i=1}^q \beta_i^- - \sum_{d=0}^{q-1} \gamma_d (q - d)
\]
and $\sum_{j=1}^q \beta_j^+ Q \bigl( \M(\V_j) \setminus\M(\V
_{j-1}) \bigr)$ may be written as
\[
\sum_{j=1}^q \beta_j^+ \bigl[ 1 - Q(\M(\V_{j-1})) \bigr]
- \sum_{j=0}^{q-1} \beta_j^+ \bigl[ 1 - Q(\M(\V_j)) \bigr]
\ = \ \sum_{d=0}^{q-1} \gamma_d \bigl[ 1 - Q(\M(\V_d)) \bigr].
\]
Consequently,
\[
h'(A)
\ = \ \sum_{i=1}^q \beta_i^-
+ \sum_{d=0}^{q-1} \gamma_d
\Bigl( \psi(\infty) \bigl[ 1 - Q(\M(\V_d)) \bigr] - (q - d)
\Bigr).
\]
Again one can easily deduce from $\gamma_d \ge0$ and $\sum
_{d=0}^{q-1} \gamma_d = \beta_q^+ = \max_i \beta_i^+$ that
Condition~1 implies \eqref{eq:Coercivity 2}. On the other hand, if
$Q(\M(\V)) \ge1 - (q - d)/\psi(\infty)$ for some $\V\in\VV_q$
with $d := \dim(\V) \in[0,q)$, we may choose the basis $u_1, u_2,
\ldots, u_q$ such that $\V= \V_d$, and with $\beta_i := 1_{[i >
d]}$ we obtain a matrix $A$ such that $h'(A) \le0$. Consequently,
\eqref{eq:Coercivity 2} and Condition~1 are equivalent in Case~1.
\end{proof}

\begin{proof}[\bf Proof of Lemma~\ref{lem:psi.and.kappa}]
Suppose that Condition~\eqref{ineq:psi.and.kappa.1} is satisfied; in
other words,
\[
\partial\log\phi(t) / \partial t \ \le\ \kappa t^{-1}
\quad\text{for all} \ t > 0.
\]
Now fix arbitrary $s > 0$ and $\lambda> 1$. For any integer $\ell> 1$,
\begin{align*}
\log\phi(\lambda s) - \log\phi(s) \
&= \ \sum_{i=1}^\ell
\bigl( \log\phi(\lambda^{i/\ell}s) - \log\phi(\lambda
^{(i-1)/\ell}s) \bigr) \\
&\le\ \sum_{i=1}^\ell(\lambda^{i/\ell}s - \lambda^{(i-1)/\ell}s)
\kappa(\lambda^{(i-1)/\ell} s)^{-1} \\
&= \ \kappa\ell(\lambda^{1/\ell} - 1)
\ \to\ \kappa\log\lambda\quad\text{as} \ \ell\to\infty.
\end{align*}
Consequently, $\log\phi(\lambda s) - \log\phi(s) \le\kappa\log
\lambda$, which proves Condition~\eqref{ineq:psi.and.kappa.2}.

On the other hand, if Condition~\eqref{ineq:psi.and.kappa.2} is
satisfied, then for $s > 0$,
\[
s \phi'(s)
\ = \ \lim_{\lambda\downarrow1} \frac{\phi(\lambda s) - \phi
(s)}{\lambda- 1}
\ \le\ \lim_{\lambda\downarrow1} \frac{(\lambda^\kappa- 1) \phi
(s)}{\lambda- 1}
\ = \ \kappa\phi(s).
\]
Hence Condition~\eqref{ineq:psi.and.kappa.1} is satisfied as well.
\end{proof}

\begin{proof}[\bf Proof of Proposition~\ref{prop:Expansion 2}]
As in the proof of Proposition~\ref{prop:Expansion 1} we start from
\[
L_\rho(\exp(A),Q)
\ = \ \tr(A) + \int\bigl[ \rho(\tr(\exp(-A)M)) - \rho(\tr(M))
\bigr] \, Q(dM)
\]
and analyze for a fixed $M \in\Rqqsympsd\setminus\{0\}$ the
difference $\rho(b) - \rho(a)$, where $a := \tr(M) > 0$ and $b :=
\tr(\exp(-A)M)$.\vadjust{\goodbreak}

Recall first that $b/a \in[e^{-\|A\|}, e^{\|A\|}]$. For $x \in\R$
define $f(x) := \rho(e^x a)$. Then $f'(x) = \rho'(e^x a) e^x a = \psi
(e^x a)$, and $f''(x) = \psi'(e^x a) e^x a = \psi_2(e^x a)$.
Consequently, for a suitable point $\xi$ between $0$ and $\log(b/a)$,
\begin{align*}
\rho(b) - \rho(a) \
&= \ f(\log(b/a)) - f(0) \\
&= \ \psi(a) \log(b/a) + \psi_2(e^\xi a) \log(b/a)^2/2 \\
&= \ \psi(a) \log(b/a) + \psi_2(a) \log(b/a)^2/2 + r_2(a,b),
\end{align*}
where
\[
|r_2(a,b)| \ \le\ \sup_{z \in[-\|A\|,\|A\|]}
\bigl| \psi_2(e^z s) - \psi_2(s) \bigr| \|A\|^2/2.
\]

Now we utilize the fact that $\log(b/a) = g(1) - g(0)$ with the
auxiliary function $g(t) := \log\tr(\exp(-tA)M)$ from Lemma~\ref
{lem:Expansion 0}. In particular, $|g(1) - g(0) - g'(0)| \le\|A\|^2/2$
and $|g(1) - g(0) - g'(0) - g''(0)/2| \le\|A\|^3 (4/\sqrt{27})/6 \le
\|A\|^3/7$. Consequently,
\begin{align*}
\psi(a) \log(b/a) \
&= \ \psi(a) \bigl( g'(0) + g''(0)/2 \bigr) + r_3(a,b), \\
\psi_2(a) \log(b/a)^2/2 \
&= \ \psi_2(a) g'(0)^2/2 + r_4(a,b),
\end{align*}
where
\begin{align*}
|r_3(a,b)| \
&< \ \psi(a) \|A\|^3/7, \\
|r_4(a,b)| \
&\le\ \psi_2(a) \bigl| \log(b/a)^2 - g'(0)^2 \bigr| / 2 \\
&\le\ \kappa\psi(a)
\bigl| \log(b/a) - g'(0) \bigr| \bigl( |\log(b/a)| + |g'(0)| \bigr
)/2 \\
&\le\ \kappa\psi(a) \|A\|^3.
\end{align*}
All in all this shows that
\[
\rho(b) - \rho(a) \ = \ \psi(a) g'(0) + \psi(a) g''(0)/2
+ \psi_2(a) g'(0)^2/2 + r_*(a,b)
\]
with
\[
|r_*(a,b)|
\ \le\ \sup_{z \in[-\|A\|,\|A\|]} \bigl| \psi_2(e^z a) - \psi
_2(a) \bigr|
\|A\|^2/2
+ \psi(a) (\kappa+ 1/7) \|A\|^3.
\]
Note that $\psi(a) g'(0) = \rho'(\tr(M)) \tr(AM)$. Moreover, it
follows from $|g'| \le\|A\|$, $0 \le g'' \le\|A\|^2$ and $\psi, \psi
_2 \ge0$ that
\begin{align*}
0 \ \le\ \psi(a) g''(0) + \psi_2(a) g'(0)^2 \
&\le\ \psi(\tr(M)) \|A\|^2 + \psi_2(\tr(M)) \|A\|^2 \\
&\le\ (1 + \kappa) \psi(\tr(M)) \|A\|^2.
\end{align*}
Furthermore, elementary calculations show that
\[
\psi(a) g''(0) + \psi_2(a) g'(0)^2
\ = \ \rho'(\tr(M)) \tr(A^2M) + \rho''(\tr(M)) \tr(AM)^2.
\]
Consequently,
\[
L(\exp(A),Q)
\ = \ \langle A, G_\rho(A)\rangle+ 2^{-1} H_\rho(A,Q) + R_{\rho,2}(A,Q)
\]
with the quadratic term\begingroup\abovedisplayskip=7pt\belowdisplayskip=7pt
\[
H_\rho(A,Q)
\ = \ \int\bigl( \rho'(\tr(M)) \tr(A^2 M) + \rho''(\tr(M)) \tr
(AM)^2 \bigr)
\, Q(dM)
\]\endgroup
and a remainder $R_{\rho,2}(A,Q)$ satisfying the asserted bounds
\eqref{eq:Expansion 2a} and \eqref{eq:Expansion 2b}.

It remains to prove inequality \eqref{eq:Expansion 2c}. Since $H_\rho
(A,Q)$ is the integral of the term $\psi(a) g''(0) + \psi_2(a)
g'(0)^2 \ge0$ with $a = \tr(M)$ and $g = g(\cdot,A,M)$, it is equal
to $0$ if, and only if, $\psi(a) g''(0) + \psi_2(a) g'(0)^2$ for
$Q$-almost all $M$. Based on Lemma~\ref{lem:Expansion 0} we may argue
as follows: In Case~0, $\psi(a) g''(0) + \psi_2(a) g'(0)^2 = q
g''(0)$ equals zero if, and only if, $M \in\bigcup_{i=1}^\ell\M(\V
_i)$. Hence $H_\rho(A,Q) > 0$ is equivalent to $Q \bigl( \bigcup
_{i=1}^\ell\M(\V_i) < 1$. In Case~1, both $\psi(a)$ and $\psi
_2(a)$ are strictly positive while $g''(0) \ge0$. Hence $\psi(a)
g''(0) + \psi_2(a) g'(0)^2$ equals zero if, and only if, $g''(0) =
g'(0) = 0$, which is equivalent to $M \in\M(\V_0)$. Consequently,
$H_\rho(A,Q) > 0$ if, and only if, $Q(\M(\V_0)) < 1$.\vspace*{-3pt}
\end{proof}


\subsection{Proofs for Section~\ref{sec:Scatter 2}}\vspace*{-3pt}
\begingroup\abovedisplayskip=6pt\belowdisplayskip=6pt
In the proof of Theorem~\ref{thm:Consistency} we utilize a well-known
elementary fact about weak convergence, adapted to random distributions:

\begin{Lemma}
\label{lem:Mallows convergence}
Let $Q$ be a fixed and $\hat{Q}_1, \hat{Q}_2, \hat{Q}_3, \ldots$ be
random probability distributions on a metric space $(\mathbb{Y},d)$
with the following two properties: For any bounded and continuous
function $f : \mathbb{Y} \to\R$,
\[
\int f \, d \hat{Q}_n \ \to_p \ \int f \, dQ.
\]
Further, for a particular continuous function $\phi: \mathbb{Y} \to
[0,\infty)$, $\int\phi\, d\hat{Q}_n < \infty$ almost surely for
all $n$, and
\[
\int\phi\, d\hat{Q}_n \ \to_p \ \int\phi\, dQ < \infty.
\]
Then
\[
\int f \, d\hat{Q}_n \ \to_p \ \int f \, dQ
\]
for any continuous function $f : \mathbb{Y} \to\R$ such that $|f|/(1
+ \phi)$ is bounded on $\mathbb{Y}$.
\end{Lemma}

\begin{proof}[\bf Proof of Lemma~\ref{lem:Mallows convergence}]
It suffices to consider any continuous function $f : \mathbb{Y} \to\R
$ such that $|f| \le\tilde{\phi} := 1 + \phi$. For any fixed number
$R \ge1$ let
\[
f_R(y) \ := \ \mathrm{sign}(f(y)) \min\{|f(y)|,R\}.
\]
Then
\begin{align*}
\Bigl| \int f \, d\hat{Q}_n - \int f \, dQ \Bigr| \
\le\ &\int|f - f_R| \, d\hat{Q}_n + \int|f - f_R| \, dQ \\
&+ \ \Bigl| \int f_R \, d\hat{Q}_n - \int f_R \, dQ \Bigr| \\
= \ &\int|f - f_R| \, d\hat{Q}_n + \int|f - f_R| \, dQ
+ o_p(1)
\end{align*}
by our first assumption. But $|f - f_R| = (|f| - R)^+ \le(\tilde{\phi
} - R)^+ = (\phi- R + 1)^+$,~so
\begin{align*}
\int|f - f_R| \, d\hat{Q}_n \
&\le\ \int(\phi- R + 1)^+ \, d\hat{Q}_n \\
&= \ \int\phi\, d\hat{Q}_n - \int\min\{\phi,R-1\} \, d\hat{Q}_n
\\
&\to_p \ \int\phi\, dQ
- \int\min\{\phi,R-1\} \, dQ
\ = \ \int(\phi- R + 1)^+ \, dQ
\end{align*}
by our assumptions. Consequently,
\[
\Bigl| \int f \, d\hat{Q}_n - \int f \, dQ \Bigr|
\ \le\ 2 \int(\phi- R + 1)^+ \, dQ
+ o_p(1),
\]
and the integral on the right hand is arbitrarily small for
sufficiently large $R$.
\end{proof}\endgroup

\begin{proof}[\bf Proof of Theorem~\ref{thm:Consistency}]
By linear equivariance we may assume without loss of generality that
$\bSigma_\rho(Q) = I_q$. Let $\W:= \{A \in\Rqqsym: \tr(A) = 0\}$
in Case~0, and $\W:= \Rqqsym$ in Case~1. For any fixed $\delta> 0$,
the set $K_\delta:= \{A \in\W: \|A\| \le\delta\}$ is compact, and
for $A \in K_\delta$,
\[
f(A,M)
\ := \ \tr(A)
+ \bigl[ \rho\bigl( \tr(\exp(-A) M) \bigr) - \rho(\tr(M)) \bigr]
\]
is continuous in $M \in\mathbb{Y}$ with
\[
|f(A,M)| \ \le\ q \delta+ \psi(e^\delta\tr(M)) \delta
\]
by Lemmas~\ref{lem:Trace inequalities} and \ref{lem:Expansion of
rho}. If $\delta$ is sufficiently small, $\psi(e^\delta\tr(M)) \le
\psi(\tr(\Sigma_o^{-1} M))$ for any $M \in\mathbb{Y}$. Then it
follows from Lemma~\ref{lem:Mallows convergence} that
\begin{align*}
L_\rho(\exp(A),Q_n) \
&= \ \int f(A,M), \hat{Q}_n(dM) \\
&\to_p \ \int f(A,M) \, Q(dM)
\ = \ L_\rho(\exp(A),Q)
\end{align*}
for any fixed $A \in K_\delta$. Moreover it follows from
Corollary~\ref{cor:Diff.and.Lipschitz} and the first part of
Lemma~\ref{lem:Taylor of exp} that
\begin{align*}
\bigl| L_\rho(\exp(A),\hat{Q}_n) - L_\rho(\exp(B),\hat{Q}_n)
\bigr| \
&\le\ J(e^\delta,\hat{Q}_n) e^\delta\|\exp(A) - \exp(B)\| \\
&\le\ J(e^\delta,\hat{Q}_n) e^{4\delta} \|A - B\|\vadjust{\goodbreak}
\end{align*}
for $A,B \in K_\delta$, and the Lipschitz constant $J(e^\delta,\hat
{Q}_n) e^{4\delta}$ converges to $J(e^\delta,Q) e^{4\delta}$ in
probability. This implies that
\[
\max_{A \in K_\delta} \,
\bigl| L_\rho(\exp(A),\hat{Q}_n) - L_\rho(\exp(A),Q) \bigr|
\ \to_p \ 0.
\]
In particular,
\[
\epsilon_n(\delta) := \min_{A \in\W: \|A\| = \delta} L_\rho(\exp
(A),\hat{Q}_n)
\ \to_p \
\epsilon(\delta) := \min_{A \in\W: \|A\| = \delta} L(\exp(A),Q)
> 0.\vadjust{\goodbreak}
\]
Whenever $\epsilon_n(\delta) > 0$, we may conclude from
Proposition~\ref{prop:Convexity} the inequality $L_\rho(\exp(A),\hat
{Q}_n) \ge\epsilon_n(\delta) \|A\|/\delta$ for all $A \in\W$ with
$\|A\| \ge\delta$. This shows that $L_\rho(\exp(A),\hat{Q}_n) \to
\infty$ as $\|A\| \to\infty$, so $\hat{Q}_n \in\QQ_\rho$ by
Proposition~\ref{prop:Coercivity} and Theorem~\ref{thm:Uniqueness}.
Moreover, since $L_\rho(\exp(0),\hat{Q}_n) = 0$, we may conclude
that $\bSigma_\rho(\hat{Q}_n) \in\{\exp(A) : A \in K_\delta\}$.
\end{proof}

\begin{proof}[\bf Proof of Theorem~\ref{thm:Diffability}]
According to Theorem~\ref{thm:Consistency}, $\hat{Q}_n \in\QQ_\rho
$ with asymptotic probability one. Thus we may replace $\LL(\hat
{Q}_n)$ with $\LL(\hat{Q}_n \,|\, \hat{Q}_n \in\QQ_\rho)$ and
thus assume that $\hat{Q}_n \in\QQ_n$ almost surely.

As in earlier proofs we define $\W:= \Rqqsym$ in Case~1' and $\W:= \{
A \in\Rqqsym: \tr(A) = 0\}$ in Case~0. Since $G_\rho(\hat{Q}_n)
\in\W$, and since $H_\rho(\hat{Q}_n)$ is a selfadjoint linear
operator on the finite-dimensional space $\W$, both $\|G_\rho(\hat
{Q}_n)\|$ and
\[
\bigl\| H_\rho(\hat{Q}_n) - H_\rho(Q) \bigr\|
\ := \ \max_{A \in\W\,:\, \|A\| \le1}
\bigl\| H_\rho(\hat{Q}_n)A - H_\rho(Q)A \bigr\|
\]
converge to $0$ in probability if, and only if, for arbitrary fixed
$A,B \in\W$,
\[
\langle A, G_\rho(\hat{Q}_n)\rangle
\ \to_p \ \langle A, G_\rho(Q)\rangle= 0
\quad\text{and}\quad
\langle A, H_\rho(\hat{Q}_n)B\rangle
\ \to_p \ \langle A, H_\rho(Q)B\rangle.
\]
But this is a consequence of Lemma~\ref{lem:Mallows convergence}: We
may write $\langle A, G_\rho(\tilde{Q})\rangle= \int g \, d\tilde
{Q}$ and $\langle A, H_\rho(\tilde{Q})B\rangle= \int h \, d\tilde
{Q}$ with
\begin{align*}
g(M) \
&:= \ \tr(A) - \rho'(\tr(M)) \tr(AM), \\
h(M) \
&:= \ \rho'(\tr(M)) \tr(ABM) + \rho''(\tr(M)) \tr(AM) \tr(BM).
\end{align*}
Both $g(M)$ and $h(M)$ are continuous in $M \in\mathbb{Y}$ and satisfy
\begin{align*}
|g(M)| \
&\le\ (q + \psi(\tr(M)) \|A\|, \\
|h(M)| \
&\le\ \bigl( \psi(\tr(M)) + \tr(M)^2 |\rho''(\tr(M))| \bigr) \|
A\|\|B\| \\
&\le\ (2 + \kappa) \psi(\tr(M)) \|A\|\|B\|,
\end{align*}
whence $\int g \, d\hat{Q}_n \to_p \int g \, dQ$ and $\int h \, d\hat
{Q}_n \to_p \int h \, dQ$.

In particular we may conclude that there exist numbers $\delta_n > 0$
such that $\delta_n \to0$ and $\Pr\bigl( \|G_\rho(\hat{Q}_n)\| >
\delta_n \bigr) \to0$. Moreover, with asymptotic probability one,
$H_\rho(\hat{Q}_n)$ is positive definite.

Now we consider $L_\rho(\exp(A),\hat{Q}_n)$ for $A \in\W$ with $\|
A\| \le\sqrt{\delta_n}$: According to Proposition~\ref
{prop:Expansion 2},
\[
L_\rho(\exp(A),\hat{Q}_n)
\ = \ \langle A, G_\rho(\hat{Q}_n)\rangle
+ 2^{-1} H_\rho(A,\hat{Q}_n) + R_{\rho,2}(A,\hat{Q}_n).
\]
But it follows from Proposition~\ref{prop:Expansion 2} that for any
fixed $\delta> 0$,
\[
\sup_{A \in\W: 0 < \|A\| \le\sqrt{\delta_n}}
\frac{|R_{\rho,2}(A,\hat{Q}_n)|}{\|A\|^2}
\ \le\ \Omega(\delta,\hat{Q}_n)/2
+ (\kappa+ 1/7) J(\hat{Q}_n) \delta
\]
as soon as $\sqrt{\delta_n} \le\delta$. But $\Omega(\delta,\tilde
{Q})/2 + (\kappa+ 1/7) J(\tilde{Q}) = \int f_\delta\, d\tilde{Q}$ with
\[
f_\delta(M) \ := \ \sup_{z \in[-\delta,\delta]}
\bigl| \psi_2(e^z \tr(M)) - \psi_2(\tr(M)) \bigr| / 2
+ (\kappa+ 1/7) \psi(\tr(M)) \delta.
\]
This is continuous in $M \in\mathbb{Y}$, and
\[
0 \ \le\ f_\delta(M) \ \le\ (3\kappa/2 + 1/7) \psi(e^\delta\tr(M))
\ \le\ (3\kappa/2 + 1/7) e^{\kappa\delta} \psi(\tr(M)).
\]
Hence we may conclude from Lemma~\ref{lem:Mallows convergence} that
\[
\sup_{A \in\W: 0 < \|A\| \le\sqrt{\delta_n}}
\frac{|R_{\rho,2}(A,\hat{Q}_n)|}{\|A\|^2}
\ \le\ \int f_\delta\, dQ + o_p(1).
\]
But the right hand side converges to $0$ as $\delta\to0$, because
$f_\delta(M) \downarrow 0$ as $\delta\downarrow0$ for any $M \in
\mathbb{Y}$. Hence the left hand side converges to $0$ in probability.

Together with our considerations about $H_\rho(\hat{Q}_n)$ we obtain
the following expansion:
\[
L_\rho(\exp(A),\hat{Q}_n)
\ = \ \langle A, G_\rho(\hat{Q}_n)\rangle
+ 2^{-1} \langle A, H_\rho(Q)A\rangle
+ \hat{\gamma}_n(A) \|A\|^2
\]
where
\[
\hat{\Gamma}_n := \sup_{A \in\W: \|A\| \le\sqrt{\delta_n}}
|\hat{\gamma}_n(A)|
\ \to_p \ 0.
\]
Now we define
\[
\hat{A}_n \ := \ - H_\rho(Q)^{-1} G_\rho(\hat{Q}_n)
\]
and note that $c(Q) \|G_\rho(\hat{Q}_n)\| \le\|\hat{A}_n\| \le C(Q)
\|G_\rho(\hat{Q}_n)\|$ for suitable constants $0 < c(Q) < C(Q)$. If
$\hat{A}_n = 0$, then $\bSigma_\rho(\hat{Q}_n) = I_q$, i.e.\ $\log
(\bSigma_\rho(\hat{Q}_n)) = 0$. Thus we focus on the event $\hat
{A}_n \ne0$. We fix an arbitrary number $\epsilon\in(0,1)$. For any
matrix $A \in\W$ with $\|A - \hat{A}_n\| = \epsilon\|\hat{A}_n\|$,
\begin{align*}
L_\rho(&\exp(A),\hat{Q}_n) - L(\exp(\hat{A}_n),\hat{Q}_n) \\
&= \ 2^{-1} \bigl\langle A - \hat{A}_n, H_\rho(Q)(A - \hat{A}_n)
\bigr\rangle
+ \hat{\gamma}_n(A) \|A\|^2 - \hat{\gamma}_n(\hat{A}_n) \|\hat
{A}_n\|^2.
\end{align*}
Note that $\|A\| \le2 \|\hat{A}_n\|$, and $2 \|\hat{A}_n\| \le\sqrt
{\delta_n}$ with asymptotic probability one. In case of $2 \|\hat
{A}_n\| \le\sqrt{\delta_n}$,
\begin{align*}
\inf_{A \in\W: \|A - \hat{A}_n\| = \epsilon\|\hat{A}_n\|}
&\bigl( L_\rho(\exp(A),\hat{Q}_n) - L(\exp(\hat{A}_n),\hat{Q}_n)
\bigr) \\
&\ge\ \bigl( 2^{-1} \lambda_{\rm min}(H_\rho(Q)) \epsilon^2
- 5 \hat{\Gamma}_n \bigr) \|\hat{A}_n\|^2 \\
&= \ \bigl( 2^{-1} \lambda_{\rm min}(H_\rho(Q)) \epsilon^2
+ o_p(1) \bigr) \|\hat{A}_n\|^2.
\end{align*}
Whenever the right hand side is strictly positive, we may conclude that
\[
\bigl\| \log(\bSigma_\rho(\hat{Q}_n)) - \hat{A}_n \bigr\|
\ \le\ \epsilon\|\hat{A}_n\|
\ \le\ \epsilon C(Q) \|G_\rho(\hat{Q}_n)\|.
\]
These considerations show that $\bigl\| \log(\bSigma_\rho(\hat
{Q}_n)) - \hat{A}_n \bigr\| \le\epsilon C(Q) \|G_\rho(\hat{Q}_n)\|
$ with as\-ymptotic probability one. Since $\epsilon> 0$ is
arbitrarily small, this proves that $\log(\bSigma_\rho(\hat{Q}_n))$
equals $\hat{A}_n + o_p \bigl( \|G_\rho(\hat{Q}_n)\| \bigr)$.
\end{proof}

The proof of Lemma~\ref{lem:OrthInvQ} relies on the following two
propositions involving the Haar distribution on the set of orthogonal
matrices in $\Rqq$. A good reference for Haar distributions in general
is the monograph by Eaton (\citeyear{Eaton_1989}).

\begin{Proposition}
\label{prop:U}
Let $U \in\Rqq$ be a random orthogonal matrix with Haar distribution,
i.e.\ $\LL(U) = \LL(U^\top) = \LL(VU)$ for any fixed orthogonal
matrix $V$. Then for arbitrary indices $i,j,k,\ell,k',\ell' \in\{
1,2,\ldots,q\}$,
\begin{align}
\label{eq:U00}
\Ex(U_{ij}^2 U_{k\ell} U_{k'\ell'})
&= \ 0
\quad\text{if} \ (k,\ell) \ne(k',\ell'), \\
\label{eq:U0}
\Ex(U_{ij}^4) \
&= \ c_{q,0}^{} := \frac{3}{q(q+2)}, \\
\label{eq:U1}
\Ex(U_{ij}^2 U_{i\ell}^2)
= \Ex(U_{ji}^2 U_{\ell i}^2) \
&= \ c_{q,1}^{} := \frac{1}{q(q+2)}
\quad\text{if} \ j \ne\ell, \\
\label{eq:U2}
\Ex(U_{ij}^2 U_{k\ell}^2)
&= \ c_{q,2}^{} := \frac{q+1}{(q-1)q(q+2)}
\quad\text{if} \ i \ne k, j \ne\ell.
\end{align}
\end{Proposition}

\begin{Proposition}
\label{prop:M}
Let $M = U \diag(\lambda) U^\top$ with a fixed vector $\lambda\in
[0,\infty)^q$ and a random orthogonal matrix $U$ as in
Proposition~\ref{prop:U}. Then for any matrix $A = A_0 + A_1$ with
$A_0 \in\W_0, A_1 \in\W_1$,
\[
\Ex(\tr(AM) M) \ = \ c_0(\lambda) A_0 + c_1(\lambda) A_1,
\]
where
\[
c_0(\lambda)
\ = \ \frac{2}{q(q+2)}
\Bigl( \|\lambda\|^2 - \frac{\lambda_+^2 - \|\lambda\|^2}{q-1}
\Bigr)
\quad\text{and}\quad
c_1(\lambda)
\ = \ \frac{\lambda_+^2}{q}
\]
and $\lambda_+ := \sum_{i=1}^q \lambda_i$.
\end{Proposition}

\begin{proof}[\bf Proof of Proposition~\ref{prop:U}]
By assumption, $U$ has the same distribution as the random matrix
$\tilde{U} = (\xi_i \zeta_j U_{ij})_{i,j=1}^q$, where $U$, $\xi$
and $\zeta$ are independent with distribution $\xi,\zeta\sim\mathrm
{Unif}(\{-1,1\}^q)$. Hence $U_{ij}^2 U_{k\ell} U_{k'\ell'}$ has the
same distribution as the random product $U_{ij}^2 U_{k\ell} U_{k'\ell
'} \xi_k \xi_{k'} \zeta_\ell\zeta_{\ell'}$. In case of $(k,\ell)
\ne(k',\ell')$, the factor $\xi_k \xi_{k'} \zeta_\ell\zeta_{\ell
'}$ is a random sign, and this implies \eqref{eq:U00}.

As to the remaining equations, note that $U$ has the same distribution
as $U^\top$ and as $\tilde{U} = (U_{\pi(i)\sigma(j)})_{i,j=1}^q$
for arbitrary permutations $\pi, \sigma$ of $\{1,2,\ldots,q\}$.
Hence it suffices to show that
\begin{align}
\label{eq:U0'}
\Ex(U_{11}^4) \
&= \ \frac{3}{q(q+2)}, \\
\label{eq:U1'}
\Ex(U_{11}^2 U_{12}^2)
&= \ \frac{1}{q(q+2)}, \\
\label{eq:U2'}
\Ex(U_{11}^2 U_{22}^2)
&= \ \frac{q+1}{(q-1)q(q+2)}.
\end{align}
Any row or column of $U$ is uniformly distributed on the unit sphere of
$\R^q$, and this implies that $U_{11}^2 \sim\mathrm{Beta}(a,b)$ with
$a = 1/2$, $b = (q-1)/2$. Hence \eqref{eq:U0'} follows from
\[
\Ex(U_{11}^4)
\ = \frac{a(a+1)}{(a+b)(a+b+1)}
\ = \ \frac{3}{q(q+2)}.
\]
Now we utilize the fact that all rows of $U$ are unit vectors. Hence
\begin{align*}
1
\ = \ \Ex\Bigl( \Bigl( \sum_{j=1}^q U_{1j}^2 \Bigr)^2 \Bigr)
\ = \ \sum_{j,\ell=1}^q \Ex(U_{1j}^2 U_{1\ell}^2)
\
&= \ q \Ex(U_{11}^4) + q(q-1) \Ex(U_{11}^2 U_{12}^2) \\
&= \ \frac{3}{q+2} + q(q-1) \Ex(U_{11}^2 U_{12}^2),
\end{align*}
so
\[
\Ex(U_{11}^2 U_{12}^2)
\ = \ \frac{1 - 3/(q+2)}{q(q-1)}
\ = \ \frac{1}{q(q+2)},
\]
which is \eqref{eq:U1'}. Similarly we deduce \eqref{eq:U2'}:
\begin{align*}
1
\ = \ \Ex\Bigl( \sum_{j=1}^q U_{1j}^2 \sum_{\ell=1}^q U_{2\ell}^2
\Bigr)
\ &= \ \sum_{j,\ell=1}^q \Ex(U_{1j}^2 U_{2\ell}^2)\\[2pt]
\
&= \ q \Ex(U_{11}^2 U_{12}^2) + q(q-1) \Ex(U_{11}^2 U_{22}^2) \\[2pt]
&= \ \frac{1}{q+2} + q(q-1) \Ex(U_{11}^2 U_{22}^2),
\end{align*}
so
\[
\Ex(U_{11}^2 U_{22}^2)
\ = \ \frac{1 - 1/(q+2)}{q(q-1)}
\ = \ \frac{q+1}{(q-1)q(q+2)}.
\]
\\[-5ex]
\end{proof}

\begin{proof}[\bf Proof of Proposition~\ref{prop:M}]
Suppose first that $A = \diag(a)$ for some $a \in\R^q$. Denoting the
columns of $U$ with $U_1, U_2, \ldots, U_q$, we may write
\begin{align*}
\Ex(\tr(AM) M) \
&= \ \sum_{j=1}^q \lambda_j \Ex(U_j^\top A U_j^{} U \diag(\lambda)
U^\top) \\[2pt]
&= \ \sum_{i,j=1}^q a_i \lambda_j \Ex(U_{ij}^2 U \diag(\lambda)
U^\top) \\[2pt]
&= \ \sum_{i,j,\ell=1}^q a_i \lambda_j \lambda_\ell
\Ex\bigl( U_{ij}^2 (U_{k\ell} U_{k'\ell})_{k,k'=1}^q \bigr).
\end{align*}
It follows from Proposition~\ref{prop:U} that
\begin{align*}
\Ex\bigl(
&U_{ij}^2 (U_{k\ell} U_{k'\ell})_{k,k'=1}^q \bigr) \\[2pt]
&= \ \diag\Bigl( \bigl( \Ex(U_{ij}^2 U_{k\ell}^2 \bigr)_{k=1}^q
\Bigr) \\[2pt]
&= \ \diag\Bigl(
\bigl( 1_{[i=k, j=\ell]}^{} c_{q,0}^{}
+ 1_{[i=k, j\ne\ell]}^{} c_{q,1}^{}
+ 1_{[i\ne k,j=\ell]}^{} c_{q,1}^{}
+ 1_{[i\ne k,j\ne\ell]}^{} c_{q,2}^{} \bigr)_{k=1}^q \Bigr).
\end{align*}
Consequently,
\[
\Ex(\tr(AM) M) \ = \ \diag(\gamma_1, \gamma_2, \ldots, \gamma_q)
\]
with $\gamma_k$ given by\vadjust{\eject}
\begin{align*}
&\sum_{i,j,\ell=1}^q a_i \lambda_j \lambda_\ell
\bigl( 1_{[i=k, j=\ell]}^{} c_{q,0}^{}
+ 1_{[i=k, j\ne\ell]}^{} c_{q,1}^{}
+ 1_{[i\ne k,j=\ell]}^{} c_{q,1}^{}
+ 1_{[i\ne k,j\ne\ell]}^{} c_{q,2}^{} \bigr) \\
&\quad{}= a_k \|\lambda\|^2 c_{q,0}^{}
+ a_k (\lambda_+^2 - \|\lambda\|^2) c_{q,1}^{} \\
&\qquad{}
+ (q \bar{a} - a_k) \|\lambda\|^2 c_{q,1}^{}
+ (q \bar{a} - a_k) (\lambda_+^2 - \|\lambda\|^2) c_{q,2}^{} \\
&\quad{}= \bigl( \|\lambda\|^2 (c_{q,0}^{} - c_{q,1}^{})
+ (\lambda_+^2 - \|\lambda\|^2) (c_{q,1}^{} - c_{q,2}^{}) \bigr)
\cdot a_k \\
&\qquad{}
+ \bigl( \|\lambda\|^2 q c_{q,1}^{}
+ (\lambda_+^2 - \|\lambda\|^2) q c_{q,2}^{} \bigr)
\cdot\bar{a} \\
&\quad{}=  \bigl( \|\lambda\|^2 (c_{q,0}^{} - c_{q,1}^{})
+ (\lambda_+^2 - \|\lambda\|^2) (c_{q,1}^{} - c_{q,2}^{}) \bigr)
\cdot(a_k - \bar{a}) \\
&\qquad{}
+ \bigl( \|\lambda\|^2 (c_{q,0}^{} + (q-1) c_{q,1}^{})
+ (\lambda_+^2 - \|\lambda\|^2) (c_{q,1}^{} + (q-1) c_{q,2}^{}) \bigr)
\cdot\bar{a} \\
&\quad{}=  \frac{2}{q(q+2)}
\Bigl( \|\lambda\|^2 - \frac{\lambda_+^2 - \|\lambda\|^2}{q-1}
\Bigr)
\cdot(a_k - \bar{a})
+ \frac{\lambda_+^2}{q} \cdot\bar{a},
\end{align*}
where $\lambda_+ := \sum_{i=1}^q \lambda_i$ and $\bar{a} := q^{-1}
\sum_{i=1}^q a_i$. Hence
\[
\Ex(\tr(AM) M) \ = \ c_0(\lambda) \diag((a_k - \bar{a})_{k=1}^q)
+ c_1(\lambda) \bar{a} I_q
\]
with $c_0(\lambda), c_1(\lambda)$ as stated.

In general let $A = V \diag(a) V^\top$ with an orthogonal matrix $V
\in\Rqq$. Then $A_0 = V \diag((a_k - \bar{a})_{k=1}^q) V^\top$ and
$A_1 = \bar{a} I_q$, so
\begin{align*}
\Ex(\tr(AM) M) \
&= \ V \Ex\bigl( \tr(\diag(a) V^\top M V) V^\top M V \bigr) V^\top
\\
&= \ V \bigl( c_0(\lambda) \diag((a_k - \bar{a})_{k=1}^q)
+ c_1(\lambda) \bar{a} I_q \bigr) V^\top\\
&= \ c_0(\lambda) A_0 + c_1(\lambda) A_1,
\end{align*}
because $\LL(V^\top M V) = \LL\bigl( (V^\top U) \diag(\lambda)
(V^\top U)^\top\bigr) = \LL(M)$.
\end{proof}

\begin{proof}[\bf Proof of Lemma~\ref{lem:OrthInvQ}]
Let $M \sim Q$ and $U$ be independent, where $U$ is a random orthogonal
matrix as in Proposition~\ref{prop:U}. If we write $M = V \diag
(\Lambda) V^T$ with a random orthogonal matrix $V \in\Rqq$ and a
random vector $\Lambda\in[0,\infty)^q$, then
\[
\LL(M)
= \LL(U V \diag(\Lambda) V^\top U^\top)
= \LL\bigl( (UV) \diag(\Lambda) (UV)^\top\bigr)
= \LL(U \diag(\Lambda) U^\top),
\]
where the first step follows from orthogonal invariance of $Q$ and the
last step follows after conditioning on $(\Lambda,V)$ and utilizing
the fact that $\LL(UV) = \LL(U)$. Consequently, we may and do assume
that $M = U \diag(\Lambda) U^\top$. Then, by Proposition~\ref{prop:M},
\begin{align*}
H_\rho(Q) A \
&= \ A + \Ex\bigl( \rho''(\tr(M)) \tr(AM) M \bigr) \\
&= \ A + \Ex\bigl( \rho''(\Lambda_+) \tr(AM) M \bigr) \\
&= \ A + \Ex\bigl( \rho''(\Lambda_+) \Ex(\tr(AM) M \,|\, \Lambda
) \bigr) \\
&= \ A + \Ex\bigl( \rho''(\Lambda_+)
\bigl( c_0(\Lambda) A_0 + c_1(\Lambda) A_1 \bigr) \bigr) \\
&= \ \bigl( 1 + \Ex\bigl( \rho''(\Lambda_+) c_0(\Lambda) \bigr)
\bigr) A_0
+ \bigl( 1 + \Ex\bigl( \rho''(\Lambda_+) c_1(\Lambda) \bigr)
\bigr) A_1.
\end{align*}
Now the assertion follows from the explicit formula for $c_0(\Lambda),
c_1(\Lambda)$ and the fact that $\Lambda_+ = \tr(M)$ and $\|\Lambda
\|^2 = \|M\|_F^2$.
\end{proof}

\begin{proof}[\bf Proof of Theorem~\ref{thm:Consistency S12}]
Note first that the nonrandom distributions $Q_n$ satisfy the
conditions of Theorem~\ref{thm:Continuity}: It follows from $P_n \to
_w P$ that $P_n^{\otimes k} = \LL(X_{n1},\ldots,X_{nk})$ converges
weakly to $P^{\otimes k} = \LL(X_1,\ldots,X_k)$, where $X_1,\ldots
,X_k$ are independent with distribution $P$. Since the mappings $\R^q
\ni x \mapsto xx^\top\in\Rqqsympsd$ and $(\R^q)^\ell\ni
(x_1,\ldots,x_\ell) \mapsto S(x_1,\ldots,x_\ell)$, $\ell\ge2$,
are continuous, $Q_n \to_w Q$ by the Continuous Mapping Theorem. As to
Condition~\eqref{eq:Continuity}, note first that for $x \in\R^q$,
\[
\psi(\lambda_o \tr(xx^\top))
\ \le\ \lambda_o^\kappa\psi(\|x\|^2)
\]
and for $\ell\ge2$ points $x_1,\ldots,x_\ell\in\R^q$,
\[
\psi\bigl( \lambda_o \tr(S(x_1,\ldots,x_\ell)) \bigr)
\ \le\ \lambda_o^\kappa(1 - 1/\ell)^{-\kappa}
\sum_{i=1}^\ell\psi(\|x_\ell\|^2),
\]
see also the derivation of \eqref{eq:From.P.to.Q} and Lemma~\ref
{lem:psi.and.kappa}. Hence we may apply Lemma~\ref{lem:Mallows
convergence} with the non-random triple $\bigl( (\R^q)^k,
P_n^{\otimes k}, P^{\otimes k} \bigr)$ in place of $(\mathbb{Y}, \hat
{Q}_n, Q)$ and the function $\phi(x_1,\ldots,x_k) := \sum_{i=1}^k
\psi(\|x_i\|^2)$ to show that under our additional assumptions with $m
= 1$,
\[
\int\psi(\lambda_o \tr(M)) \, Q_n(dM)
\ \to\ \int\psi(\lambda_o \tr(M)) \, Q(dM).
\]

Now we show that the random distributions $\hat{Q}_n$ satisfy
Conditions~\eqref{eq:Consistency 1} and \eqref{eq:Consistency 2} in
Theorem~\ref{thm:Consistency}. Because of the preceding considerations
for $(Q_n)_n$, it suffices to show that
\begin{equation}
\label{eq:Consistency S12}
\Ex\Bigl| \int g \, d(\hat{Q}_n - Q_n) \Bigr| \ \to\ 0
\end{equation}
whenever $g : \mathbb{Y} \to\R$ is a bounded measurable function or
$g(M) = \phi(M) := \psi(\lambda_o \tr(M))$.

In both cases the expected value of $\int g \, d\hat{Q}_n$ equals
$\int g \, dQ_n \in\R$. Consequently, if $g$ is bounded, then
\[
\Ex\Bigl| \int g \, d(\hat{Q}_n - Q_n) \Bigr|
\ \le\ \Bigl( \Var\Bigl( \int g \, d\hat{Q}_n \Bigr) \Bigr)^{1/2}
\ \le\ \|g\|_\infty/ \sqrt{n/k}.
\]
In case of $k = 1$, the latter inequality follows from the well-known identity
\[
\Var\bigl( \int g \, d\hat{Q}_n \bigr)
\ = \ \Var(g(X_{n1}^{}X_{n1}^\top)) / n
\ \le\ \|g\|_\infty^2/n.
\]
For $k \ge2$ it follows from inequalities by Hoeffding (\citeyear
{Hoeffding_1948}) for $U$-statistics, see also Dudley (\citeyear
{Dudley_2002}, Section~11.9). This proves \eqref{eq:Consistency S12}
for bounded $g$.

In case of $g = \phi$ we fix an arbitrary $R > 0$ and write
\begin{align*}
\Ex\Bigl| \int\phi\, d(\hat{Q}_n - Q_n) \Bigr| \
&\le\ 2 \int(\phi- R)^+ \, dQ_n
+ \Ex\Bigl| \int\min\{\phi,R\} \, d(\hat{Q}_n - Q_n) \Bigr| \\
&\le\ 2 \int(\phi- R)^+ \, dQ_n
+ R / \sqrt{n/k} \\
&\to\ 2 \int(\phi- R)^+ \, dQ,
\end{align*}
because $(\phi- R)^+ = \phi- \min\{\phi,R\}$. This implies
Condition~\eqref{eq:Consistency 2}, because the limit $\int(\phi-
R)^+ \, dQ$ tends to $0$ as $R \to\infty$.
\end{proof}

\begin{proof}[\bf Proof of Theorem~\ref{thm:CLT1}]
As in the proof of Theorem~\ref{thm:Consistency S12} it can be shown that
\[
\int\psi(\tr(M))^\ell\, Q_n(dM)
\ \to\ \int\psi(\tr(M))^\ell\, Q(dM)
\quad\text{for} \ \ell= 1,2,
\]
and that the random distributions $\hat{Q}_n$ satisfy
Conditions~\eqref{eq:Consistency 1} and \eqref{eq:Consistency 2'}.
Hence Theorem~\ref{thm:Diffability} implies that $\hat{Q}_n \in\QQ
_\rho$ with asymptotic probability one, and
\[
\sqrt{n} \log(\bSigma_{\rho}(\hat{Q}_n))
\ = \ H_\rho(Q)^{-1} \bigl( - \sqrt{n} G_\rho(\hat{Q}_n) \bigr)
+ o_p \bigl( \sqrt{n} \|G_\rho(\hat{Q}_n)\| \bigr).
\]
Thus we have to analyze the random matrix
\[
\tilde{W}_n \ := \ - \sqrt{n} G_\rho(\hat{Q}_n)
\ = \ \sqrt{n} \int\bigl( \rho'(\tr(M)) M - I_q \bigr) \, \hat{Q}_n(dM)
\ \in\ \W
\]
in more detail.

In case of $k = 1$ the random matrix $\tilde{W}_n$ equals
\[
\frac{1}{\sqrt{n}} \sum_{i=1}^n \tilde{Z}(X_{ni})
\quad\text{with}\quad
\tilde{Z}(x) \ := \ \rho'(\|x\|^2) xx^\top- I_q
\]
for $x \in\mathbb{X}$. Here $\Ex\tilde{Z}(X_{n1}) = G_\rho(Q_n) =
0$ and $\|\tilde{Z}(\cdot)\|_F \le\psi(\|x\|^2) + \sqrt{q}$. This
implies that $\tilde{W}_n = O_p(1)$. Moreover, continuity of $\rho'$
on $(0,\infty)$ and of $\psi$ on $[0,\infty)$ in Case~1' with $\psi
(0) = 0$ implies that $\tilde{Z} : \mathbb{X} \to\Rqqsym$ is continuous.

In case of $k \ge2$ we may write
\[
\tilde{W}_n
\ = \ \sqrt{n} \binom{n}{k}^{-1} \sum_{1 \le i_1 < \cdots< i_k \le n}
M(X_{ni_1},\ldots,X_{ni_k})
\]
with
\[
M(x_1,\ldots,x_k)
\ := \ \rho' \bigl( \tr(S(x_1,\ldots,x_k)) \bigr) S(x_1,\ldots
,x_k) - I_q.
\]
In Case~0, we define $M(x_1,\ldots,x_k) := 0$ whenever $S(x_1,\ldots
,x_k) = 0$. Here
\begin{align}
\nonumber
\| M(x_1,\ldots,x_k) \|_F \
&\le\ \psi\bigl( \tr(S(x_1,\ldots,x_k)) \bigr) + \sqrt{q} \\
\label{ineq:CLT1}
&\le\ (k/(k-1))^\kappa\sum_{i=1}^k \|x_i\|^2 + \sqrt{q},
\end{align}
and $\Ex M(X_{n1},\ldots,X_{nk}) = G_\rho(Q_n) = 0$. Hence standard
considerations for \mbox{$U$-statistics} as in Dudley (\citeyear
{Dudley_2002}, Section~11.9), with straightforward extensions to
vector- or matrix-valued ones, imply that
\[
\tilde{W}_n \ = \ \frac{1}{\sqrt{n}} \sum_{i=1}^n \tilde
{Z}_n(X_{ni}) + o_p(1)
\ = \ O_p(1),
\]
where
\[
\tilde{Z}_n(x) \ := \ k \Ex M(x,X_{n2},\ldots,X_{nk}) \ = \ k \Ex
(M(X_{n1},X_{n2},\ldots,X_{nk})|X_{n1}=x)
\]
satisfies $\Ex\tilde{Z}_n(X_{n1}) = 0$. In addition we define
\[
\tilde{Z}(x) \ := \ k \Ex M(x,X_2,\ldots,X_k).
\]
We may conclude from \eqref{ineq:CLT1}, continuity of $\rho'$ on
$(0,\infty)$ and of $\psi$ on $[0,\infty)$ in Case~1' and dominated
convergence that both functions $\tilde{Z}_n$ and $\tilde{Z}$ are
continuous on $\R^q$. Further there exists a constant $C$ such that
\[
\|\tilde{Z}_n(x)\|_F, \|\tilde{Z}(x)\|_F \ \le\ C + C \psi(\|x\|^2)
\]
for all $n \ge k$ and $x \in\mathbb{X}$. Thus it suffices show that
\begin{align*}
\Ex\biggl( \Bigl\| &
\frac{1}{\sqrt{n}} \sum_{i=1}^n \tilde{Z}_n(X_{ni})
- \frac{1}{\sqrt{n}} \sum_{i=1}^n
\bigl( \tilde{Z}(X_{ni}) - \Ex\tilde{Z}(X_{n1}) \bigr)
\Bigr\|_F^2 \biggr) \\
&\le\ \Ex\bigl( \bigl\|
\tilde{Z}_n(X_{n1}) - \tilde{Z}(X_{n1}) \bigr\|_F^2 \bigr)
\ \to\ 0.
\end{align*}

To this end we use a well-known result about weak convergence and
almost surely convergent representations (Skorohod, \citeyear
{Skorohod_1956}; Dudley, \citeyear{Dudley_1968}): There exists a
probability space $(\Omega_o,\mathcal{A}_o,\Pr_o)$ with random
variables $Y \sim P$ and $Y_n \sim P_n$ for $n \ge k$ such that $Y_n
\to Y$ almost surely. Now we define $(\Omega,\mathcal{A},\Pr) :=
(\Omega_o^k, \mathcal{A}_o^{\otimes k}, \Pr_o^{\otimes k})$ and
$X_i(\omega) = Y(\omega_i)$, $X_{ni}(\omega) := Y_n(\omega_i)$ for
$1 \le i \le k$, $n \ge k$ and $\omega= (\omega_i)_{i=1}^k \in\Omega
$. This construction implies that $(X_{ni})_{i=1}^k \to(X_i)_{i=1}^k$
almost surely. With $\mathcal{A}_*$ denoting the $\sigma$-field
generated by $X_1$ and $(X_{n1})_{n \ge k}$ we may write
\[
\tilde{Z}_n(X_{n1}) - \tilde{Z}(X_{n1})
\ = \ \Ex(\tilde{V}_n \,|\, \mathcal{A}_*)
\]
with
\[
\tilde{V}_n \ := \ M(X_{n1},X_{n2},\ldots,X_{nk}) -
M(X_{n1},X_2,\ldots,X_k),
\]
and
\[
\Ex\bigl( \bigl\| \tilde{Z}_n(X_{n1}) - \tilde{Z}(X_{n1}) \bigr\|
_F^2 \bigr)
\ = \ \Ex\bigl( \bigl\|
\Ex(\tilde{V}_n \,|\, \mathcal{A}_*) \bigr\|_F^2 \bigr)
\ \le\ \Ex\bigl( \|\tilde{V}_n\|_F^2 \bigr).
\]
But $\tilde{V}_n \to0$ almost surely, and
\[
\|\tilde{V}_n\|_F^2 \ \le\ B_n := C' \sum_{i=1}^k
\bigl( \psi(\|X_{ni}\|^2)^2 + \psi(\|X_i\|^2)^2 \bigr)
\]
for a suitable constant $C'$. Furthermore, $B_n \to B := 2 C' \sum
_{i=1}^k \psi(\|X_i\|^2)^2$ almost surely, and $\Ex(B_n) \to\Ex(B)
< \infty$. Hence for any fixed $R > 0$,
\[
\Ex\bigl( \|\tilde{V}_n\|_F^2 \bigr)
\ \le\ \Ex\bigl( \min\{\|\tilde{V}_n\|_F^2,R\} \bigr)
+ \Ex((B_n - R)^+)
\ \to\ \Ex((B - R)^+),
\]
and the right hand side tends to $0$ as $R \to\infty$.
\end{proof}

For the proof Remark~\ref{rem:CLT1.R1} we need an elementary fact
about symmetric matrices:

\begin{Proposition}
\label{prop:half-invariant M}
Let $M \in\Rqqsym$ and $x \in\R^q$ such that
\[
BMB^\top\ = \ M \quad
\text{for any orthogonal} \ B \in\Rqq\ \text{with} \ Bx = x.
\]
Then there exist real numbers $\gamma, \beta$ such that
\[
M \ = \ \gamma\, xx^\top+ \beta\, I_q.
\]
\end{Proposition}

\begin{proof}[\bf Proof of Proposition~\ref{prop:half-invariant M}]
Let $u \in x^\perp$ with $\|u\| = 1$. Then $B := I_q - 2 uu^\top$
defines an orthogonal matrix such that $B^\top= B$, $Bx = x$ and $Bu =
-u$. Consequently,
\[
u^\top Mx \ = \ u^\top BMB^\top x
\ = \ (Bu)^\top M (Bx) \ = \ - u^\top M x.
\]
Hence $Mx \perp x^\perp$ which is equivalent to $Mx = \lambda x$ for
some $\lambda\in\R$. In particular, $M(x^\perp) \subset x^\perp$.

Next let $u$ and $v$ be unit vectors in $x^\perp$ such that $u^\top v
= 0$ and $Mu = \beta_u u$, $Mv = \beta_v v$ for real numbers $\beta
_u,\beta_v$. Then $B := I_q - uu^\top- vv^\top+ uv^\top+ vu^\top$
defines an orthogonal matrix $B$ such that $B^\top= B$, $Bx = x$, $Bu
= v$ and $Bv = u$. Consequently,
\[
\beta_u \ = \ u^\top M u \ = \ (Bu)^\top M (Bu) \ = \ v^\top M v \ = \
\beta_v.
\]
Consequently, there exists a real number $\beta$ such that $My = \beta
y$ for all $y \in x^\perp$.

All in all we obtain the representation $M = \gamma\, xx^\top+ \beta
\,I_q$, where $\gamma= \lambda\|x\|^{-2} - \beta$ in case of $x \ne0$.
\end{proof}

\begin{proof}[\bf Proof of Remark~\ref{rem:CLT1.R1}]
Spherical symmetry of $P$ implies that $Q$ is orthogonally invariant.
Hence Lemma~\ref{lem:OrthInvQ} applies to $H_\rho(Q)$, and it
suffices to show that $\tilde{Z}(x) = \gamma(\|x\|^2) xx^\top+ \beta
(\|x\|^2) I_q$ with certain real numbers $\gamma(\|x\|^2)$ and $\beta
(\|x\|^2)$. But this is a consequence of Proposition~\ref
{prop:half-invariant M}: For any orthogonal matrix $B \in\Rqq$,
\[
S(Bx,BX_2,\ldots,BX_k) \ = \ B S(x,X_2,\ldots,X_k) B^\top,
\]
so it follows from $\LL(BX_j) = \LL(X_j)$ for $2 \le j \le k$ that
\[
\tilde{Z}(Bx)
= \Ex M(Bx,BX_2,\ldots,BX_k)
= B \Ex M(x,X_2,\ldots,X_k) B^\top
= B \tilde{Z}(x) B^\top.
\]
Restricting our attention temporarily to matrices $B$ such that $Bx =
x$ reveals that
\[
\tilde{Z}(x) \ = \ \tilde{\gamma}(x) xx^\top+ \tilde{\beta}(x) I_q
\]
with certain numbers $\tilde{\gamma}(x)$ and $\tilde{\beta}(x)$.
But for arbitrary orthogonal $B \in\Rqq$,
\[
\tilde{Z}(Bx) \ = \
\begin{cases}
\tilde{\gamma}(Bx) (Bx)(Bx)^\top+ \tilde{\beta}(Bx) I_q
\ = \ B \bigl( \tilde{\gamma}(Bx) xx^\top+ \tilde{\beta}(Bx) I_q
\bigr) B^\top, \\
B \tilde{Z}(x) B^\top
\ = \ B \bigl( \tilde{\gamma}(x) xx^\top+ \tilde{\beta}(x) I_q
\bigr) B^\top,
\end{cases}
\]
whence
\[
\tilde{\gamma}(Bx) xx^\top+ \tilde{\beta}(Bx) I_q
\ = \ \tilde{\gamma}(x) xx^\top+ \tilde{\beta}(x) I_q.
\]
Multiplying the latter equation with $y^\top$ from the left and with
$y$ from the right, where $0 \ne y \in x^\perp$, reveals that $\tilde
{\beta}(Bx) = \tilde{\beta}(x)$, i.e.\ $\tilde{\beta}(x) = \beta
(\|x\|^2)$. Then multiplication with $x^\top$ from the left and $x$
from the right reveals that $\tilde{\gamma}(Bx) = \tilde{\gamma
}(x)$, i.e.\ $\tilde{\gamma}(x) = \gamma(\|x\|^2)$.
\end{proof}

\begin{proof}[\bf Proof of Remark~\ref{rem:CLT1.R2}]
If $P$ is spherically symmetric around $0$, we may represent a random
vector $X \sim P$ as $X = RU$ with independent random variables $R \ge
0$ and $U \in\R^q$, where $U$ is uniformly distributed on the unit
sphere of $\R^q$.

In case of $\nu= 0$ (Case~0) we know already that
\[
H_\rho(Q) A \ = \ \frac{q}{q+2} \, A_0
\]
for any $A = A_0 + a I_q \in\Rqqsym$, where $a = q^{-1} \tr(A)$ and
$\tr(A_0) = 0$. Hence
\begin{align*}
Z(x) \
&= \ H_\rho(Q)^{-1} \bigl( q \|x\|^{-2} xx^\top- I_q \bigr)
\ = \ \frac{q}{\|x\|^2} H_\rho(Q)^{-1} A_0(x)
\ = \ \frac{q+2}{\|x\|^2} A_0(x) \\
&= \ (\nu+ \|x\|^2)^{-1} \bigl( c_0 A_0(x) + c_1 a(x) I_q \bigr)
\end{align*}
with $\nu= 0$ and $c_0, c_1$ as stated. Note that $c_1 = 0$ when $\nu
= 0$.

In case of $\nu> 0$ (Case~1'), Proposition~\ref{prop:M}, applied with
$\lambda= (1,0,\ldots,0)^\top$, entails that $A = A_0 + a I_q$ as
above is mapped to
\begin{align*}
H_\rho(Q)A \
&= \ A - \Ex\Bigl( \frac{(\nu+ q)R^4}{(\nu+ R^2)^2} \, U^\top A U
\, UU^\top\Bigr) \\
&= \ A - \Ex\Bigl( \frac{(\nu+ q) R^4}{(\nu+ R^2)^2} \Bigr)
\Bigl( \frac{2}{q(q+2)} \, A_0 + \frac{1}{q} \, a I_q \Bigr) \\
&= \ \Bigl( 1 - \frac{2(q - \nu+ \beta\nu)}{q(q+2)} \Bigr) \, A_0
+ \Bigl( 1 - \frac{q - \nu+ \beta\nu}{q} \Bigr) a I_q \\
&= \ \frac{q + 2\nu(1 - \beta)/q}{q+2} \, A_0
+ \frac{\nu(1 - \beta)}{q} \, a I_q,
\end{align*}
because
\[
\Ex\Bigl( \frac{(\nu+ q) R^4}{(\nu+ R^2)^2} \Bigr)
\ = \ \Ex\Bigl(
\frac{(\nu+ q)(R^2 - \nu)}{\nu+ R^2}
+ \frac{(\nu+ q)\nu^2}{(\nu+ R^2)^2} \Bigr)
\ = \ q - (1 - \beta)\nu
\]
by the definition of $\beta$ and since $\Sigma_\rho(Q)=I_q$. Note
that the latter implies the equations $\Ex(R^2/(\nu+R^2)) = q/(\nu
+q)$ and $\Ex(1/(\nu+R^2)) = 1/(\nu+q)$.
Consequently,
\[
H_\rho(Q)^{-1}A
\ = \ \frac{q+2}{q + 2(1 - \beta)\nu/q} \, A_0
+ \frac{q}{(1 - \beta)\nu} \, a I_q.
\]
This yields the representation
\begin{align*}
Z(x) \
&= \ H_\rho(Q)^{-1} \bigl( \rho'(\|x\|^2) xx^\top- I_q \bigr) \\
&= \ (\nu+ \|x\|^2)^{-1} H_\rho(Q)^{-1} \bigl( (\nu+ q) (A_0(x) +
a(x) I_q + I_q)
- (\nu+ \|x\|^2) I_q \bigr) \\
&= \ (\nu+ \|x\|^2)^{-1} H_\rho(Q)^{-1} \bigl( (\nu+ q) A_0(x) +
\nu a(x) I_q \bigr) \\
&= \ (\nu+ \|x\|^2)^{-1}
\Bigl( \frac{(\nu+ q)(q + 2)}{q + 2(1 - \beta)\nu/q} \, A_0(x)
+ \frac{q}{1 - \beta} \, a(x) I_q \Bigr) \\
&= \ (\nu+ \|x\|^2)^{-1}
\bigl( c_0 A_0(x) + c_1 a(x) I_q \bigr)
\end{align*}
with $c_0$ and $c_1$ as stated.
\end{proof}


\subsection{Proofs for Section~\ref{sec:Location and Scatter}}

\begin{proof}[\bf Proof of Theorem~\ref{thm:MTLS}]
Note that $\tilde{L}(\cdot,\tilde{P})$ is equal to the scatter-only
functional $L(\cdot,Q)$ with $Q = Q^1(\tilde{P}) = \LL\bigl(
y(X)y(X)^\top\bigr)$, $X \sim P$. In what follows let $\mathbb{H}_0
:= \bigl\{ (x^\top, 0)^\top: x \in\R^q \bigr\}$ and $\mathbb
{H}_1 := \bigl\{ (x^\top, 1)^\top: x \in\R^q \bigr\} = \{y(x) : x
\in\R^q\}$. For any linear subspace $\mathbb{W}$ of $\R^{q+1}$ with
$1 \le\dim(\mathbb{W}) \le q$, elementary linear algebra reveals
that either $\mathbb{W} \subset\mathbb{H}_0$ or
\[
\mathbb{W} \cap\mathbb{H}_1 \ = \ \bigl\{ y(a+v) : v \in\mathbb
{V} \bigr\}
\]
for some $a \in\R^q$ and a linear subspace $\mathbb{V}$ of $\R^q$
with $\dim(\mathbb{V}) = \dim(\mathbb{W}) - 1$.

In case of $\nu= 1$, we know from Theorem~\ref{thm:Uniqueness} that
$\tilde{L}(\cdot,\tilde{P}) = L(\cdot,Q)$ possesses a unique
minimizer up to multiplication with positive scalars if, and only if,
\[
Q(\mathbb{M}(\mathbb{W}))
\ = \ \tilde{P}(\mathbb{W})
\ < \ \frac{\dim(\mathbb{W})}{q+1}
\]
for arbitrary linear subspaces $\mathbb{W}$ of $\R^{q+1}$ with $1 \le
\dim(\mathbb{W}) \le q$. In view of the previous considerations, and
since $\tilde{P}(\mathbb{H}_0) = 0$, this is equivalent to
\[
P(a + \mathbb{V})
\ < \ \frac{\dim(\mathbb{V})+1}{q+1}
\]
for arbitrary $a \in\R^q$ and any linear subspace $\mathbb{V}$ of
$\R^q$ with $0 \le\dim(\mathbb{V}) < q$.

In case of $\nu> 1$, we apply Theorem~\ref{thm:Uniqueness} to $\rho
(s) = \rho_{\nu-1,q+1}(s) = (\nu+ q) \log(\nu+ s -1)$, i.e.\ $\psi
(\infty) = \nu+ q$. Hence $\tilde{L}(\cdot,\tilde{P}) = L(\cdot
,Q)$ possesses a unique minimizer $\Gamma\in\R^{(q+1)\times
(q+1)}_{{\rm sym},>0}$ if, and only if
\[
Q(\mathbb{M}(\mathbb{W}))
\ = \ \tilde{P}(\mathbb{W})
\ < \ \frac{\dim(\mathbb{W}) + \nu- 1}{q + \nu}
\]
for arbitrary linear subspaces $\mathbb{W}$ of $\R^{q+1}$ with $0 \le
\dim(\mathbb{W}) \le q$. Since \mbox{$\tilde{P}(\{0\}) = 0$}, it suffices
to consider the case $\dim(\mathbb{W}) \ge1$, and then the previous
considerations show that our requirement on $\tilde{P}$ is equivalent to
\[
P(a + \mathbb{V})
\ < \ \frac{\dim(\mathbb{V})+\nu}{q+\nu}
\]
for arbitrary $a \in\R^q$ and any linear subspace $\mathbb{V}$ of
$\R^q$ with $0 \le\dim(\mathbb{V}) < q$.

It remains to show that for $\nu> 1$, a minimizer $\Gamma$ of $\tilde
{L}(\cdot,\tilde{P})$ satisfies\break  \mbox{$\Gamma_{q+1,q+1} = 1$}. To this end,
recall that $\Gamma$ satisfies the fixed-point equation
\begin{equation}
\label{eq:Gamma}
\Gamma
\ = \ \Psi(Q) = \Ex\Bigl( \frac{q + \nu}{Y^\top\Gamma^{-1}Y +
\nu-1} \, YY^\top\Bigr)
\end{equation}
with $Y := y(X)$, $X \sim P$. In particular, since $Y_{q+1} = 1$ almost surely,
\[
\Gamma_{q+1,q+1} \ = \ \Ex\frac{q + \nu}{Y^\top\Gamma^{-1}Y + \nu
-1}.
\]
But \eqref{eq:Gamma} implies also that
\begin{align*}
q+1 \ = \ \tr(\Gamma^{-1} \Psi(Q)) \
&= \ \Ex\frac{(q+\nu) Y^\top\Gamma^{-1}Y}{Y^\top\Gamma^{-1}Y +
\nu-1} \\
&= \ q+\nu- \Ex\frac{(q + \nu)(\nu-1)}{Y^\top\Gamma^{-1}Y + \nu
-1} \\
&= \ q+\nu- (\nu- 1) \Gamma_{q+1,q+1} \\
&= \ q+1 + (\nu- 1)(1 - \Gamma_{q+1,q+1}),
\end{align*}
i.e.\ $\Gamma_{q+1,q+1} = 1$.
\end{proof}

\begin{proof}[\bf Proof of Theorem~\ref{thm:WeakDiffability.t}]
It follows from Theorem~\ref{thm:Diffability} that with asymptotic
probability one there exists a unique minimizer $\bs{\Gamma}(\hat
{P}_n)$ of
\[
\int\bigl[ \rho(y(x)^\top\Gamma^{-1} y(x)) - \rho(\|y(x)\|^2)
\bigr]
\, \hat{P}_n(dx) + \log\det(\Gamma)
\]
over all $\Gamma\in\R^{(q+1)\times(q+1)}_{{\rm sym},>0}$. In case
of $\nu= 1$ we also require that $\det(\Gamma) = 1$. Moreover,
\[
\bs{\Gamma}(\hat{P}_n) \ = \ I_{q+1} - \tilde{H}(P)^{-1} \tilde
{G}(\hat{P}_n)
+ o_p \bigl( \|\tilde{G}(\hat{P}_n)\| \bigr)
\]
with the operator $\tilde{H}(P)$ as stated, and
\[
\tilde{G}(\hat{P}_n)
\ := \ I_{q+1} - \int\tilde{\rho}'(\|y(x)\|^2) y(x)y(x)^\top\, \hat
{P}_n(dx)
\ \to_p \ 0.
\]
Now we set
\[
\tilde{Z}(x) \ := \ \tilde{H}(P)^{-1}
\bigl( \rho'(\|x\|^2) y(x) y(x)^\top- I_{q+1} \bigr)
\]
for $x \in\R^q$. This defines a bounded, continuous function $\tilde
{Z} : \R^q \to\R_{{\rm sym}}^{(q+1)\times(q+1)}$ with $\int\tilde
{Z} \, dP = 0$. Since the operator $\tilde{H}(P)$ is non-singular,
both $\|\tilde{G}(\hat{P}_n)\|$ and $\delta_n := \bigl\| \int
\tilde{Z} \, d\hat{P}_n \bigr\|$ tend to zero in probability at the
same speed, and we may write
\[
\bs{\Gamma}(\hat{P}_n) \ = \ I_{q+1} + \int\tilde{Z} \, d\hat{P}_n
+ o_p(\delta_n).
\]
But then
\begin{align*}
&
\begin{bmatrix}
\bSigma(\hat{P}_n) + \bmu(\hat{P}_n)\bmu(\hat{P}_n)^\top
& \bmu(\hat{P}_n) \\
\bmu(\hat{P}_n)^\top
& 1
\end{bmatrix}
\\
&\quad= \ (\bs{\Gamma}(\hat{P}_n)_{q+1,q+1})^{-1} \bs{\Gamma
}(\hat{P}_n) \\
&\quad= \ \Bigl( 1 - \int\tilde{Z}_{q+1,q+1} \, d\hat{P}_n +
o_p(\delta_n) \Bigr)
\Bigl( I_{q+1} + \int\tilde{Z} \, d\hat{P}_n + o_p(\delta_n) \Bigr
) \\
&\quad= \ I_{q+1}
+ \int\bigl( \tilde{Z} - \tilde{Z}_{q+1,q+1} I_{q+1} \bigr) \,
d\hat{P}_n
+ o_p(\delta_n).
\end{align*}
In particular, $\bmu(\hat{P}_n) = O_p(\delta_n)$, whence $\bmu(\hat
{P}_n) \bmu(\hat{P}_n)^\top= O_p(\delta_n^2) = o_p(\delta_n)$ and thus
\[
\begin{bmatrix}
\bSigma(\hat{P}_n) - I_q & \bmu(\hat{P}_n) \\
\bmu(\hat{P}_n)^\top& 0
\end{bmatrix}
\ = \ \int\bigl( \tilde{Z} - \tilde{Z}_{q+1,q+1} I_{q+1} \bigr) \,
d\hat{P}_n
+ o_p(\delta_n).
\]

It remains to show that $\tilde{Z}_{q+1,q+1}(x) = 0$ for any fixed $x
\in\R^q$ in case of $\nu> 1$. To this end we consider the nonrandom
distributions $P_n := (1 - n^{-1}) P + n^{-1} \delta_x$. For
sufficiently large $n$, $\bs{\Gamma}(P_n)$ is well-defined with $\bs
{\Gamma}(P_n)_{q+1,q+1} = 1$. On the other hand, $\int\tilde{Z} \,
dP_n = n^{-1} \tilde{Z}(x)$ and
\[
\bs{\Gamma}(P_n)
\ = \ I_{q+1} + n^{-1} \tilde{Z}(x) + o(n^{-1}),
\]
which implies that $\tilde{Z}_{q+1,q+1}(x) = 0$.
\end{proof}

\begin{proof}[\bf Proof of Remarks~\ref{rem:WeakDiffability.t.R1} and
\ref{rem:WeakDiffability.t.R2}]
Recall that $\tilde{G}(P) = 0$ is equivalent to
\begin{equation}
\label{eq:SomeMoments.1}
\int\frac{\nu+ q}{\nu+ \|x\|^2} \,
\begin{bmatrix} xx^\top& x \\ x^\top& 1
\end{bmatrix}
\, P(dx)
\ = \
\begin{bmatrix} I_q & 0 \\ 0 & 1
\end{bmatrix}
,
\end{equation}
in particular,
\begin{equation}
\label{eq:SomeMoments.2}
\int\frac{(\nu+ q)\|x\|^2}{\nu+ \|x\|^2} \, P(dx) \ = \ q
\quad\text{and}\quad
\int\frac{\nu+ q}{\nu+ \|x\|^2} \, P(dx) \ = \ 1.
\end{equation}
Now we introduce the auxiliary objects
\begin{align*}
\Psi_2 = \Psi_2(P) \
&:= \ \int\frac{\nu+ q}{(\nu+ \|x\|^2)^2} \, xx^\top\, P(dx) \ \in
\ \Rqqsympd, \\
\beta= \beta(P) \
&:= \ \int\frac{(\nu+ q)\nu}{(\nu+ \|x\|^2)^2} \, P(dx) \ > \ 0,
\end{align*}
i.e.
\[
\beta+ \tr(\Psi_2) \ = \ \int\frac{\nu+q }{\nu+ \|x\|^2} P(dx) \
= \ 1,
\]
and the operator $H = H(P) : \Rqqsym\to\Rqqsym$ given by
\[
HA \ := \ A - \int\frac{\nu+ q}{(\nu+ \|x\|^2)^2} \, x^\top Ax \,
xx^\top\, P(dx).
\]
Then for a matrix
\[
M \ = \
\begin{bmatrix}
A & b \\ b^\top& c
\end{bmatrix}
\]
with $A \in\Rqqsym$, $b \in\R^q$ and $c \in\R$, we may write
\begin{align*}
\tilde{H}(P) M \,
= \, &
\begin{bmatrix}
A & \!b \\ b^\top& \!c
\end{bmatrix}
- \int\frac{(\nu+ q)(x^\top Ax + 2x^\top b + c)}{(\nu+ \|x\|^2)^2}
\begin{bmatrix} xx^\top& \!x \\ x^\top& \!1
\end{bmatrix}
\, P(dx) \\
= \, &
\begin{bmatrix}
HA - c \Psi_2 & \!0 \\
0 & \!(1 - \beta/\nu)c - \langle\Psi_2, A\rangle
\end{bmatrix}
+
\begin{bmatrix}
0 & \!(I_q - 2\Psi_2) b \\
b^\top(I_q - 2 \Psi_2) & \!0
\end{bmatrix}
\end{align*}
Here we utilized the fact that any term of the form $f(xx^\top) x$
integrates to $0$, due to the symmetry of $P$. Consequently,
\begin{align*}
\biggl\{ \tilde{H}(P)
\begin{bmatrix} A & 0 \\ 0 & c
\end{bmatrix}
:
A \in\Rqqsym, c \in\R\biggr\} \
&\subset\ \biggl\{
\begin{bmatrix} A & 0 \\ 0 & c
\end{bmatrix}
:
A \in\Rqqsym, c \in\R\biggr\} \\
\biggl\{ \tilde{H}(P)
\begin{bmatrix} 0 & b \\ b^\top& 0
\end{bmatrix}
:
b \in\R^q \biggr\} \
&= \ \biggl\{
\begin{bmatrix} 0 & b \\ b^\top& 0
\end{bmatrix}
:
b \in\R^q \biggr\},
\end{align*}
where the latter equality follows from $\tilde{H}(P)$ being
nonsingular on $\tilde{\mathbb{M}}$. In particular, $B = B(P) := (I_q
- 2 \Psi_2(P))^{-1} \in\Rqqsym$ exists, and
\begin{align*}
\tilde{Z}(x) \
&= \ \tilde{H}(P)^{-1} \bigl( \rho'(\|x\|^2) y(x)y(x)^{-1} - I_{q+1}
\bigr) \\
&= \ \tilde{H}(P)^{-1}
\begin{bmatrix}
\rho'(\|x\|^2) xx^\top- I_q & \!\!0 \\
0 & \!\!\rho'(\|x\|^2) - 1
\end{bmatrix}
+ \rho'(\|x\|^2)
\tilde{H}(P)^{-1}
\begin{bmatrix} 0 & \!\!x \\ x^\top& \!\!0
\end{bmatrix}
\\
&= \
\begin{bmatrix} Z(xx^\top) & 0 \\ 0 & z(\|x\|^2)
\end{bmatrix}
+ \rho'(\|x\|^2)
\begin{bmatrix} 0 & Bx \\ x^\top B & 0
\end{bmatrix}
\end{align*}
with certain bounded, continuous functions $Z : \Rqqsympsd\to\Rqqsym
$ and $z : [0,\infty) \to\R$.

This proves Remark~\ref{rem:WeakDiffability.t.R1}. In the special case
of $P$ being spherically symmetric around $0$, a random vector $X \sim
P$ may be written as $X = RU$ with independent random variables $R \ge
0$ and $U \in\R^q$, where $U$ is uniformly distributed on the unit
sphere of $\R^q$. Then \eqref{eq:SomeMoments.2} and the definition of
$\beta$ translate to
\[
\Ex\Bigl( \frac{(\nu+ q) R^2}{\nu+ R^2} \Bigr) \ = \ q,
\quad
\Ex\Bigl( \frac{\nu+ q}{\nu+ R^2} \Bigr) \ = \ 1
\quad\text{and}\quad
\beta\ = \ \Ex\Bigl( \frac{(\nu+ q)\nu}{(\nu+ R^2)^2} \Bigr).
\]
Further, it follows from $\Ex(UU^\top) = q^{-1} I_q$ that
\[
\Psi_2
\ = \ \Ex\Bigl( \frac{(\nu+ q) R^2}{(\nu+ R^2)^2} \, UU^\top\Bigr)
\ = \ \Ex\Bigl( \frac{\nu+ q}{\nu+ R^2} - \frac{(\nu+ q)\nu
}{(\nu+ R^2)^2} \Bigr)
\, \frac{1}{q} \, I_q
\ = \ \gamma_1 \, I_q
\]
with
\[
\gamma_1 \ := \ \frac{1 - \beta}{q}.
\]
Now we write $A \in\Rqqsym$ as $A = A_0 + a I_q$ with $a := \tr
(A)/q$, so $\tr(A_0) = 0$. Then
\[
\langle\Psi_2, A\rangle
\ = \ \frac{1 - \beta}{q} \tr(A)
\ = \ \gamma_1 q a
\]
and, as shown in the proof of Remark~\ref{rem:CLT1.R2},
\[
H(P)A
\ = \ \gamma_0 \, A_0 + \gamma_1\nu\, a I_q,
\]
with
\[
\gamma_0 \ := \ \frac{q + 2\gamma_1\nu}{q+2}.
\]
Hence for $A_0 \in\Rqqsym$ with $\tr(A_0) = 0$, $a \in\R$, $b \in
\R^q$ and $c \in\R$,
\[
\tilde{H}(P)
\begin{bmatrix} A_0 + aI_q & b \\ b^\top& c
\end{bmatrix}
\ = \
\begin{bmatrix}
\gamma_0 A_0 + \gamma_1 (\nu a - c) I_q & (1 - 2\gamma_1) b \\
(1 - 2\gamma_1) b^\top& (1 - \beta/\nu) c - \gamma_1 q a
\end{bmatrix}
.
\]

In case of $\nu= 1$ we only consider the case $\tr(M) = 0$, i.e.\ $c
= - qa$. Then
\[
\tilde{H}(P)
\begin{bmatrix} A_0 + aI_q & b \\ b^\top& -q a
\end{bmatrix}
\ = \
\begin{bmatrix}
\gamma_0 A_0 + \gamma_1 (q+1) a I_q & (1 - 2\gamma_1) b \\
(1 - 2\gamma_1) b^\top& - q \gamma_1 (q+1) a
\end{bmatrix}
,
\]
and this shows that
\[
\tilde{H}(P)^{-1}
\begin{bmatrix} A_0 + aI_q & b \\ b^\top& -q a
\end{bmatrix}
\ = \
\begin{bmatrix}
\gamma_0^{-1} A_0 + \gamma_1^{-1} (q+1)^{-1} a I_q & (1 - 2\gamma
_1)^{-1} b \\
(1 - 2\gamma_1)^{-1} b^\top& - q \gamma_1^{-1} (q+1)^{-1} a
\end{bmatrix}
.
\]
Now we consider $x \in\R^q$ and write $xx^\top= A_0(x) + a(x) I_q +
I_q$ with $a(x) := q^{-1} \|x\|^2 - 1$, so $\tr(A_0(x)) = 0$. Then
\begin{align*}
\tilde{Z}(x) \
&= \ (1 + \|x\|^2)^{-1} \tilde{H}(P)^{-1}
\bigl( (1 + q) y(x)y(x)^\top- (1 + \|x\|^2) I_{q+1} \bigr) \\
&= \ (1 + \|x\|^2)^{-1} \tilde{H}(P)^{-1}
\begin{bmatrix}
(1 + q) A_0(x) + a(x) I_q & (1 + q) x \\
(1 + q) x^\top& - q a(x)
\end{bmatrix}
\\
&= \ (1 + \|x\|^2)^{-1}
\begin{bmatrix}
\displaystyle
\frac{1 + q}{\gamma_0} A_0(x) + \frac{1}{\gamma_1 (1+q)} a(x) I_q
& (1 - 2\gamma_1)^{-1} x \\
(1 - 2\gamma_1)^{-1} x^\top
& \displaystyle
\frac{-q}{\gamma_1 (1+q)} a(x)
\end{bmatrix}
.
\end{align*}
Consequently,
\[
\tilde{Z}(x) - \tilde{Z}(x)_{q+1,q+1} I_{q+1}
\ = \ (1 + \|x\|^2)^{-1}
\begin{bmatrix}
c_0 A_0(x) + c_1 a(x) I_q & c_2 x \\
c_2 x^\top& 0
\end{bmatrix}
\]
with
\begin{align*}
c_0 \ &:= \ \frac{1 + q}{\gamma_0} \ = \ \frac{(q + 1)(q + 2)}{q +
2(1 - \beta)/q}, \\
c_1 \ &:= \ \frac{1}{\gamma_1} \ = \ \frac{q}{1 - \beta}, \\
c_2 \ &:= \ (1 - 2 \gamma_1)^{-1} \ = \ \frac{q}{q - 2(1 - \beta)}.
\end{align*}

In case of $\nu> 1$, elementary calculations reveal that the inverse
of the mapping
\[
\begin{bmatrix} a \\ c
\end{bmatrix}
\ \mapsto\
\begin{bmatrix}
\gamma_1 (\nu a - c) \\
(1 - \beta/\nu) c - \gamma_1 q a
\end{bmatrix}
=
\begin{bmatrix}
\gamma_1 \nu& - \gamma_1 \\
- \gamma_1 q & 1 - \beta/\nu
\end{bmatrix}
\begin{bmatrix} a \\ c
\end{bmatrix}
\]
is given by
\[
\begin{bmatrix} a \\ c
\end{bmatrix}
\ \mapsto\ \frac{1}{\nu- 1}
\begin{bmatrix}
(1 - \beta/\nu)/\gamma_1 & 1 \\
q & 1 \nu
\end{bmatrix}
\begin{bmatrix} a \\ c
\end{bmatrix}
.
\]
Consequently
\[
\tilde{H}(P)^{-1}
\begin{bmatrix} A_0 + aI_q & b \\ b^\top& c
\end{bmatrix}
\ = \
\begin{bmatrix}
\displaystyle
\gamma_0^{-1} A_0 + \frac{(1 - \beta/\nu) \gamma_1^{-1} a + c}{\nu
- 1} \, I_q
& (1 - 2\gamma_1)^{-1} b \\
(1 - 2\gamma_1)^{-1} b^\top
& \displaystyle
\frac{qa + \nu c}{\nu- 1}
\end{bmatrix}
.
\]
Hence
\begin{align*}
\tilde{Z}(x) \
&= \ (\nu+ \|x\|^2)^{-1} \tilde{H}(P)^{-1}
\bigl( (\nu+ q) y(x)y(x)^\top- (\nu+ \|x\|^2) I_{q+1} \bigr) \\
&= \ (\nu+ \|x\|^2)^{-1} \tilde{H}(P)^{-1}
\begin{bmatrix}
(\nu+ q) A_0(x) + \nu a(x) I_q & (\nu+ q) x \\
(\nu+ q) x^\top& - q a(x)
\end{bmatrix}
\\
&= \ (1 + \|x\|^2)^{-1}
\begin{bmatrix}
\displaystyle
\frac{\nu+ q}{\gamma_0} A_0(x) + \frac{q}{1 - \beta} a(x) I_q
& (1 - 2\gamma_1)^{-1} x \\[1.5ex]
(1 - 2\gamma_1)^{-1} x^\top
& 0
\end{bmatrix}
\\
&= \ (1 + \|x\|^2)^{-1}
\begin{bmatrix}
c_0 A_0(x) + c_1 a(x) I_q & c_2 x \\
c_2 x^\top& 0
\end{bmatrix}
\end{align*}
with $c_0, c_1, c_2$ as stated.
\end{proof}


\section*{Acknowledgement}
Constructive comments by an associate editor and two referees are
gratefully acknowledged. Many thanks to David Tyler for stimulating
discussions, in particular for encouraging us to drop the assumption of
$\rho'$ being non-increasing.


\section*{List of notation and assumptions}

\paragraph{Linear and affine transformations}
Let $P$ and $Q$ be probability distributions on $\R^q$ and $\Rqqsympsd
$, respectively. For $a \in\R^q$, $B \in\Rqqns$ and $X \sim P$, $S
\sim Q$,
\[
P^B \ := \ \LL(BX), \quad
P^{a,B} \ := \ \LL(a + BX),
\]
and
\[
Q^B \ := \ \LL(BSB^\top), \quad
Q_B \ := \ \LL(B^{-1}SB^{-\top}).
\]

\paragraph{Special (empirical) distributions}
Let $X = X_1, X_2, X_3, \ldots$ be i.i.d.\ $\sim P$. Then for $k \ge2$,
\[
Q^1(P) \ := \ \LL(XX^\top)
\quad\text{and}\quad
Q^k(P) \ := \ \LL\bigl( S(X_1,X_2,\ldots,X_k) \bigr)
\]
with $S(x_1,x_2,\ldots,x_k)$ denoting the sample covariance matrix of
$x_1, x_2, \ldots,\break  x_k \in\R^q$. Furthermore,
\[
\hat{P} \
:= \ n_{}^{-1} \sum_{i=1}^n \delta_{X_i}^{}, \quad
\hat{Q}^1 \
:= \ n_{}^{-1} \sum_{i=1}^n \delta_{X_i^{}X_i^\top}^{}
\]
and
\[
\hat{Q}^k \
:= \ \binom{n}{k}^{-1} \sum_{1 \le i_1 < i_2 < \cdots< i_k \le n}
\delta_{S(X_{i_1},X_{i_2},\ldots,X_{i_k})}^{}.
\]

\paragraph{Log-likelihood functions (times $-2$) and derivatives}
\begin{align*}
L(\mu,\Sigma,P) \
&:= \ \int\bigl[ \rho\bigl( (x - \mu)^\top\Sigma^{-1} (x - \mu)
\bigr)
- \rho(x^\top x) \bigr] \, P(dx) + \log\det(\Sigma), \\
L_\rho(\Sigma,Q) \
&:= \ \int\bigl[ \rho(\tr(\Sigma^{-1}M)) - \rho(\tr(M)) \bigr]
\, Q(dM)
+ \log\det\Sigma.
\end{align*}
Under certain conditions, as $\Rqqsym\ni A \to0$,
\begin{align*}
L_\rho(\exp(A),Q) \
&= \ \langle G_\rho(Q),A\rangle+ o(\|A\|) \\
&= \ \langle G_\rho(Q),A\rangle+ 2^{-1} H_\rho(A,Q) + o(\|A\|^2),
\end{align*}
where
\begin{align*}
G_\rho(Q) \
&:= \ I_q - \Psi_\rho(Q),
\quad\Psi_\rho(Q) \ := \ \Psi_\rho(I_q,Q), \\
\Psi_\rho(\Sigma,Q) \
&:= \ \int\rho'(\tr(\Sigma^{-1} M)) M \, Q(dM), \\
H_\rho(A,Q) \
&:= \ \int\bigl( \rho'(\tr(M)) \tr(A^2M) + \rho''(\tr(M)) \tr
(AM)^2 \bigr) \, Q(dM).
\end{align*}
Moreover, $H_\rho(A,Q) = \langle H_\rho(Q) A, A \rangle$ with the
linear operator $H_\rho(Q) : \Rqqsym\to\Rqqsym$ given by
\[
H_\rho(Q) A
\ := \ 2_{}^{-1} \bigl( A \Psi_\rho(Q) + \Psi_\rho(Q) A \bigr)
+ \int\rho''(\tr(M)) \tr(AM) M \, Q(dM).
\]
Sometimes we write $\Rqqsym= \W_0 \oplus\W_1$ with
\[
\W_0 \ := \ \{A \in\Rqqsym: \tr(A) = 0\}
\quad\text{and}\quad
\W_1 \ := \ \{s I_q : s \in\R\}.
\]
In Case~0, we view $H_\rho(Q)$ as an endomorphism of $\W_0$.

\paragraph{Assumptions on $\rho$ and $Q$}
We assume that $\rho$ is continuously differentiable on $(0,\infty)$
with derivative $\rho' > 0$. For $s > 0$ we define
\[
\psi(s) \ := \ s \rho'(s).
\]

\paragraph{Case 0} $\rho(s) = q \log(s)$ for $s > 0$, and $Q(\{0\})
= 0$.

\paragraph{Case 1} $\psi$ is strictly increasing on $(0,\infty)$
with limits $\psi(0) = 0$ and $\psi(\infty) \in(q,\infty]$.
Moreover, $J_\rho(\lambda,Q) := \int\psi(\lambda\tr(M)) \, Q(dM)
< \infty$ for any $\lambda\ge1$.

\paragraph{Case 1'} $\rho$ is twice continuously differentiable on
$(0,\infty)$ with $\psi' > 0$, and $\psi$ has limits $\psi(0) = 0$
and $\psi(\infty) \in(q,\infty]$. Moreover, $J_\rho(Q) := \int
\psi(\tr(M)) \, Q(dM) < \infty$, and there exists a constant $\kappa
\ge0$ such that $s \psi'(s) \le\kappa\psi(s)$ for all $s > 0$.

\paragraph{Existence of $\bSigma_\rho(Q)$}
Let $\QQ_\rho$ be the set of all distributions $Q$ such that $L_\rho
(\cdot,Q)$ is real-valued and has a unique minimizer $\bSigma_\rho
(Q) \in\Rqqsympd$, where $\det(\bSigma_\rho(Q)) = 1$ in Case~0. To
characterize $\QQ_\rho$ let
\begin{align*}
\VV_q \
&:= \ \{\V: \V\ \text{a linear subspace of} \ \R^q\}, \\
\M(\V) \
&:= \ \bigl\{ M \in\Rqqsym: M \R^q \subset\V\bigr\}
\quad\text{for} \ \V\in\VV_q.
\end{align*}
Necessary and sufficient condition for $Q \in\QQ_\rho$:

\paragraph{Condition~0 (for Case~0)} For any $\V\in\VV_q$ with $1
\le\dim(\V) < q$,
\[
Q(\M(\V)) \ < \ \frac{\dim(\V)}{q}.
\]

\paragraph{Condition~1 (for Case~1)} For any $\V\in\VV_q$ with $0
\le\dim(\V) < q$,
\[
Q(\M(\V)) \ < \ \frac{\psi(\infty) - q + \dim(\V)}{\psi(\infty
)}.
\]

\bibliographystyle{imsart-nameyear}

\end{document}